\documentclass[10pt,a4paper]{article}

\usepackage[final]{optional}

\usepackage{amsfonts, amsmath, wasysym}

\usepackage{tikz}

\usepackage{amssymb,amsthm,
paralist
}

\usepackage{
latexsym,
}

\usepackage{url}

\definecolor{darkgreen}{rgb}{0,0.5,0}
\definecolor{darkred}{rgb}{0.7,0,0}
\usepackage[colorlinks, 
citecolor=darkgreen, linkcolor=darkred
]{hyperref}

\theoremstyle{plain}
\newtheorem{lemma}{Lemma}[section]
\newtheorem{thm}[lemma]{Theorem}
\newtheorem{prop}[lemma]{Proposition}
\newtheorem{cor}[lemma]{Corollary}

\theoremstyle{definition}

\newtheorem{rmk}[lemma]{Remark}

\numberwithin{equation}{section}

\newcommand{\pl}[2]{{\frac{\partial #1}{\partial #2}}}
\newcommand{\partt}{ {\frac{\partial}{\partial t} } }

\newcommand{\ti}{\tilde}

\newcommand{\al}{\alpha}
\newcommand{\be}{\beta}
\newcommand{\ga}{\gamma}

\newcommand{\de}{\delta}
\newcommand{\om}{\omega}
\newcommand{\Om}{\Omega}
\newcommand{\ka}{\kappa}
\newcommand{\la}{\lambda}

\newcommand{\si}{\sigma}

\newcommand{\ep}{\varepsilon}

\newcommand{\R}{\ensuremath{{\mathbb R}}}
\newcommand{\N}{\ensuremath{{\mathbb N}}}

\newcommand{\curlO}{{\mathcal O}}

\newcommand{\downto}{\downarrow}
\newcommand{\upto}{\uparrow}

\newcommand{\lap}{\Delta}

\newcommand{\grad}{\nabla}

\newcommand{\intersect}{\cap}

\DeclareMathOperator{\Vol}{Vol}
\DeclareMathOperator{\VolB}{VolB}

\DeclareMathOperator{\inj}{inj}

\def\blbox{\quad \vrule height7.5pt width4.17pt depth0pt}

\newcommand{\beq}{\begin{equation}}
\newcommand{\eeq}{\end{equation}}
\newcommand{\beqa}{\begin{equation}\begin{aligned}}
\newcommand{\eeqa}{\end{aligned}\end{equation}}
\newcommand{\brmk}{\begin{rmk}}
\newcommand{\ermk}{\end{rmk}}
\newcommand{\partref}[1]{\hbox{(\csname @roman\endcsname{\ref{#1}})}}
\newcommand{\half}{\frac{1}{2}}

\newcommand{\cmt}[1]{\opt{draft}{\textcolor[rgb]{0.5,0,0}{
$\LHD$ #1 $\RHD$\marginpar{\blbox}}}}

\newcommand{\bcmt}[1]{\opt{draft}{\textcolor[rgb]{0,0,0.9}{
$\LHD$ #1 $\RHD$\marginpar{\blbox}}}}

\newcommand{\Rm}{{\mathrm{Rm}}}
\newcommand{\Ric}{{\mathrm{Ric}}}
\newcommand{\Rc}{{\mathrm{Rc}}}

\newcommand{\Sc}{{\mathrm{R}}}

\newcommand{\f}{\ensuremath{{\cal F}}}

\usepackage{soul}

\title{{
\bf
Local control on the geometry\\ in 3D Ricci flow
} 
\\ 
\cmt{DRAFT with comments}
}
\author{Miles Simon and Peter M. Topping}
\date{\today}

\begin{document}

\maketitle
\parskip=10pt

\begin{abstract}
The geometry of a ball within a Riemannian manifold is coarsely controlled if it has a lower bound on its Ricci curvature and a positive lower bound on its volume. 
We prove that such coarse local geometric control 
must persist for a definite amount of time under
three-dimensional Ricci flow, and leads to local $C/t$ decay of the full curvature tensor, irrespective of what is happening beyond the local region.

As a by-product, our results generalise the Pseudolocality theorem of Perelman \cite[\S 10.1 and \S 10.5]{P1} and 
Tian-Wang \cite{TW} in this dimension by not requiring the Ricci curvature to be almost-positive, and not asking the volume growth to be almost-Euclidean. 

Our results also have applications to the topics of starting Ricci flow with manifolds of unbounded curvature, to the use of Ricci flow as a mollifier, and to the well-posedness of Ricci flow starting with Ricci limit spaces. In \cite{ST2} we use results from this paper to prove that 3D Ricci limit spaces are locally bi-H\"older equivalent to smooth manifolds, going beyond a full resolution of the conjecture of Anderson, Cheeger, Colding and Tian in this dimension.
\end{abstract}

\parskip=-9pt
\tableofcontents
\parskip=10pt

\section{Statement of the main result}
\label{intro}

Given a complete $n$-dimensional Riemannian manifold $(M,g)$ and a point $x_0\in M$, it is well-known by the theorem of Bishop-Gromov that if $\Ric\geq 0$ in $B_g(x_0,1)$ and the volume of 
$B_g(x_0,1)$ is equal to the volume $\om_n$ of the unit ball in $n$-dimensional Euclidean space (i.e. the largest it can be) then 
$B_g(x_0,1)$ is isometric to the Euclidean unit ball.
More generally, if we only have 
\beq
\label{almostEucl}
\Ric\geq -\ep\quad\text{ in }B_g(x_0,1)\qquad\text{ and }\qquad\VolB_g(x_0,1)\geq (1-\ep)\om_n,
\eeq
for some small $\ep>0$, then $B_g(x_0,1)$ can be considered to be \emph{almost Euclidean}, cf. \cite[Theorem 0.8]{colding_vol_cgnce}. 

According to Perelman's celebrated pseudolocality result, as we describe in a moment, the Ricci flow evolution of $g$, governed by the equation
$$\pl{g}{t}=-2\Ric_{g(t)},$$
protects almost-Euclidean regions, and we obtain estimates on the curvature of the flow at the centre of such regions for a definite amount of time, irrespective of how wild the flow is outside.

In this paper, we consider the behaviour of Ricci flow when confronted by a local region that is not almost-Euclidean, but has the coarser geometric control that the Ricci curvature has \emph{some} lower bound on $B_g(x_0,1)$, and the volume of this ball has \emph{some} positive lower bound. This coarser control completely changes the picture compared with almost-Euclidean control, and in particular we must restrict to the three-dimensional case because lower Ricci bounds are respected in this situation. Our main result is that this coarse geometric control is preserved for a uniform amount of time, and as a by-product we obtain $C/t$ decay of the full curvature tensor.

\begin{thm}[Main theorem]
\label{mainthm}
Suppose that $(M^3,g(t))$ is a complete Ricci flow for $t\in [0,T)$
with bounded curvature 
and $x_0\in M$. 
Suppose further that 
\beq
\label{volume_hyp_main}
\VolB_{g(0)}(x_0,1)\geq v_0>0
\eeq
and 
\beq
\label{Ric_hyp_main}
\Ric_{g(0)}\geq -K<0\qquad\text{on }B_{g(0)}(x_0,1+\si),
\eeq
for some $\si>0$.
Then there exist 
$\tilde T=\tilde T(v_0, K,\si)>0$, $\tilde v_0=\tilde v_0(v_0, K)>0$, $\tilde K=\tilde K(v_0, K,\si)>0$ and $c_0=c_0(v_0,K,\si)<\infty$ 
such that for all $t\in [0,T)\intersect (0,\tilde T)$ we have
\begin{compactenum}[\quad $1$.]
\item
\label{volume_conseq}
$\VolB_{g(t)}(x_0,1)\geq \tilde v_0>0$,
\item
\label{Ric_conseq_main}
$\Ric_{g(t)}\geq -\tilde K<0\qquad\text{on }B_{g(t)}(x_0,1)$,
\item
\label{conc3decay}
$|\Rm|_{g(t)}\leq \frac{c_0}{t}\qquad\text{on }B_{g(t)}(x_0,1)$.
\end{compactenum}
\end{thm}

Here we are using the shorthand $\VolB_g(x_0,r):=\Vol_g(B_g(x_0,r))$ to be the volume with respect to the metric $g$ of the ball centred at $x_0$ and of radius $r>0$ with respect to $g$.
We are also using the shorthand $\Ric_{g}\geq -K$ to mean that
$\Ric_{g}\geq -Kg$ as bilinear forms.

While the local control on the coarse geometry, as given by Conclusions \ref{volume_conseq} and \ref{Ric_conseq_main}, is the central issue in Theorem \ref{mainthm}, the third conclusion also gives a substantial improvement on Perelman's celebrated pseudolocality theorem in this dimension as we now explain.
Similar control to this final conclusion was established in the very interesting recent work of Hochard \cite{hochard} that was carried out independently to ours.
Perelman proved that an $n$-dimensional Ricci flow will enjoy $\al/t$ local curvature decay if it is initially locally almost Euclidean, where the closeness to Euclidean was given principally in terms of an almost-Euclidean isoperimetric inequality and a lower scalar curvature bound holding locally \cite[\S 10.1]{P1}. 
As alluded to by Perelman \cite[\S 10.5]{P1} and proved by Tian and Wang \cite{TW}, pseudolocality also applies for almost-Euclidean in the sense of \eqref{almostEucl}.

\begin{thm}[Perelman, Tian-Wang]
\label{perelman10.5}
Suppose that $(M^n,g(t))$ is a complete Ricci flow for $t\in [0,T)$ with bounded curvature, that $\al>0$, and that $x_0\in M$. Then there exists
$\ep=\ep(n,\al)>0$ such that if
\beq
\label{volume_hyp_mainAE}
\VolB_{g(0)}(x_0,1)\geq (1-\ep)\om_n
\eeq
and 
\beq
\label{Ric_hyp_mainAE}
\Ric_{g(0)}\geq -\ep\qquad\text{on }B_{g(0)}(x_0,1),
\eeq
then
for all $t\in [0,T)\intersect (0,\ep)$ we have
$$|\Rm|_{g(t)}(x_0)\leq \frac{\al}{t}.$$
\end{thm}
This almost-Euclidean result can be considered a perturbative result. The model example is the Ricci flow one obtains
starting with  a very shallow cone. For example, when $n=2$, if we flow from a cone with cone angle $2\pi (1-\be)$, then the curvature at the tip decays like 
$\frac{\be}{2(1-\be)t}$ (see e.g. Section 4, Chapter 2 of \cite{ChowKnopf}). We can take the Cartesian product with $\R$ to give a three-dimensional flow. Theorem \ref{perelman10.5} would apply when $\be\geq 0$ was extremely small.
(Strictly speaking, we would apply it from an arbitrarily small time onwards, so as to consider only smooth flows.)

In contrast, the control we obtain in Conclusion \ref{conc3decay} of Theorem \ref{mainthm} is nonperturbative, making the result and its proof quite different. Our result will handle Ricci flows starting at cones with arbitrary sharp cone points, e.g. corresponding to $\be\in [0,1)$ above as close as we like to $1$.
Whether or not Perelman's original pseudolocality result
\cite[Theorem 10.1]{P1} can be extended in the analogous way to assume only \emph{some} isoperimetric inequality rather than an almost-Euclidean one (at the expense of not being able to specify the $\al$ of the $\al/t$ curvature decay, and acquiring a dependence on the isoperimetric constant) is an interesting open question.

Our local lower bound on the Ricci curvature and our pseudolocality decay on the curvature tensor combine to give additional geometric control in terms of the distance function. 
On a smaller ball, the distance function at later times will be equivalent to the distance function initially.
The following proposition will follow rapidly from Theorem \ref{mainthm} and Lemma \ref{DCL3}, as we demonstrate at the end of Section \ref{balls_sect}.

\begin{prop}
\label{dist_equiv_prop}
In the setting of Theorem \ref{mainthm} (even allowing $\si=0$ in \eqref{Ric_hyp_main}) there exist $\ti T=\ti T(v_0,K)>0$ and $\ka=\ka(v_0, K)>0$ such that
for all $x,y\in B_{g(0)}(x_0,1/10)$ and all 
$t\in [0,T)\intersect [0,\tilde T)$ we have
\beq
d_{g(0)}(x,y)-\ka\sqrt{t}\leq 
d_{g(t)}(x,y)\leq (1+\ka t)d_{g(0)}(x,y).
\eeq
\end{prop}


\brmk
To see that Theorem \ref{mainthm} fails in dimension $4$ (and higher),
one can consider blow-downs of the Eguchi-Hanson metric. In this way, we can construct a sequence of stationary (Ricci flat) solutions to the Ricci flow satisfying the hypotheses of our theorem (with $K=0$ and uniform $v_0>0$) that exist for all time, with larger and larger curvature (that is constant in time).
\ermk

\brmk
To see that other aspects of our main theorem are sharp, key examples include flows from cones as above of various cone angles, both above and below $2\pi$, and flows that develop singularities, including shrinking cylinders and neckpinches.
\ermk

\cmt{Applying the main result \emph{just} before a singularity to
the shrinking cylinder and the neck-pinch,
we see that 
$\hat T$ also depends on $v_0$ and $K$ resp.
}

As we will prove in Section \ref{mainthm_proof_sect}, Theorem \ref{mainthm} is equivalent to the following result with the Conclusions \ref{volume_conseq}, \ref{Ric_conseq_main}
and \ref{conc3decay} replaced by the analogous conclusions in which we define balls with respect to the time $0$ metric.

\begin{thm}
\label{mainthm_time0} 
In the setting of Theorem \ref{mainthm}, there exist constants
$\tilde T$, $\tilde v_0$, $\tilde K$ and $c_0$ as in Theorem 
\ref{mainthm} such that for all $t\in [0,T)\intersect (0,\tilde T)$ we have
\begin{compactenum}[\quad $1'$.]
\item
\label{volume_conseq_prime}
$\Vol_{g(t)}(B_{g(0)}(x_0,1))\geq \tilde v_0>0$,
\item
\label{Ric_conseq_main_prime}
$\Ric_{g(t)}\geq -\ti K<0\qquad\text{on }B_{g(0)}(x_0,1)$,
\item
\label{conc3decay_prime}
$|\Rm|_{g(t)}\leq \frac{c_0}{t}\qquad\text{on }B_{g(0)}(x_0,1)$.
\end{compactenum}
\end{thm}

Theorems \ref{mainthm} and \ref{mainthm_time0} are related to two previous papers of the first author. In \cite{MilesCrelle3D}, a \emph{global} result controlling Ricci curvature from below for Ricci flows on three-dimensional manifolds was established. In \cite{MilesJGA}, a local result controlling sectional curvature from below, also in three dimensions, was given.
Also, Hochard \cite{hochard} proved lower Ricci curvature bounds of the form 
$\Ric_{g(t)} \geq -\frac{1}{t}$ that degenerate as $t\downto 0$, and volume lower bounds on balls of radius $\sqrt{t}$.

\cmt{Hochard's control won't give Proposition \ref{dist_equiv_prop}, right?}

\brmk
The boundedness of the curvature in Theorem \ref{mainthm} is required only because at some point we apply one of Perelman's pseudolocality theorems, and it is currently necessary to assume such boundedness to make the proof of pseudolocality complete.
\ermk

\brmk
With an additional argument, e.g. extending theory of Cheeger-Colding-Naber, we expect that it is possible 
to set $\tilde v_0=v_0/2$ in Theorems \ref{mainthm} and \ref{mainthm_time0}.
\ermk

\cmt{need to evaluate the remark below}

\brmk
Our work has applications to the problem of starting the Ricci flow with manifolds that 
may have unbounded curvature, but have coarse control on the geometry near each point $x_0$, given by hypotheses \eqref{volume_hyp_main} and \eqref{Ric_hyp_main} of our main theorem.
Indeed, if we combine our main theorem \ref{mainthm} and Proposition \ref{dist_equiv_prop} with the work of Hochard \cite{hochard}, then we find not only that it is possible to start the Ricci flow from such a manifold (as shown by Hochard), but also that the Ricci curvature of the resulting flow will be bounded from below, in view of Theorem \ref{mainthm}, and that we have quantitative control on the convergence as $t\downto 0$ in the Gromov-Hausdorff sense, in view of Proposition \ref{dist_equiv_prop}. In particular, this allows one to start the Ricci flow with arbitrary Gromov-Hausdorff limits of sequences of complete Riemannian manifolds satisfying hypotheses \eqref{volume_hyp_main} and \eqref{Ric_hyp_main} of our main theorem for each $x_0\in M$, \emph{and} be sure that the constructed Ricci flow satisfies estimates of the type given in the conclusions of Theorem \ref{mainthm} and Proposition \ref{dist_equiv_prop} for $t>0$, and achieves its initial data in the Gromov-Hausdorff sense.
In a subsequent paper \cite{ST2} we develop these ideas further in order to prove a stronger result than the three-dimensional Anderson-Cheeger-Colding-Tian conjecture, describing the structure of Ricci limit spaces for which we only assume a noncollapsedness condition like \eqref{volume_hyp_main} at one point in each approximating manifold.
\ermk
Implicitly, all Ricci flows in this paper are smooth, and all manifolds are connected and without boundary.

\section{Ingredients in the proof of the main theorem \ref{mainthm}}
\label{ingred_section}


In this section we give three of the main supporting results that will ultimately be combined to give the proof of Theorem \ref{mainthm}. Each of the supporting results can be viewed as weaker versions of the main theorem, in which we take one or two of the conclusions as additional hypotheses, and deduce the remaining conclusions.
We will also give an indication of some of the localisation issues that will arise in the proof by combining the three local supporting results to prove an easier global version of Theorem \ref{mainthm} more along the lines of what was proved in \cite{MilesCrelle3D}.

\bcmt{check sentence above is still valid}

\cmt{references to earlier work moved forward to first section}

We first obtain $C_0/t$ decay on the sectional curvatures as in Theorem \ref{mainthm}, but with the stronger hypothesis \eqref{Ric_hyp1} that the Ricci curvature is bounded below not just at $t=0$ but  for all times. Despite the closeness of these results, Lemma \ref{Ric_lower_lemma} will just be one ingredient in the proof of Theorem \ref{mainthm}, and will be applied at some different scale as part of a contradiction argument.
A subtle distinction is that the stronger hypothesis allows us to drop the hypothesis that the Ricci flow is part of a larger complete Ricci flow with bounded curvature.

\begin{lemma}[\bf $\Ric\geq -K$ lower bound implies $|\Rm|\leq C_0/t$ upper bound]
\label{Ric_lower_lemma}
Suppose that $(M^3,g(t))$ is a  Ricci flow for $t\in [0,T)$ such that
$B_{g(t)}(x_0,1)\subset\subset M$ for each $t\in [0,T)$ and some $x_0\in M$. 
Suppose further that 
\beq
\label{volume_hyp1}
\VolB_{g(0)}(x_0,1)\geq v_0>0
\eeq
and 
\beq
\label{Ric_hyp1}
\Ric_{g(t)}\geq -K<0\qquad\text{on }B_{g(t)}(x_0,1)\text{ for all }
t\in [0,T),
\eeq
and that $\ga\in (0,1)$ is any constant.

Then there exists $\hat T=\hat T(v_0, K, \ga)>0$, $C_0=C_0(v_0,K,\ga)<\infty$ 
and $\eta_0=\eta_0(v_0, K, \ga)>0$
such that for all $t\in (0,T)\intersect (0,\hat T)$ we have
$$|\Rm|_{g(t)}< \frac{C_0}{t}\qquad\text{on }B_{g(t)}(x_0,\ga),$$
and
\beq
\label{second_conclusion}
\VolB_{g(t)}(x_0,1) \geq \eta_0.
\eeq
\end{lemma}

\cmt{Need vol bd to apply pseudoloc, but this conclusion also makes the lemma compare with the main theorem}

\cmt{I changed the $[0,\hat T]$ from our notes to $[0,\hat T)$. Clearly the latter implies the former - unless we want $T=\hat T$, and that is convenient to allow. In fact, might be more convenient to allow all time intervals closed, which is what we use in applications.}

We prove Lemma \ref{Ric_lower_lemma} in Section \ref{Ric_lower_lemma_proof_sect}. It will involve a Perelman-style point picking procedure (for which we give a clean exposition in Lemma \ref{decay_or_no_decay}) allowing us to blow up a contradicting sequence in order to get a $\ka$-solution with positive asymptotic volume ratio, which is impossible by a result of Perelman \cite[\S 10.4]{P1}. Related arguments have been used by several authors since Perelman's work.

If one can see Lemma \ref{Ric_lower_lemma} as a version of Theorem \ref{mainthm} in which we assume Conclusion \ref{Ric_conseq_main} as an additional hypothesis, then the following lemma can be considered a version in which we assume the curvature decay of Conclusion \ref{conc3decay} as a hypothesis in order to obtain the persistence of lower Ricci bounds of Conclusion \ref{Ric_conseq_main}.

\begin{lemma}[Lower Ricci bounds]
\label{DB_conseq}
Let $c_0\geq 1$, $\ti K >0$ be arbitrary. 
Suppose that $(M^3,g(t))$ is a Ricci flow for $t\in [0,T)$, and $x_0\in M$
satisfies $B_{g(t)}(x_0,2)\subset\subset M$ for all $t\in [0,T)$.
We assume further  that
\begin{itemize}
\item[(a)] $\Ric_{g(0)}\geq -\ti K \qquad\text{on }B_{g(0)}(x_0,2)$
\item[(b)] $|\Rm|_{g(t)}\leq \frac{c_0}{t}\qquad\text{on }B_{g(t)}(x_0,2)$
for all $t\in (0,T)$.
\end{itemize}
Then there exists a $\hat T = \hat T(c_0,\ti K)>0$ 
such that 
$$\Ric_{g(t)}\geq -100\ti Kc_0\qquad\text{on }B_{g(t)}(x_0,1)$$
for all $t\in [0,T)\intersect [0,\hat T)$.
\end{lemma}

This lemma and its proof is considerably more unorthodox. We do use the remarkable properties of the distance function under Ricci flow in order to construct  useful cut-off functions, in the spirit of Perelman, following Hamilton. Indeed we give a self-contained, re-usable exposition of the theory we need in 
Section \ref{cutoff_sect} -- see Lemma \ref{cutoff}. 
Moreover, the local Ricci control of Lemma \ref{DB_conseq} can be considered a relative of the local scalar curvature control developed by B.-L. Chen \cite{strong_uniqueness}. Indeed, we will use similar control as an ingredient in the proof, and we give the theory we need in Section \ref{scalar_sect} -- see Lemma \ref{DBscalar}.
However, the proof of Lemma \ref{DB_conseq} is very different in that it requires a \emph{double bootstrap} argument. If one interprets the lemma as obtaining $L^\infty$ control on the negative part of the Ricci curvature, then our argument proceeds in two steps, the first effectively achieving only $L^p$ control in time, for finite $p$.
Lemma \ref{DB_conseq} is just a rescaling of Lemma \ref{DB}, which we prove in Section \ref{db_sect}.

A third result that we highlight at this point can be considered  the statement that if we assume both Conclusions \ref{Ric_conseq_main}
and \ref{conc3decay} in addition, then we can deduce Conclusion \ref{volume_conseq}. This result even works in higher dimension.

\cmt{One for the future: what is the non-Ricci flow version of this result}

\cmt{in the middle of changing $S$ to $T$, and $T$ to $\hat T$.}

\newcommand{\thatwasT}{\ensuremath{{\hat T}}}
\newcommand{\twasS}{\ensuremath{{T}}}

\begin{lemma}[Lower volume control]
\label{volume_control_lem}
Suppose that $(M^n ,g(t))$ is  a  Ricci flow for ${t\in [0,\twasS )}$,  
such that  $B_{g(t)}(x_0,\ga)\subset\subset M$ 
for some $x_0\in M$ and $\ga>0$,  and all $t \in [0,\twasS )$.
Assume further that
\begin{enumerate}[(i)]
\item
\label{ric_bdi}
$\Ric_{g(t)} \geq -K$ on $B_{g(t)}(x_0,\ga)$, for some $K>0$ and all $t \in [0,\twasS )$,
\item
\label{curv_bdii}
$|\Rm|_{g(t)} \leq \frac{c_0}{t} $    on
$ B_{g(t)}(x_0,\ga)$, for some $c_0<\infty$ and all $t \in
  (0,\twasS )$,
\item
\label{vol_bdiii}
$\VolB_{g(0)}(x_0,\ga) \geq v_0$ for some $v_0>0$.
\end{enumerate}
Then there exists $\ep_0 = \ep_0(v_0,K,\ga,n) >0$ and $\thatwasT  =
\thatwasT (v_0,c_0,K,\ga,n)>0 $ such that 
$$\VolB_{g(t)}(x_0,\ga)  > \ep_0$$ 
for all
$t \in [0,\thatwasT ]\cap[0,\twasS )$.
\end{lemma}

\cmt{a glance at the proof suggests that for $\ga=1$ and $K=1$, we can actually take $\ep_0$ to be a small universal constant times $v_0^n $. Scaling seems very weird, if correct.
What example would make $v_0^n $ seem like the right scaling?!!
Note also that the hypotheses are vacuous unless $v_0\leq \Om_n $ as we mention in 
\eqref{v0upperbd}.}

Thus, under somewhat weak curvature hypotheses, the volume of a ball of fixed radius cannot drop too rapidly. Without these hypotheses the theorem fails. Indeed, there exists a Ricci flow starting with the unit disc in Euclidean space whose volume becomes as small as we like in as small a time as we like. An example was given in \cite[Theorem A.3]{GT3} in two dimensions, which then trivially extends to arbitrary dimension. 

Even if one considers only \emph{complete} or \emph{closed} Ricci flows, then the volume of a unit disc could collapse to an arbitrarily small value in an arbitrarily short time, as we now briefly sketch: 
Take two round two-dimensional spheres of radius, say, $1/10$. Connect them by a tiny neck to give a rotationally symmetric dumbbell surface. We pick $x_0$ at the centre of the neck
and start the Ricci flow. Initially, the volume of the unit ball centred at $x_0$ is of order $1$ because the unit ball contains the two spheres. But virtually instantly, the two spheres fly apart and we are left standing on a really thin neck, and thus the volume of the unit ball centred at $x_0$ is now as small as we like. The example trivially extends to higher dimensions by taking a product with Euclidean space, and can be made rigorous using far simpler technology than that in \cite{RF_bursts}.

We relegate the proof of Lemma \ref{volume_control_lem} to Section \ref{vol_control_sect} because those familiar with the details of Cheeger-Colding theory (not just the statements) could take a different path
by verifying that the statements of Cheeger-Colding \cite{cheeger_colding_vol_cgnce} also hold in the case that the balls being considered are compactly contained in  possibly non-complete manifolds, at which point a more direct argument is possible.


At this point, we have three results that can each be viewed as weak forms of the main theorem. In Section \ref{mainthm_proof_sect} we will show how they can be combined to give the full theorem \ref{mainthm}. 
Much of the significance of these results is that they apply locally.
However, there are still two further ingredients that will be required in order to make the main theorem apply locally, namely the prior inclusion lemma and the pseudolocality improvement lemma, and we describe these in Sections \ref{sharkfinsect} and \ref{pseudo_improve_sect} respectively.
In order to understand the subtleties of the localisation aspects, we now sketch how one could 
prove the simpler global result in which we assume that the hypotheses of the theorem hold \emph{globally}, i.e. we have $\Ric_{g(0)}\geq -K$ throughout, and $\VolB_{g(0)}(x,1)\geq v_0>0$ for each $x\in M$, and ask for the lower volume bound of Conclusion \ref{volume_conseq} for \emph{all} $x_0$, and for the curvature control of Conclusions \ref{Ric_conseq_main} and \ref{conc3decay} globally.

By reducing (and then fixing) $v_0>0$, we can assume even that 
$\VolB_{g(0)}(x,r)\geq v_0 r^3$ for each $x\in M$ and $r\in (0,1]$, by Bishop-Gromov.
With these constants $v_0$ and $K$, and with $\ga=\half$, we can appeal to Lemma 
\ref{Ric_lower_lemma} for constants $\hat T$, $C_0$ and $\eta_0$, in preparation for the application of this lemma later. The lemma will be applied not to $g(t)$, but a
scaled-up version of $g(t)$. 
We can choose $c_0=C_0+1$ in the theorem.
In preparation to apply the lower Ricci bounds lemma \ref{DB_conseq}, also to a rescaled flow, 
we reduce $\hat T$ if necessary so that this lemma applies
with $\ti K=K/(100c_0)$. Here we assume the condition $|\Rm|_{g(t)} \leq \frac{c_0}{t}$ for the $c_0$ that we have just defined.

We now claim, as required in this global version of the theorem, that 
$|\Rm|_{g(t)}\leq \frac{c_0}{t}$ for all 
$t\in (0,\hat T/(100c_0)]$. If not, then 
take the first time $\de>0$ for which this estimate fails, and pick
$z_0\in M$ such that
$|\Rm|_{g(\de)}(z_0)= \frac{c_0}{\de}$.
If we then parabolically rescale up time by a factor so that $\de$ increases to $\hat T$, which is a factor of at least $100c_0$,
then we get a new Ricci flow $\tilde g(t)$ on $[0,\hat T]$ such that
$|\Rm|_{\tilde g(t)}\leq \frac{c_0}{t}$,
with equality at $t=\hat T$ at the point $z_0$.
Also, after the scaling, $\Ric_{\tilde g(0)}\geq -K/(100c_0)$.

We can then apply Lemma \ref{DB_conseq} to $\ti g(t)$, centred at $z_0$, to deduce that $\Ric_{\tilde g(t)}\geq -K$ in a ball centred at $z_0$,
for all $t\in [0,\hat T]$.
Applying Lemma \ref{Ric_lower_lemma}, again to $\ti g(t)$ and centred at $z_0$,
we deduce that 
$|\Rm|_{\tilde g(t)}(z_0)\leq C_0/t<c_0/t$ 
for $t\in [0,\hat T]$,
but by hypothesis, $|\Rm|_{\tilde g(\hat T)}(z_0)= \frac{c_0}{\hat T}$, which is a contradiction.
Thus, we've proved that
$$|\Rm|_{g(t)}\leq \frac{c_0}{t}\text{ over }t\in (0,\tilde T]$$
where $\tilde T:= \hat T/(100c_0)$.
Lemma \ref{DB_conseq}, 
applied this time to $g(t)$, then gives the 
bound 
$$\Ric_{g(t)}\geq - 100c_0K   \quad\text{ for }t\in [0,\ti T],$$
which is the second conclusion of the (global) theorem.
Finally, 
Lemma \ref{volume_control_lem} gives the desired volume lower bounds of the first conclusion of the theorem, and this completes the sketch proof of the main theorem in the global case.

The main subtlety that arises in the proof of the \emph{local} theorem \ref{mainthm} is that the contradiction point $z_0$ in the argument above could arise near the boundary of our local region. To overcome this we will need the full strength of the local aspects of our three main ingredients, and more. The proof can be found in Section \ref{mainthm_proof_sect}.


\section{Nested balls}
\label{balls_sect}

At the heart of this work is the remarkable behaviour of the distance function under Ricci flow.
In this section we use 
local versions of the estimates of Hamilton \cite[\S 17]{formations}
and Perelman \cite{P1},  valid in any dimension, to establish sharp control on when balls of different radii at different times are nested. 
None of the Ricci flows in this section need be complete.


\begin{lemma}[The expanding balls lemma]
\label{balls_lemma2}
Suppose $K>0$ and $0<r<R<\infty$, and define $T_0:=\frac{1}{K}\log (R/r)>0$.
Suppose $(M,g(t))$ is a Ricci flow for $t\in [0,T]$ on a manifold $M$ of any dimension.
Suppose that $x_0\in M$ and that $B_{g(t)}(x_0,R)\subset\subset M$
and $\Ric_{g(t)}\geq -K$ on $B_{g(t)}(x_0,R)$
for each $t\in [0,T]$.
Then
\beq
\label{claimline2}
B_{g(0)}(x_0,r)\subset B_{g(t)}(x_0,re^{Kt})\subset B_{g(t)}(x_0,R)
\eeq
for all $t\in [0,T]\intersect [0,T_0]$.
\end{lemma}

\begin{proof}[Proof of the Lemma \ref{balls_lemma2}]
By scaling, we may assume that $R=1$.
Define 
\beqa
T_1 &:=\sup\left\{s\in [0,\min\{T,T_0\}]\ :\ 
B_{g(0)}(x_0,r)\subset B_{g(t)}(x_0,1) \text{ for all }t\in[0,s)\right\}\\
& \quad\in (0,\min\{T,T_0\}].
\eeqa
Because $B_{g(0)}(x_0,r)\subset B_{g(t)}(x_0,1)$ for all $t\in[0,T_1)$, we can exploit the Ricci lower bound hypothesis to find that for all $x\in B_{g(0)}(x_0,r)$, we have
$$d_{g(t)}(x,x_0)\leq e^{Kt}d_{g(0)}(x,x_0),$$
for all $t\in [0,T_1)$, and by continuity, even for $t\in [0,T_1]$.
In particular, we have 
\beq
\label{dist_est}
d_{g(t)}(x,x_0)< r e^{Kt}\leq e^{-K(T_0-T_1)}
\eeq
by definition of $T_0$, and hence that
$$B_{g(0)}(x_0,r)\subset B_{g(t)}(x_0,e^{-K(T_0-T_1)})$$
for all $t\in [0,T_1]$.
If $T_1<\min\{T,T_0\}$, so that $e^{-K(T_0-T_1)}<1$, then this would contradict the definition of $T_1$. Therefore we have $T_1=\min\{T,T_0\}$, and \eqref{dist_est} implies \eqref{claimline2} as required.
\end{proof}

With an \emph{upper} Ricci bound, we can nest balls in the opposite sense.
By convention, balls with negative radii are empty.

\begin{lemma}[The shrinking balls lemma]
\label{shrinking_balls_lemma}
Suppose $(M,g(t))$ is a Ricci flow for $t\in [0,T]$ on a manifold $M$ of any dimension $n$.
Then there exists a constant $\be\geq 1$ depending only on $n$
such that the following is true.
Suppose $x_0\in M$ and that $B_{g(0)}(x_0,r)\subset\subset M$ 
for some $r>0$, and 
$\Ric_{g(t)}\leq (n-1)f^2(t)$ on $B_{g(0)}(x_0,r)$, or merely on 
$B_{g(0)}(x_0,r)\intersect B_{g(t)}\left(\textstyle{x_0,r-\frac{\be}{2}\int_0^t f}\right)$,
for each $t\in (0,T]$, where $f:(0,T]\to [0,\infty)$ is a continuous integrable function.
Then 
\beq
\label{shrinking_claimline}
B_{g(0)}(x_0,r)\supset B_{g(t)}\left(\textstyle{x_0,r-\frac{\be}{2}\int_0^t f}\right)
\eeq
for all $t\in [0,T]$.
\end{lemma}

\cmt{probably we can drop the continuity of $f$, but I think it would be an unnecessary distraction. Anyone needing it could probably get it from our statement and a messy approximation argument}

\cmt{used to have $\be=2\al$. The constant $\be$ is always at least about 6.52, so we've updated the statements above and below to have $\be\geq 1$, which simplifies phrasing of proofs. There may be some remaining simplifications in proofs to be had from this...}

Noting that $|\Rm|$ bounds the size of the largest sectional curvature, we instantly have the following corollary, which will account for most applications in this paper.

\begin{cor}[The shrinking balls corollary]
\label{shrinking_balls_cor}
Suppose $(M,g(t))$ is a Ricci flow for $t\in [0,T]$ on a manifold $M$ of any dimension $n$.
Then with $\be=\be(n)\geq 1$ as in Lemma \ref{shrinking_balls_lemma},  the following is true.
Suppose $x_0\in M$ and that $B_{g(0)}(x_0,r)\subset\subset M$ 
for some $r>0$,  
and $|\Rm|_{g(t)}\leq c_0/t$,
or more generally $\Ric_{g(t)}\leq (n-1)c_0/t$,
on $B_{g(0)}(x_0,r)\intersect B_{g(t)}(x_0,r-\be\sqrt{c_0 t})$ 
for each $t\in (0,T]$ and some $c_0>0$.
Then 
\beq
B_{g(0)}(x_0,r)\supset B_{g(t)}\left(\textstyle{x_0,r-\be\sqrt{c_0 t}}\right)
\eeq
for all $t\in [0,T]$.
More generally, for $0\leq s\leq t\leq T$, we have 
$$B_{g(s)}\left(\textstyle{x_0,r-\be\sqrt{c_0 s}}\right)
\supset B_{g(t)}\left(\textstyle{x_0,r-\be\sqrt{c_0 t}}\right).$$
\end{cor}

\cmt{In fact, we only need Ricci control on a smaller region (a la Perelman) but it's a bit complicated - we need control near the centre and near the boundary - so let's forget about it}

\begin{proof}[Proof of Lemma \ref{shrinking_balls_lemma}]
It is a result of Hamilton and Perelman (embedded in the discussion in \cite[\S 17]{formations} and \cite[Lemma 8.3(b)]{P1}, using the idea of Bonnet-Myers) that if $g(t)$ is a Ricci flow on $M$ for $t$ in a nontrivial interval containing $t_0\in \R$, and $\ga:[0,1]\to M$ is any minimising geodesic with respect to $g(t_0)$ such that $\Ric_{g(t_0)}\leq (n-1)K$ at all points along $\ga$, then 
\beq
\label{Lderiv}
\frac{d}{dt}\bigg|_{t=t_0}L_{g(t)}(\ga)\geq -\frac{\be}{2}\sqrt{K},
\eeq
where $\be:=8\sqrt{2/3}(n-1)\geq 1$. This $\be$ is to be the value of $\be$ in the lemma.

\cmt{The proof here is that if we reparametrise $\ga$ by arclength, over $[0,L]$, then 
$$-\frac{d}{dt}\bigg|_{t=t_0}L_{g(t)}(\ga)=\int_\ga \Ric(\dot\ga,\dot\ga).$$
This integral can be controlled using the second variation formula, applied to variations as in Perelman's work. We get an upper bound of $2(n-1)(2/3 Kr_0+r_0^{-1})$ where $r_0\leq L/2$, and then we optimise the estimate over $r_0$ (keeping in mind that $r_0=L/2$ might be best). It might be useful at some point to note that we only really need the upper Ricci bound within a distance $r_0$ of the endpoints.}

To prove the lemma, it suffices to show that for arbitrary $\ep>0$, we have
\beq
\label{shrinking_claimline2}
B_{g(0)}(x_0,r)\supset \overline{B_{g(t)}\left(\textstyle{x_0,r-\frac{\be}{2}\int_0^t f}-\ep(1+t)\right)}
\eeq
for all $t\in [0,T]$.

Suppose instead that this is false for some Ricci flow, $\ep>0$ and function $f$ etc.
Then there exists a first time $t_0\in(0,T]$ at which it fails, and we can find 
a point 
$$y \in \overline{B_{g(t_0)}\left(\textstyle{x_0,r-\frac{\be}{2}\int_0^{t_0} f}-\ep(1+t_0)\right)}$$
such that $d_{g(0)}(y,x_0)=r$.
We can then find a minimising geodesic $\ga$ connecting $x_0$ and $y$ with length 
\beq
\label{L1}
L_{g(t_0)}(\ga)\leq r-\frac{\be}{2}\int_0^{t_0} f-\ep(1+t_0),
\eeq
all of which lies within both $\overline{B_{g(0)}(x_0,r)}$ and
$B_{g(t_0)}\left(\textstyle{x_0,r-\frac{\be}{2}\int_0^{t_0} f}\right)$,
where $\Ric_{g(t_0)}\leq (n-1)f^2(t_0)$ has been assumed. In particular, by \eqref{Lderiv},
we have
\beq
\label{another_Lderiv}
\frac{d}{dt}\bigg|_{t=t_0}L_{g(t)}(\ga)\geq -\frac{\be}{2} f(t_0).
\eeq
On the other hand, because $t_0$ is the first time at which \eqref{shrinking_claimline2} fails, 
and $y\notin B_{g(0)}(x_0,r)$, for all $s\in [0,t_0)$ we have
$d_{g(s)}(x_0,y)>r-\frac{\be}{2}\int_0^s f-\ep(1+s)$, and in particular,
$$L_{g(s)}(\ga)>r-\frac{\be}{2}\int_0^s f-\ep(1+s).$$
Subtracting from \eqref{L1}, we obtain
$$L_{g(t_0)}(\ga)-L_{g(s)}(\ga)<-\frac{\be}{2}\int_s^{t_0} f-\ep(t_0-s),$$
and hence 
$$\frac{d}{dt}\bigg|_{t=t_0}L_{g(t)}(\ga)\leq -\frac{\be}{2} f(t_0) -\ep,$$
which contradicts \eqref{another_Lderiv}.
\end{proof}

The following lemma gives the nesting of balls whose centre is different from the centre of the balls on which we have curvature control. It also gives control on the distance function, 
for possible future use.

\begin{lemma}[Displaced balls and distance comparison lemma]
\label{DCL3}
Suppose that $(M^n,g(t))$ is a Ricci flow for $t\in [0,T]$, and 
that for some $x_0\in M$ and $r>0$, we have $B_{g(t)}(x_0,5r)\subset\subset M$ for all $t\in [0,T]$.
Suppose further that for some $K>0$ we have
$$\Ric_{g(t)}\geq -K<0\quad\text{ on }B_{g(t)}(x_0,5r),$$
for all $t\in [0,T]$.
Then for any $t\in [0,\min\{T, \frac{1}{K}\log\frac{3}{2}\}]$ there holds:
\begin{compactenum}
\item
For any $x\in B_{g(0)}(x_0,r)$ and radius $s\in [0,2r]$, we have
$B_{g(0)}(x,s)\subset B_{g(t)}(x,se^{Kt})\subset B_{g(t)}(x_0,5r)$;
\item
For any $x,y\in B_{g(0)}(x_0,r)$, we have
$d_{g(t)}(x,y)\leq e^{Kt}d_{g(0)}(x,y)$.
\end{compactenum}
If in addition $|\Rm|_{g(t)}\leq c_0/t$, or more generally
$\Ric_{g(t)}\leq (n-1)c_0/t$, on $B_{g(t)}(x_0,5r)$,
for some $c_0>0$ and for all $t\in [0,T]$, then taking $\be=\be(n)\geq 1$ 
as in Lemma \ref{shrinking_balls_lemma}, and  taking any 
$t\in [0,\min\{T,\frac{1}{K}\log \frac{5}4\}]$, there holds:
\begin{compactenum}
\setcounter{enumi}{2}
\item
For any $x\in B_{g(0)}(x_0,r)$ and radius $s\in [0,3r]$, we have
$B_{g(0)}(x,s)\supset B_{g(t)}\left(\textstyle{x,s-\be\sqrt{c_0 t}}\right)$;
\item
For any $x,y\in B_{g(0)}(x_0,r)$, if additionally $t\leq \frac{r^2}{4c_0\be^2}$, then we have
\beq
\label{distcontrolnew2}
d_{g(0)}(x,y)-\be\sqrt{c_0 t}\leq d_{g(t)}(x,y).
\eeq
\end{compactenum}
\end{lemma}

\cmt{we left a little implicit the fact that $d_{g(t)}(x,y)$ is always realised by a geodesic that lies within the ball $B_{g(t)}(x_0,5r)$.}

\begin{proof}[Proof of Lemma \ref{DCL3}]
By Lemma \ref{balls_lemma2}, with $R=2r$, we have
$$x\in B_{g(0)}(x_0,r)\subset B_{g(t)}(x_0,2r)$$
for $t\in [0,T]$ such that $t\leq \frac{1}{K}\log 2$, and hence
\beq
\label{3to5nest}
B_{g(t)}(x,3r)\subset B_{g(t)}(x_0,5r).
\eeq
Since this is where we have the Ricci lower bound, we can apply Lemma \ref{balls_lemma2} again, this time centred at $x$, and with $r$ there equal to $s\in [0,2r]$ and $R=3r$ to find that
\beq
B_{g(0)}(x,s)\subset B_{g(t)}(x,se^{Kt})\subset 
B_{g(t)}(x,3r),
\eeq
%
for $t\in [0,T]$ with $t\leq \frac{1}{K}\log \frac{3}2$, as required in the first part of the lemma.
For the second part, we can take $s=d_{g(0)}(x,y)$, in which case
$$y\in \overline{B_{g(0)}(x,s)}\subset \overline{B_{g(t)}(x,se^{Kt})}$$
as required for part 2.

Now assume that we have the additional upper curvature bound, which by \eqref{3to5nest}
will hold throughout $B_{g(t)}(x,3r)$. Part 3 then follows instantly from Corollary \ref{shrinking_balls_cor}. 
To prove part 4, we will apply part 3 with $s:=d_{g(t)}(x,y)+\be\sqrt{c_0 t}$.
Part 2 of the lemma tells us that
$d_{g(t)}(x,y)\leq \frac{5}{4}d_{g(0)}(x,y)\leq \frac{5r}{2}$, and so $s\leq 3r$ as required
for an application of part 3, which then tells us that
$$y\in \overline{B_{g(t)}(x,d_{g(t)}(x,y))}=
\overline{B_{g(t)}(x,s-\be\sqrt{c_0 t})}\subset
\overline{B_{g(0)}(x,s)},$$
i.e. that $d_{g(0)}(x,y)\leq s$, as required.
\end{proof}

\begin{proof}[Proof of Proposition \ref{dist_equiv_prop}]
By parabolically scaling up the Ricci flow $g(t)$ so that the balls of radius $1$ become of radius $2$, we can apply Theorem \ref{mainthm} with $\si=1$. Scaling back to the original flow, we are left with bounds $\Ric_{g(t)}\geq -\ti K$ and 
$|\Rm|_{g(t)}\leq \frac{c_0}{t}$ on $B_{g(t)}(x_0,1/2)$, for each 
$t\in [0,T)\intersect (0,\tilde T)$, where $\ti T$, $c_0$ and $\ti K$ depend only on $v_0$ and $K$.
We can therefore apply Lemma \ref{DCL3} with $r=1/10$. The second and fourth parts of that lemma imply the proposition, for a possibly smaller
$\ti T>0$.
\end{proof}

\section{Prior inclusion lemma}
\label{sharkfinsect}

Any smooth locally defined Ricci flow will enjoy $c_0/t$ curvature decay over some short time interval. The following useful lemma tells us that while this is true on a rapidly shrinking local ball, we can take any point $z_0$ in one of these rapidly shrinking balls, say at time $t_0$, and be sure that there is a well-controlled space-time region centred at $z_0$, that lives within the region with $c_0/t$ curvature decay.
The closer this space-time `cylinder' gets to time $t_0$, the thinner we must make it in space.

\begin{lemma}[Prior inclusion lemma]
\label{sharks_fin_lemma}
Given an integer $n\geq 2$, take $\be\geq 1$ as in Lemma \ref{shrinking_balls_lemma}.
Suppose that $(M^n,g(t))$ is a Ricci flow for $t\in [0,T)$,
$x_0\in M$, $r_0>0$, $c_0>0$, $L>0$ and 
$t_0\in (0,T)$ with $t_0<\frac{r_0^2}{\be^2c_0(L+1)^2}$. 
Suppose further that $B_{g(0)}(x_0,r_0)\subset\subset M$, and 
$|\Rm|_{g(t)}\leq\frac{c_0}{t}$, or more generally
$\Ric_{g(t)}\leq (n-1)\frac{c_0}{t}$
throughout
$B_{g(0)}(x_0,r_0)\intersect
B_{g(t)}(x_0,r_0-(L+1)\be\sqrt{c_0 t})\text{ for }t\in (0,t_0]$.
Then for any
$z_0\in \overline{B_{g(t_0)}(x_0,r_0-(L+1)\be\sqrt{c_0 t_0})}$ 
and for all $\al\in (0,1)$ and $t\in [0,\al^2 t_0]$,
we have
$$B_{g(t)}(z_0,L(1-\al)\be\sqrt{c_0 t_0})\ \subset\ 
B_{g(t)}(x_0,r_0-(L+1)\be\sqrt{c_0 t})\ \subset\ B_{g(0)}(x_0,r_0).$$
\end{lemma}
\begin{proof}
The shrinking balls corollary \ref{shrinking_balls_cor}, with $c_0$ there equal to
$(L+1)^2c_0$ here, and $r$ there equal to $r_0$ here, tells us that 
$B_{g(t)}(x_0,r_0-(L+1)\be\sqrt{c_0 t})\subset B_{g(0)}(x_0,r_0)$ for all $t\in [0,t_0]$.
We then make a second application of the shrinking balls corollary \ref{shrinking_balls_cor}, but now with $c_0$ there equal to $c_0$ here, and
with $r<r_0$ chosen so that the rapidly shrinking ball $B_{g(t)}(x_0,r_0-(L+1)\be\sqrt{c_0 t})$ above
agrees with the (previously smaller) ball $B_{g(t)}(x_0,r-\be\sqrt{c_0 t})$ at time $t_0$.
That is, we take $r=r_0-L\be \sqrt{c_0 t_0}>0$.
The output of the shrinking balls corollary is now, for each $t\in [0,t_0]$, the inclusion
$$B_{g(t)}\left(\textstyle{x_0,r-\be\sqrt{c_0 t}}\right)
\supset B_{g(t_0)}\left(\textstyle{x_0,r-\be\sqrt{c_0 t_0}}\right)
=B_{g(t_0)}(x_0,r_0-(L+1)\be\sqrt{c_0 t_0}),$$
and so by definition of $r$, we have
$$z_0\in \overline{B_{g(t)}\left(\textstyle{x_0,r_0-L\be \sqrt{c_0 t_0}-\be\sqrt{c_0 t}}\right)}.$$
Fattening by an amount $L\be\sqrt{c_0}(\sqrt{t_0}-\sqrt{t})$, we find that
$$B_{g(t)}\left(z_0,L\be\sqrt{c_0}(\sqrt{t_0}-\sqrt{t})\right)\subset
B_{g(t)}\left(x_0,r_0-(1+L)\be \sqrt{c_0 t}\right).$$
If we further constrain $t\leq \al^2 t_0$, then this smaller ball contains 
$B_{g(t)}(z_0,L(1-\al)\be\sqrt{c_0 t_0})$ as required.
\end{proof}

One key application of the lemma above is the following, in which $z_0$ is further constrained to be a first point where good curvature decay fails.

\begin{lemma}
\label{decay_or_earlier_control}
Given an integer $n\geq 2$, take $\be\geq 1$ as in Lemma \ref{shrinking_balls_lemma}.
Suppose that $(M^n,g(t))$ is a Ricci flow for $t\in [0,T)$,
$x_0\in M$, $r_0>0$, $c_0>0$ and $L>0$. 
Suppose further that $B_{g(0)}(x_0,r_0)\subset\subset M$.
Then at least one of the following assertions is true:
\begin{compactenum}
\item For each $t\in (0,T)$ with $t<\frac{r_0^2}{\be^2c_0(L+1)^2}$, we have 
$B_{g(t)}(x_0,r_0-(L+1)\be\sqrt{c_0 t})\subset B_{g(0)}(x_0,r_0)$ and 
$$|\Rm|_{g(t)}<\frac{c_0}{t}\quad\text{ throughout }
B_{g(t)}(x_0,r_0-(L+1)\be\sqrt{c_0 t}).$$
\item 
There exist $t_0\in (0,T)$ with $t_0<\frac{r_0^2}{\be^2c_0(L+1)^2}$ and 
$z_0\in \overline{B_{g(t_0)}(x_0,r_0-(L+1)\be\sqrt{c_0 t_0})}$ such that 
$$Q:=|\Rm|(z_0,t_0)=\frac{c_0}{t_0},$$
and for all $\al\in (0,1)$ and  $t\in (0,\al^2 t_0]$, we have
$$B_{g(t)}(z_0,L(1-\al)\be\sqrt{c_0 t_0})\ \subset\ 
B_{g(t)}(x_0,r_0-(L+1)\be\sqrt{c_0 t})\ \subset\  B_{g(0)}(x_0,r_0),$$
and $|\Rm|(x,t)<c_0/t$ for all
$x\in B_{g(t)}(z_0,L(1-\al)\be\sqrt{c_0 t_0})$.
\end{compactenum}
\end{lemma}

\begin{proof}
Suppose first that for each $t\in (0,T)$ with $t<\frac{r_0^2}{\be^2c_0(L+1)^2}$, we have 
$$|\Rm|_{g(t)}<\frac{c_0}{t}\quad\text{ throughout }
B_{g(t)}(x_0,r_0-(L+1)\be\sqrt{c_0 t}).$$
Then the shrinking balls corollary \ref{shrinking_balls_cor} tells us that 
$B_{g(t)}(x_0,r_0-(L+1)\be\sqrt{c_0 t})\subset B_{g(0)}(x_0,r_0)$ and we are clearly in case 1 of the lemma.
Otherwise, there must exist a first time 
$t_0\in (0,T)$ with $t_0<\frac{r_0^2}{\be^2c_0(L+1)^2}$ and 
$z_0\in \overline{B_{g(t_0)}(x_0,r_0-(L+1)\be\sqrt{c_0 t_0})}$ such that 
$$Q:=|\Rm|(z_0,t_0)=\frac{c_0}{t_0}.$$
In this case, Lemma \ref{sharks_fin_lemma} tells us that we must be in case 2.
\end{proof}


\section{{Curvature decay under Ricci lower bounds - 
proof of Lemma \ref{Ric_lower_lemma}}}
\label{Ric_lower_lemma_proof_sect}

Arguably, the most logical order for this paper would now see us proving Lemma \ref{volume_control_lem}. However, we relegate that proof to Section \ref{vol_control_sect} for the reasons discussed in Section \ref{ingred_section}.

Before beginning the proof of Lemma \ref{Ric_lower_lemma}, we give a simple but important lemma clarifying that our Ricci flows either enjoy good curvature decay properties or they contain a point of large curvature, with no point of much larger curvature nearby.

\begin{lemma}[Curvature decay or no decay lemma]
\label{decay_or_no_decay}
Given an integer $n\geq 2$, take $\be>0$ as in Lemma \ref{shrinking_balls_lemma}.
Suppose that $(M^n,g(t))$ is a Ricci flow for $t\in [0,T]$,
$x_0\in M$, $r_0>0$ and $c_0>0$. 
Suppose further that $B_{g(t)}(x_0,r_0)\subset\subset M$ for each $t\in [0,T]$.
Then at least one of the following assertions is true:
\begin{compactenum}
\item 
\label{ass1}
For each $t\in (0,T]$ with $t<\frac{r_0^2}{\be^2c_0}$, we have 
$B_{g(t)}(x_0,r_0-\be\sqrt{c_0 t})\subset B_{g(0)}(x_0,r_0)$ and 
$$|\Rm|_{g(t)}<\frac{c_0}{t}\quad\text{ throughout }B_{g(t)}(x_0,r_0-\be\sqrt{c_0 t}).$$
\item 
\label{ass2}
There exist $\bar t\in (0,T]$ with $\bar t<\frac{r_0^2}{\be^2c_0}$ and 
$\bar x\in B_{g(\bar t)}\left(x_0,r_0-\half\be\sqrt{{\vphantom{\bar T} 
c_0 \bar t}}\right)$ such that 
$$Q:=|\Rm|(\bar x,\bar t)\geq \frac{c_0}{\bar t},$$
and 
\beq
\label{Rm_4bd}
|\Rm|(x,t)\leq 4Q=4|\Rm|(\bar x,\bar t)
\eeq
whenever 
$d_{g(\bar t)}(x,\bar x)<\frac{\be c_0}{8} Q^{-1/2}$
and $\bar t-\frac18 c_0 Q^{-1}\leq t\leq \bar t$.
\end{compactenum}
\end{lemma}

\cmt{Here we are using the convenient notation
$|\Rm|(x,t):=|\Rm|_{g(t)}(x)$ since the flow is unambiguous.}

Note that in Assertion \ref{ass2}, we are obtaining the control of \eqref{Rm_4bd} for 
$$x\in B_{g(\bar t)}(\bar x,{\textstyle \frac{\be c_0}{8}} Q^{-1/2})
\subset B_{g(\bar t)}\left(\bar x,{\textstyle\frac{\be }{8}}\sqrt{c_0 \bar t}\right)
\subset B_{g(\bar t)}(x_0,r_0).$$
This lemma has as content a point picking argument of Perelman.
It is philosophically important to appreciate that although the region in Assertion \ref{ass1} is always a subset of the ball $B_{g(0)}(x_0,r_0)$, and in Assertion \ref{ass2} the point $\bar x$ that we find is always within $B_{g(\bar t)}(x_0,r_0)$, in contrast $\bar x$  could be a vast distance from $x_0$ with respect to $g(0)$.

\cmt{example is an extremely long thin cylinder with capped ends}

\begin{proof}[Proof of Lemma \ref{decay_or_no_decay}]
By parabolically rescaling $g(t)$, we may assume that $r_0=1$.
We begin by proving an initial claim that under the conditions of the lemma,  either Assertion \ref{ass1} holds (with $r_0=1$) or the following holds (or both hold).
\begin{compactenum}
\setcounter{enumi}{2}
\item
\label{ass3}
There exist $\bar t\in (0,T]$ with $\bar t<\frac{1}{\be^2c_0}$ and 
$\bar x\in B_{g(\bar t)}\left(x_0,1-\half\be\sqrt{\vphantom{\bar T}c_0 \bar t}\right)$ such that 
$|\Rm|(\bar x,\bar t)\geq c_0/\bar t$, 
and so $\bar r:=d_{g(\bar t)}(\bar x,x_0)+\frac{\be c_0}{4}|\Rm|(\bar x,\bar t)^{-1/2}
\leq 1-\frac14 \be\sqrt{c_0 \bar t}\leq 1$,
but also so that 
for all $t\in \left(0,\vphantom{\bar T}\bar t\right]$ and $x\in M$ with $d_{g(t)}(x,x_0)<\bar r$, and $|\Rm|(x,t)\geq c_0/t$,
we have
$$|\Rm|(x,t)\leq 4|\Rm|(\bar x,\bar t).$$
\end{compactenum}
By smoothness, for sufficiently small $t>0$, we will have $B_{g(t)}(x_0,1-\be\sqrt{c_0 t})\subset\subset M$ and $|\Rm|_{g(t)}< c_0/t$ within that ball.
By the shrinking balls corollary \ref{shrinking_balls_cor}, while these facts are true, we must have
$B_{g(t)}(x_0,1-\be\sqrt{c_0 t})\subset B_{g(0)}(x_0,1)$, and in particular the only way that Assertion \ref{ass1} of the lemma can fail is if there exists some first time $t_1\in (0,T]$ with $t_1<\frac{1}{\be^2c_0}$ and a point $x_1\in \overline{B_{g(t_1)}(x_0,1-\be\sqrt{c_0 t_1})}$ such that $|\Rm|(x_1,t_1)=c_0/{t_1}$.

If $t_1$ and $x_1$ serve as the $\bar t$ and $\bar x$ that we seek for Assertion \ref{ass3} to hold, then we have proved our initial claim. If not, then there must exist some $t_2\in (0,t_1]$ and $x_2\in M$ with 
\beqa
d_{g(t_2)}(x_2,x_0)&<d_{g(t_1)}(x_1,x_0)+\frac{\be c_0}{4}|\Rm|(x_1,t_1)^{-1/2}\\
&\leq (1-\be\sqrt{c_0 t_1})+\frac14\be\sqrt{c_0 t_1}\\
&=1-(1-\frac14)\be\sqrt{c_0 t_1}\\
&\leq 1-\frac34\be\sqrt{c_0 t_2}
\eeqa
such that $|\Rm|(x_2,t_2)\geq c_0/t_2$ and 
$|\Rm|(x_2,t_2)>4|\Rm|(x_1,t_1)$. Again, if $t_2$ and $x_2$ serve as the $\bar t$ and $\bar x$ that we seek, then our initial claim is proved. If not, we pick $t_3\in (0,t_2]$ and $x_3\in M$ with
\beqa
d_{g(t_3)}(x_3,x_0)&<d_{g(t_2)}(x_2,x_0)+\frac{\be c_0}{4}|\Rm|(x_2,t_2)^{-1/2}\\
&\leq 
d_{g(t_1)}(x_1,x_0)+\frac{\be c_0}{4}|\Rm|(x_1,t_1)^{-1/2}
+\half\frac{\be c_0}{4}|\Rm|(x_1,t_1)^{-1/2}\\
&\leq 1-(1-\frac14-\frac18)\be\sqrt{c_0 t_1}\\
&\leq 1-\frac58\be\sqrt{c_0 t_3}
\eeqa
such that $|\Rm|(x_3,t_3)\geq c_0/t_3$
and $|\Rm|(x_3,t_3)>4|\Rm|(x_2,t_2)$.
After $k-1$ iterations of this procedure, we have $t_k\in (0,T]$ with $t_k<\frac{1}{\be^2c_0}$, and $x_k\in M$ with
\beqa
d_{g(t_k)}(x_k,x_0)
&<
1- \left(1- \frac14 -\frac18 - \frac1{16} \cdots -\frac1{2^{k}}\right) \be \sqrt{c_0t_1}\\
&=1-(\half+2^{-k})\be\sqrt{c_0 t_1}\\
&\leq 1-\half\be\sqrt{c_0 t_k}
\eeqa
such that $|\Rm|(x_k,t_k)\geq c_0/t_k$ and 
$|\Rm|(x_k,t_k)>4^{k-1}|\Rm|(x_1,t_1)$.
Since the curvature is blowing up under this iteration, but the curvature is bounded uniformly over $t\in [0,T]$ and $B_{g(t)}(x_0,1)$, the iteration must eventually terminate, and our initial claim follows.

\cmt{Another way of finding $\bar x$ and $\bar t$ is to maximise the quantity
$$F(x,t):=\psi_t(d_{g(t)}(x,x_0))^2|\Rm|(x,t)$$
over all $t\in (0,t_1]$ and $x\in B_{g(t)}(x_0,1-\half\be\sqrt{c_0 t})$
such that $|\Rm|(x,t)\geq c_0/t$, where $\psi_t(r)$ is the function that is $1$ for 
$0\leq r\leq 1-\be\sqrt{c_0 t}$, zero for $r\geq  1-\half\be\sqrt{c_0 t}$,
and linear in between.}

To prove the lemma, it suffices to show that when Assertion \ref{ass3} holds, then Assertion \ref{ass2} must hold using the same $\bar x$ and $\bar t$ (and with $r_0=1$).

We claim that for $\bar t-\frac18 c_0 Q^{-1}\leq t\leq \bar t$, 
and $x\in M$ with 
$d_{g(t)}(x,x_0)<\bar r$, 
we have \eqref{Rm_4bd}.
For such a value of $t$, because $Q=|\Rm|(\bar x,\bar t)\geq c_0/\bar t$,
we deduce that $\frac78\bar t\leq t$.
Thus if the claim were not true for some $x$, $t$, 
then we would have
$$|\Rm|(x,t)>4Q\geq \frac{4c_0}{\bar t}\geq \frac{7c_0}{2t}\geq \frac{c_0}{t},$$
and so by Assertion \ref{ass3}, we deduce that \eqref{Rm_4bd} must hold after all, giving a contradiction.

It remains to show that the ball $B_{g(\bar t)}(\bar x,\frac{\be c_0}{8} Q^{-1/2})$
considered in Assertion \ref{ass2} lies within the ball $B_{g(t)}(x_0,\bar r)$ where we have just established the curvature bound \eqref{Rm_4bd}.
Using the curvature bound, we can apply Lemma \ref{shrinking_balls_lemma} with constant $f=2Q^{1/2}$ and $r=\bar r$ to find that 
$$B_{g(t)}(x_0,\bar r)\supset B_{g(\bar t)}(x_0,\bar r-\be Q^{1/2}(\bar t-t))
.$$
But by the constraint on $t$, we have
$$\bar r-\be Q^{1/2}(\bar t-t)\geq \bar r-\frac{1}{8}\be c_0 Q^{-1/2}
= 
d_{g(\bar t)}(\bar x,x_0)+\frac{\be c_0}{8}Q^{-1/2},$$
and so 
$$B_{g(t)}(x_0,\bar r)\supset B_{g(\bar t)}(x_0,d_{g(\bar t)}(\bar x,x_0)+\frac{\be c_0}{8}Q^{-1/2})\supset B_{g(\bar t)}(\bar x, \frac{\be c_0}{8}Q^{-1/2}),$$
as required.
\end{proof}

We now turn to the proof of Lemma \ref{Ric_lower_lemma}.
A global version of this result can be found in  
\cite[Lemma 4.3]{MilesCrelle3D}.

\begin{proof}[Proof of Lemma \ref{Ric_lower_lemma}]
By Bishop-Gromov, $\VolB_{g(0)}(x_0,\ga)$ has a positive lower bound depending only on
$v_0$, $K$ and $\ga$.
Applying Lemma \ref{volume_control_lem} to $g(t)$, we see that there exists $\ep_0>0$ depending only on $v_0$, $K$ and $\ga$ such that 
for each $c_0<\infty$, there exists $\hat T$ depending on $v_0$, $K$, $\ga$ and $c_0$ such that
prior to time $\hat T$ and while $|\Rm|_{g(t)}\leq c_0/t$ still holds on $B_{g(t)}(x_0,\ga)$, we have a lower volume bound
\beq
\label{extra_volume_hyp}
\VolB_{g(t)}(x_0,1)\geq \ep_0. 
\eeq

From this we deduce that it suffices to prove the lemma with the additional hypothesis that 
\eqref{extra_volume_hyp} holds for each $t\in [0,T)$.
In particular, we can ignore the second conclusion \eqref{second_conclusion}.

Let us assume that the lemma is false, even with the extra hypothesis \eqref{extra_volume_hyp},
for some $v_0,K>0$ and $\ga\in (0,1)$. 
%
Then for any sequence $c_n\to\infty$, we can find Ricci flows that fail the lemma (despite the extra hypothesis \eqref{extra_volume_hyp}) with $C_0=c_n$ in an arbitrarily short time, and in particular within a time $t_n$ that is sufficiently small so that $c_nt_n\to 0$ as $n\to\infty$.
By reducing $t_n$ to the first time at which the desired conclusion fails, we have
a sequence of three-dimensional Ricci flows $(M_n,\ti g_n(t))$
for $t\in [0,t_n]$ 
with $t_n\downto 0$, and even $c_n t_n\to 0$, and a sequence of points $x_n\in M_n$
with $B_{\ti g_n(t)}(x_n,1)\subset\subset M_n$ for each $t\in [0,t_n]$, 
such that 
\beq
\label{volume_hyp3}
\VolB_{\ti g_n(0)}(x_n,1)\geq v_0,
\eeq
\beq
\label{extra_volume_hyp2}
\VolB_{\ti g_n(t)}(x_n,1)\geq \ep_0
\qquad\text{for all }t\in [0,t_n],
\eeq
\beq
\label{Ric_hyp3}
\Ric_{\ti g_n(t)}\geq -K\qquad\text{on }B_{\ti g_n(t)}(x_n,1)\text{ for all }
t\in [0,t_n],
\eeq
and
\beq
\label{curv_hyp}
|\Rm|_{\ti g_n(t)}<\frac{c_n}{t}\quad\text{ on }B_{\ti g_n(t)}(x_n,\ga)\text{ for }
t\in (0,t_n),
\eeq
but so that
\beq
\label{going_to_infinity}
|\Rm|_{\ti g_n(t_n)}=\frac{c_n}{t_n}\quad\text{ at some point in }
\overline{B_{\ti g_n(t_n)}(x_n,\ga)}.
\eeq
As a consequence, if we apply Lemma \ref{decay_or_no_decay} to each $\ti g_n(t)$,
with $r_0=(1+\ga)/2$ and $c_0=c_n$
(after deleting a finite number of the initial terms so that $c_nt_n$ is small enough)
we find that Assertion \ref{ass1} cannot hold, and thus Assertion \ref{ass2} must hold for each, giving times $\bar t_n\in (0,t_n]$ and points
$\bar x_n\in B_{\ti g_n(\bar t_n)}(x_n,r_0-\half\be\sqrt{\vphantom{\bar T}c_n \bar t_n})$ such that 
\beq
\label{Rm_4bdstar}
|\Rm|_{\ti g_n(t)}(x)\leq 4|\Rm|_{\ti g_n(\bar t_n)}(\bar x_n)
\eeq
whenever 
$d_{\ti g_n(\bar t_n)}(x,\bar x_n)<\frac{\be c_n}{8} Q_n^{-1/2}$
and $\bar t_n-\frac18 c_n Q_n^{-1}\leq t\leq \bar t_n$,
where $Q_n:=|\Rm|_{\ti g_n(\bar t_n)}(\bar x_n)\geq c_n/\bar t_n\to \infty$.

\cmt{now have to use this notation rather than $|\Rm|(x,t)$ type notation in order to specify the metric.}

Conditions \eqref{extra_volume_hyp2} and \eqref{Ric_hyp3}, together with Bishop-Gromov, imply that we have uniform volume  
ratio control
\beq
\label{vol_rat}
\frac{\VolB_{\ti g_n(\bar t_n)}(\bar x_n,r)}{r^3}\geq \eta>0
\eeq
for all $0<r<(1-\ga)/2$, 
where $\eta$ depends on $\ep_0$, $K$ and $\ga$.
\cmt{This part is why we need $r_0$ controllably less than $1$.}

We then perform a parabolic rescaling to give new Ricci flows defined by
$$g_n(t):=Q_n\ti g_n(\tfrac{t}{Q_n}+\bar t_n),$$
for $t\in [-\frac18 c_n,0]$.
The scaling factor is chosen so that 
and
\beq
|\Rm|_{g_n(0)}(\bar x_n)= 1
\eeq
but by \eqref{Rm_4bdstar} the curvature of $g_n(t)$ is uniformly bounded 
for $t\in  [-\frac18 c_n,0]$ and $x\in B_{g_n(0)}(\bar x_n, \frac18 \be c_n)$.
The volume ratio estimate \eqref{vol_rat} transforms to 
\beq
\label{vol_rat2}
\frac{\VolB_{g_n(0)}(\bar x_n,r)}{r^3}\geq \eta>0
\eeq
for all $0<r<\frac{1-\ga}2 Q_n^{1/2}\to\infty$.

With this control we can apply Hamilton's compactness theorem to give convergence
$(M_n,g_n(t),\bar x_n)\to (N,g(t),x_\infty)$, for some complete bounded-curvature Ricci flow $(N,g(t))$, for $t\in (-\infty,0]$, and $x_\infty\in N$.
By Hamilton's refinement of Hamilton-Ivey pinching, in particular its application given by Chow and Knopf \cite[Corollary 9.8]{ChowKnopf}, we know that 
every three-dimensional complete ancient Ricci flow of bounded curvature, and in particular $g(t)$, must in fact have \emph{nonnegative} sectional curvature.

\cmt{Interesting that this is the point that the argument would fail with Eguchi-Hanson in 4D}

Moreover, \eqref{vol_rat2} passes to the limit to force $g(t)$ to have positive asymptotic volume ratio.
This in turn tells us that $g(t)$ is a $\kappa$-solution in the sense of Perelman \cite[\S 11]{P1}. But Perelman's theorem \cite[\S 11.4]{P1} then tells us that the asymptotic volume ratio at each time must be zero, which is a contradiction.
\end{proof}

\section{Pseudolocality improvement lemma}
\label{pseudo_improve_sect}

In this section we develop Lemma \ref{Ric_lower_lemma} using Perelman's pseudolocality lemma. In order to apply that result, we need to assume that we are working on a complete bounded-curvature Ricci flow. 

The idea is that whereas before we assumed that the Ricci curvature was bounded below for all times, now we only assume such a lower bound for some initial time period.
This improvement will ultimately help us to localise our results. A previous instance where pseudolocality was used in a localisation argument is \cite[Theorems 1.1 and 1.5]{MilesJGA}
although the argument in the current paper is considerably shorter.  
R. Hochard used a related method to prove  \cite[Theorem 2.4]{hochard}.

\begin{lemma}
\label{pseudo_improve_lemma3}
Given $v_0>0$ and $K>0$, there exist $\hat T\in (0,1]$, 
$C_0\in [1,\infty)$ and $\al\in (0,1)$, such that the following is true.
Suppose that $(M^3,g(t))$ is a complete bounded-curvature Ricci flow for $t\in [0,T]$, with $0<T\leq \hat T$.
If, for some $x_0\in M$, we have
\beq 
\label{volume_hyp2new2}
\VolB_{g(0)}(x_0,1)\geq v_0>0
\eeq
and 
\beq
\label{Ric_hyp2new}
\Ric_{g(t)}\geq -K<0\qquad\text{on }B_{g(t)}(x_0,1)\text{ for all }
t\in [0, \al^2 T],
\eeq
then for all $t\in (0, T]$ we have
\beq
\label{pseudo_imp_conseq2}
|\Rm|_{g(t)}(x_0)\leq \frac{C_0}{t}. 
\eeq
\end{lemma}



We will need a slight extension of Perelman's second pseudolocality result:

\begin{thm}[{cf. Perelman \cite[\S 11.3]{P1}}]
\label{perelman11.3}
Given $n\in \N$ and $\ti v_0>0$, there exists $\ep>0$ such that if $r_0>0$ and $(M^n,g(t))$ is a complete bounded-curvature Ricci flow for $t\in [0,T]$, $0<T\leq (\ep r_0)^2$, with the properties that $|\Rm|_{g(0)}\leq r_0^{-2}$
on $B_{g(0)}(x_0, r_0)$ and $\VolB_{g(0)}(x_0,r_0)\geq \ti v_0r_0^n$, for some $x_0\in M$, then $|\Rm|_{g(t)}\leq (\ep r_0)^{-2}$ for any $t\in [0,T]$, throughout
$B_{g(t)}(x_0,\ep r_0)$.
\end{thm}

\begin{proof}[Proof of Theorem \ref{perelman11.3}]
By scaling, we may assume that $r_0=1$.
The only difference to what was stated by Perelman is that he assumed that $\VolB_{g(0)}(x_0,1)$ was almost the volume of the unit ball in Euclidean space. To reduce to that case, note that our weaker volume hypothesis, coupled with the curvature hypothesis, implies a positive lower bound on the injectivity radius $\inj_{g(0)}(x_0)$ depending only on $n$ and $v_0$. 
Gunther's theorem \cite[\S 3.101]{GHL} then tells us that the volume ratio
$\VolB_{g(0)}(x_0,r)/r^n$ is large enough to invoke Perelman's version for $r>0$ sufficiently small depending on $n$ and $v_0$, giving a curvature bound at later times on balls of an even smaller radius.
\end{proof}

\begin{proof}[{Proof of Lemma \ref{pseudo_improve_lemma3}}]
For the $v_0$ and $K$ as in the lemma, we can pick $\hat T$ to be the constant
$\hat T(v_0,K,1/2)$ given by Lemma \ref{Ric_lower_lemma}, or take $\hat T=1$, whichever is the smaller. For each $\al\in (0,1)$, that lemma then applies to our Ricci flow $g(t)$ 
to give that
for each $t\in (0,\al^2 T]$ we have $|\Rm|_{g(t)}<C_0/t$ on $B_{g(t)}(x_0,1/2)$
and $\VolB_{g(t)}(x_0,1) \geq \eta_0>0$, with $C_0$ and $\eta_0$ depending only on
$v_0$ and $K$.
Without loss of generality, we may assume that $C_0\geq 1$.

Thus we have established the desired conclusion for 
$t\in (0,\al^2 T]$. 
We now need to show that we can establish a curvature bound  at $x_0$ for the remaining time 
$t\in [\al^2 T, T]$, 
provided we choose 
$\al\in (1/2,1)$ large enough.

Defining $r_0:=\half\sqrt{\frac{T}{C_0}}\leq \half$,
we have shown that 
$$|\Rm|_{g(\al^2 T)}<\frac{C_0}{\al^2 T}\leq \frac{4C_0}{T}=r_0^{-2}
\qquad\text{ on }B_{g(\al^2 T)}(x_0,1/2)\supset 
B_{g(\al^2 T)}(x_0,r_0).$$
Bishop-Gromov will give us the volume bound $\VolB_{g(\al^2 T)}(x_0,r_0)\geq \ti v_0 r_0^3$ for some $\ti v_0>0$ depending only on 
$\eta_0$ and $K$, and hence only on $v_0$ and $K$.
We can then apply the pseudolocality theorem \ref{perelman11.3}, starting at time $\al^2 T$, to deduce a curvature bound 
$$|\Rm|_{g(t)}\leq (\ep r_0)^{-2}=\frac{4C_0}{\ep^2 T}
\quad\text{ for } 
t\in [\al^2 T,\al^2 T+(\ep r_0)^2] \cap [0,T] \quad\text{ over }
B_{g(t)}(x_0,\ep r_0),$$
where $\ep\in (0,1)$ depends only on $\ti v_0$, i.e. only on $v_0$ and $K$.

Thus 
we have established the desired conclusion provided we pick $\al\in (1/2,1)$ large enough so that $\al^2 T+(\ep r_0)^2\geq  T$, e.g. we can take $\al^2  =  1- \ep^2/(4C_0)$,
albeit with \eqref{pseudo_imp_conseq2} holding for $C_0$ replaced with $C_0'$ sufficiently large so that $C_0'/t$ dominates
$\frac{4C_0}{\ep^2 T}$ over this time interval
$[\al^2 T,\al^2 T+(\ep r_0)^2] \cap [0,T]$ 
(e.g. we can choose $C_0'=4C_0\ep^{-2}$).
\end{proof}


\newcommand{\CO}{c_0}
\newcommand{\ellH}{H}


\section{Perelman cut-off functions}
\label{cutoff_sect}

In this section we explain how it is possible to construct useful cut-off
functions that are sub-solutions to the heat equation on manifolds
evolving by Ricci flow using the estimates and methods of Perelman \cite{P1}.

\cmt{We're working on a time interval $[0,T)$, so  the Ricci flow could be singular as $t\upto T$, but this doesn't seem to cause any problems}

\begin{lemma}{\label{cutoff}}
Let $\CO >0$, $\ep \in (0,1)$, $0<r_1 <r_2 $ and $n \in \N$ be arbitrary. 
Suppose that $(M^n,g(t))$   is
a smooth Ricci flow for $t\in [0,T)$, and $x_0\in M$
satisfies $B_{g(t)}(x_0,r_2 )\subset\subset M$ for all $t\in [0,T)$.
We assume further  that
$$\Ric_{g(t)}\leq \frac{\CO (n-1)}{t}\qquad\text{on }B_{g(t)}(x_0,\sqrt{t}),$$
%
%
for all $t\in (0,T)$.
Then there exist positive constants $\hat T = \hat
T(\CO ,r_1 ,r_2 ,n)>0$, $k= k(r_1 ,r_2 ,\ep)>0$, $V=V(r_1 ,r_2 ,\ep)>0$,  and a locally Lipschitz continuous function 
$h: M \times [0,\hat T) \cap [0,T) \to \R$, such that
\begin{itemize}
\item[(i)]  $h(\cdot,t) = e^{-kt} $ on $B_{g(t)}(x_0,r_1 )$ and $h(\cdot,t) =0 $ outside  $B_{g(t)}(x_0,r_2 )$ for all $t\in [0,\hat T) \cap [0,T)$, and $h(x,t) \in [0,e^{-kt}]$ for all $(x,t) \in M  \times [0,\hat T) \cap [0,T)$
\item[(ii)] $\partt h \leq \lap_{g} h $  and $|\grad h| \leq Vh^{1-\ep}$
 {\it in the following barrier sense}. For any  $(x,t) \in M
  \times (0,\hat T) \cap (0,T) $ we can find a
space-time neighbourhood $\curlO\subset M\times (0,\hat T) \cap (0,T)$ of $(x,t)$ and a
 {\it smooth} function $\ellH :\curlO \to [0,\infty)$ such that $\ellH  \leq h $ on
 $\curlO$, $\ellH (x,t) = h(x,t)$ and 
$\partt \ellH (x,t)  \leq \lap_{g(t)} \ellH  (x,t)$ and $|\grad \ellH |(x,t) \leq V|\ellH (x,t)|^{1-\ep}$.
\end{itemize}
\end{lemma}

\cmt{new subtle change above - we restricted $H$ to be nonnegative. This is nice e.g. at the end of the next section, when estimating}

\brmk
An inspection of the proof shows that in fact we can take $V=\frac{\al}{\ep (r_2-r_1)}$,
$k=\frac{\al}{\ep (r_2-r_1)^2}$,
and
$\hat T=\min\{r_1^2,\al\frac{(r_2 -r_1 )^2 }{ (\CO+1)^2n^2}\}$
for some universal $\al>0$.
\ermk

\cmt{as a result, I ditched the $n$ and $\CO$ dependencies of $k$. Also propagates into next couple of sections}

\brmk
\label{k_clarification}
Although we state the lemma for a given $k$, it will then also hold if we increase $k$ to some other $\hat k\geq k$, simply by replacing 
$h(x,t)$ by $e^{(k-\hat k)t}h(x,t)$.
\ermk

\begin{proof}
The construction is similar to the one given in Perelman \cite{P1}. The definition of $h$ will follow \cite{MilesGandT}.
Let  $\phi:\R \to [0,1]$ be a smooth function
with the following properties.
\begin{compactenum}[(i)]
\item $ \phi(r) = 1$ for all $r \leq r_1 $, 
$ \phi(r) = 0$  for all $r \geq r_2  $,
\item $ \phi$ is decreasing: $ \phi' \leq 0$,
\item $ \phi'' \geq -a \phi$,
\item $| \phi'|^2 \leq b$
\end{compactenum}
where $a$ and $b$  can each be taken as positive universal constants multiplied by $(r_2-r_1)^{-2}$.
To find such a $\phi$, first construct a function $\ti \phi$ having the
properties (i)-(iv)  for
$r_1 =1,$ $r_2 =2$ and for some universal constants $a,b>0$, and then define
 $\phi(x) = \ti \phi( \frac{x + r_2  - 2r_1 }{r_2 -r_1 }  )$.
%

For $0<r_1 <r_2 <1$, set $\eta_{r_1 ,r_2 } = \eta = \phi^p$, with $p:=1/\ep$.
Then we have:
$|\eta'|^2= |p\phi^{p-1} \phi'|^2 = p^2 \phi^{2p-2}|\phi'|^2 \leq
bp^2(\phi^{p})^{\frac{2(p-1)}{p}} = V^2 |\eta|^{2(1-\ep)}
$ with $V^2 = bp^2$,
and hence 
\beq
\label{etaprime}
|\eta'| \leq V |\eta|^{1-\ep}.
\eeq
Furthermore, 
$\eta''(x)= p(p-1)\phi^{p-2} (\phi')^2 + p\phi^{p-1}\phi'' \geq
-ap\phi^{p-1}\phi = -ap\phi^{p} = -ap\eta$, and so with $k = ap$, we have
\beq
\label{etaprimeprime}
\eta''\geq  -k\eta,
\eeq
Let $t \in (0,T)\cap (0,r_1 ^2)$.
Then we have $\Ric\leq (n-1)K$ on $B_{g(t)}(x_0,r_0)$,
where $K:=\frac{\CO }{t}$ and $r_0 = \sqrt{t}\leq r_1$. 
Using  \cite[Lemma 8.3(a)]{P1}, for
$x \in  B_{g(t)}(x_0,r_2 )\setminus B_{g(t)}(x_0,r_0)$ we have
\beqa
\left(\partt - \lap\right) d(x,t) & \geq -(n-1)(\frac 2 3 Kr_0 + r_0^{-1}) \\
& = -(n-1)( \frac 2 3  \frac{\CO }{t}\sqrt{t}+ \frac{1}{ \sqrt{t}} )\\
& = -\frac{m_0}{\sqrt{t}}
\eeqa
where $d(x,t)$ is shorthand for $d_{g(t)}(x,x_0)$, and 
$m_0 = (n-1)(\frac{2}{3}  \CO   + 1)$.
This differential inequality is to be understood in the following barrier sense: Each $(x,t)$ as considered has some space-time neighbourhood $\curlO\subset M\times (0,T)$ on which there exists a smooth function
$D:\curlO\to [0,\infty)$ with the properties that
(i) $d(y,s)\leq D(y,s)$ for all $(y,s)\in \curlO$, 
(ii) $d(x,t)= D(x,t)$, and
(iii) $\left(\partt - \lap\right) D(x,t) \geq -\frac{m_0}{\sqrt{t}}$.
(Perelman proves this by constructing a family of piecewise
smooth paths from $x_0$ to points $y$ near $x$, whose lengths with respect to any smooth metric represent a smooth  function of $y$. He then sets $D(y,s)\geq d(y,s)$ to be the length of the path to $y$ with respect to $g(s)$ and shows that it satisfies the required differential inequality.) Note that because $d$ is a Lipschitz function with $|\grad d|\leq 1$ where it is differentiable, we see
that $|\grad D|(x,t)\leq 1$ automatically.

Choose $\ti r_1  = \frac{r_1   + r_2 }{2}$ between $r_1 $ and $r_2 $, and let $\eta=\eta_{\ti r_1 ,r_2 }$ be the cut-off function for
$0<\ti r_1 < r_2  $ as described above, and
define $h:M\times [0,\hat T)\cap [0,T)\to [0,1]$ by
$$h(x,t) = e^{-kt} \eta (d(x,t) +  4m_0 \sqrt{t}),$$ 
where $\hat T:=\min\{r_1^2,\frac{(r_2 -r_1 )^2 }{ (8m_0)^2}\}$.
As demanded by the lemma, if $d(x,t)\geq r_2$ then
$h(x,t)=0$. Moreover, 
if $d(x,t)< r_1$ (note $t< \hat T$) then $d(x,t) +  4m_0 \sqrt{t}<r_1+(r_2 -r_1)/2=\ti r_1 $, 
and we must have $h(x,t)=e^{-kt}$.
Clearly $h(x,t)\in [0,e^{-kt}]$ always.
This settles part (i) of the lemma.

To make sense of the differential inequalities of the lemma, we need to define a barrier function $H$ near an arbitrary 
$(x,t)\in M\times (0,\hat T)\cap (0,T)$.
For the given $t$, if $x$ lies outside $B_{g(t)}(x_0,r_2)$, then we have already seen that $h(x,t)=0$ and we can take the barrier $H\equiv 0$ on an arbitrary neighbourhood of $(x,t)$.
If $x$ lies within $B_{g(t)}(x_0,r_0)$, then we have already seen that $h(x,t)=e^{-kt}$, and revisiting that argument we see that if fact $h(y,s)=e^{-ks}$ for $(y,s)$ in a neighbourhood of $(x,t)$. Therefore in this case we can set $H(y,s)=e^{-ks}$ in that neighbourhood.

In the remaining case that $x \in  B_{g(t)}(x_0,r_2 )\setminus B_{g(t)}(x_0,r_0)$, 
we define
$$H(y,s) = e^{-ks} \eta (D(y,s) +  4m_0 \sqrt{s}),$$ 
for $(y,s)$ in the neighbourhood $\curlO$ of $(x,t)$ where $D$ is defined as before,
truncated if necessary so that $\curlO$ lies within the domain of definition
of $h$.
Clearly $H(x,t)=h(x,t)$, and because $\eta$ is a decreasing function, we have $H(y,s)\leq h(y,s)$ for all $(y,s)\in \curlO$.
To prove that $|\grad h| \leq Vh^{1-\ep}$ in the claimed barrier sense, we must merely prove the same inequality for $H$ at $(x,t)$, so we compute, using \eqref{etaprime},
$|\grad H(x,t)|\leq e^{-kt}|\eta '(D(x,t) +4m_0\sqrt{t})|
\leq V e^{-kt}|\eta (D(x,t) +4m_0\sqrt{t})|^{1-\ep}  \leq
V |H(x,t)|^{1-\ep}e^{-\ep kt}\leq 
V |H(x,t)|^{1-\ep}$, as desired.
Furthermore, using $\eta' \leq 0$ and 
\eqref{etaprimeprime}, we see that 
\beqa
(\partt - \lap) H (x,t)& = e^{-kt}\eta'( D(x,t) + 4m_0\sqrt{t})  \cdot (\partt - \lap ) D(x,t) \\
&\qquad +  e^{-kt}\eta'( D(x,t) + 4m_0\sqrt{t})  \cdot \frac{2m_0}{\sqrt t}\\
&\qquad -e^{-kt} \eta''(D(x,t) + 4m_0\sqrt{t})|\grad D|^2(x,t) \\
& \qquad -ke^{-kt}\eta(D(x,t) +4m_0\sqrt{t}) 
\\
& \leq  -e^{-kt}\eta'( D(x,t) + 4m_0\sqrt{t}) \cdot  \frac{m_0}{\sqrt
  t}+ e^{-kt} \cdot  \eta'( D(x,t) + 4m_0\sqrt{t})\cdot\frac{2m_0}{\sqrt t} \\
  & \qquad + ke^{-kt} \eta(D(x,t) +4m_0\sqrt{t})
-ke^{-kt}\eta(D(x,t) + 4m_0\sqrt{t})  \\
& \leq e^{-kt} \eta'( D(x,t) + 4m_0\sqrt{t})\cdot\frac{m_0}{\sqrt t}   \\
&\leq 0,
\eeqa
which completes the proof.
\end{proof}


\section{Local lower scalar curvature bounds}
\label{scalar_sect}

In this section we prove local preservation of lower bounds for the scalar curvature, in preparation for a similar result for Ricci curvature. Our lemma should be compared with earlier
estimates of B.-L. Chen \cite[Proposition 2.1]{strong_uniqueness}, and the techniques in \cite{MilesGandT}.

\cmt{we can't say that Theorem 4.1 of \cite{MilesGandT} is similar! That result is completely different. It's the techniques of the proof where there is some overlap.}

\begin{lemma}{\label{Scalarestimate}}
\label{DBscalar}
Let $\CO , K >0$, $\ga \in (0,1)$  and $n \in \N$ be arbitrary. 
Suppose that $(M^n,g(t))$ is a Ricci flow for $t\in [0,T)$, and $x_0\in M$
satisfies $B_{g(t)}(x_0,1)\subset\subset M$ for all $t\in [0,T)$.
We assume further  that 
\begin{compactenum}[(a)]
\item $\Sc_{g(0)} \geq -K \qquad\text{on }B_{g(0)}(x_0,1)$
\item 
$\Ric_{g(t)}\leq \frac{\CO (n-1)}{t}\qquad\text{on }B_{g(t)}(x_0,\sqrt{t})$
for all $t\in (0,T)$.
\end{compactenum}
Then there exist $\hat T = \hat T(\CO ,K,\ga,n)>0$
and $\si=\si(K,\ga)>0$
such that 
$$\Sc_{g(t)} \geq - Ke^{\si t}\geq -2K \qquad\text{on }B_{g(t)}(x_0,1-\ga)$$
for all $t\in [0,T)\intersect [0,\hat T)$.
\end{lemma}

\brmk
In fact, by incorporating the shrinking balls corollary \ref{shrinking_balls_cor}, we need only assume that the unit ball centred at $x_0$ is compactly contained in $M$ at time $0$, not at all times, although we will not require this fact.
\ermk

\cmt{I think we can choose $\si$ to be a universal constant times
$\frac{1}{K\ga^4}+\frac{1}{\ga^2}$, though I didn't check too carefully.}

\brmk
A slight modification of the proof would allow us to replace the conclusion of the lemma with the stronger assertion that $\Sc_{g(t)}\geq -K$ on $B_{g(t)}(x_0,1-\ga)$, provided we impose a lower bound on $K$, depending on $\ga$. We do not use this fact.
\ermk

\cmt{This remark needs checking, but here is the rough idea: we exploit the quadratic term in the evolution equation for $\Sc$.
Roughly we can allow $b$ to be negative, indeed to be $-k$.
Then at $(p_0,t_0)$ we get $(-\Sc)\leq He^{kt_0}(-\Sc)\leq (1+\de)K$.
Then in the last line of the proof, we use Young's inequality slightly differently to get 
\beqa0&\geq \frac23H\Sc^2 + 2V^2\Sc H^{1/2}+ bK\\
&\geq \frac13H\Sc^2-3 V^4 -k K\\
&\geq \frac13 e^{-kt_0} (1+\de)^2K^2-3 V^4 -k K\\
&>0,
\eeqa
if $K$ is large enough so the first term dominates.
}

\begin{proof}
Let $h$ be a cut-off function of the type defined in Lemma \ref{cutoff}, with $\ep =\frac{1}{4}$, $r_1 = (1-\ga)$
and $r_2  = 1$. 
In particular, there exist  $k=k(\ga)>0$, $\hat T=\hat T(\CO ,n,\ga)>0$ and  $V=V(\ga)>0$
so that $h$ is defined on $M\times [0,T)\cap [0,\hat T)$,
$h(x,t)\in [0,e^{-kt}]$ throughout,
and for all $t \in [0,\hat T)\cap[0,T)$, we have
$h(\cdot,t) = e^{-kt}$ on
$B_{g(t)}(x_0,1-\ga)$ and $h(\cdot,t) = 0$ outside $B_{g(t)}(x_0,1)$.
Moreover,  
we have
$( \partt -\lap_{g} ) h \leq 0$ and $|\grad h|^2 \leq V^2 h^{3/2}$ in the barrier sense as in that lemma.

For fixed, arbitrary $\de>0$,
define the function $f: M\times [0,T)\cap [0,\hat T)\to\R$ by
$$f(x,t)=h(x,t)\Sc(x,t)+(1+\de) Ke^{bt},$$
where $b:=3V^4/K$.
Clearly we have $f(\cdot,0)>0$ throughout $M$, and
$f(x,t)>0$ for each $t\in [0,T)\cap [0,\hat T)$ and $x$ outside
$B_{g(t)}(x_0,1)$.

{\bf Claim:} We have $f>0$ throughout $M\times [0,T)\cap [0,\hat T)$.

Assuming the claim for a moment, for each $t\in [0,T)\cap [0,\hat T)$ and throughout  $B_{g(t)}(x_0,1-\ga)$,
we have 
$e^{-kt}\Sc +(1+\de) Ke^{bt}>0$, i.e. 
$\Sc>-(1+\de) Ke^{(b+k)t}$,
and because $\de>0$ was arbitrary, this implies
$$\Sc\geq - Ke^{(b+k)t}.$$
By reducing $\hat T>0$ if necessary so that $e^{(b+k)\hat T}\leq 2$, this implies the lemma.

\emph{Proof of claim:}
Suppose contrary to the claim that there exists some first time $t_0\in (0,T)\cap (0,\hat T)$ at which there exists a point $p_0\in M$
where $f(p_0,t_0)=0$.
It is clear that $\Sc(p_0,t_0)<0$ and $h(p_0,t_0)>0$.
By Lemma \ref{cutoff}, there is a neighbourhood $\curlO$ of 
$(p_0,t_0)$ where a nonnegative barrier function $H\leq h$ is defined, 
with $H(p_0,t_0)=h(p_0,t_0)$, and both
$( \partt -\lap_{g} ) H \leq 0$ and $|\grad H|^2 \leq V^2 H^{3/2}$ at $(p_0,t_0)$.
By restricting $\curlO$ further, we may assume that $\Sc<0$
throughout $\curlO$.
Therefore, if we define
$$F(x,t)=H(x,t)\Sc(x,t)+(1+\de) Ke^{bt},$$
within $\curlO$, then we can be sure that $F\geq f\geq 0$ in 
$\curlO\cap (M\times [0,t_0])$,
with $F(p_0,t_0)=f(p_0,t_0)=0$. From this we can deduce that
at $(p_0,t_0)$ we have three facts. First,
\beq
\label{fact1}
0=\grad F=\grad(H\Sc)=H\grad \Sc + \Sc\grad H.
\eeq
Second, also using \eqref{fact1},
\beqa
\label{fact2}
0 \leq \lap F=\lap(H\Sc)
&=H\lap\Sc+\Sc\lap H+2g(\grad H,\grad\Sc)\\
&=H\lap\Sc+\Sc\lap H-2\frac{\Sc}{H}|\grad H|^2\\
&\leq H\lap\Sc+\Sc\lap H-2V^2\Sc H^{1/2}.
\eeqa
Third, 
\beqa
\label{fact3}
0\geq \pl{F}{t}&=\pl{}{t}(H\Sc)+b(1+\de) Ke^{bt_0}\\
&\geq H\pl{\Sc}{t}+\Sc\pl{H}{t} + bK.
\eeqa
Keeping in mind
the evolution equation
$$\pl{\Sc}{t}=\lap \Sc+2|\Ric|^2\geq \lap \Sc+ \frac{2}{3}\Sc^2,$$
(see e.g. \cite[Proposition 2.5.4 and Corollary 2.5.5]{RFnotes}) 
and the definition $b:=3V^4/K$, these three facts tell us that
\beqa
0&\geq H(\lap \Sc + \frac23\Sc^2) + \Sc\lap H + bK\\
&\geq \frac23H\Sc^2 + 2V^2\Sc H^{1/2}+ bK\\
&\geq -\frac32 V^4+ bK=\frac32 V^4\\
&>0,
\eeqa
which is a contradiction.
\end{proof}

\cmt{used nonnegativity of $H$ in last computation, although actually only at $(p_0,t_0)$, so it was not really necessary to go back and ask that $H\geq 0$ in the last section.}


\section{Local lower Ricci curvature bounds: The double bootstrap argument}
\label{db_sect}

In this section we show that local Ricci lower bounds persist for a short time, in the presence of $c_0/t$ upper Ricci control.

The argument is unorthodox in that we have to make two iterations in our estimates to control Ricci from below, initially only gaining a weaker $L^p$ control before obtaining effectively $L^\infty$ control on the second attempt.

\cmt{This time we need decay of the full curvature tensor, not just Ricci, because we need to apply Shi's estimates at some point}

\begin{lemma}[Local lower Ricci bounds]
\label{DB}
Let $\CO\geq 1$, $K >0$ and $\ga \in (0,1)$ be arbitrary. 
Suppose that $(M^3,g(t))$ is a Ricci flow for $t\in [0,T)$, and $x_0\in M$
satisfies $B_{g(t)}(x_0,1)\subset\subset M$ for all $t\in [0,T)$.
We assume further  that
\begin{compactenum}[(a)]
\item 
\label{hyp_a}
$\Ric_{g(0)}\geq -K \qquad\text{on }B_{g(0)}(x_0,1)$
\item
\label{hyp_b} 
$|\Rm|_{g(t)}\leq \frac{\CO }{t}\qquad\text{on }B_{g(t)}(x_0,1)$
for all $t\in (0,T)$.
\end{compactenum}
Then there exists a $\hat T = \hat T(\CO ,K,\ga)>0$ 
such that 
$$\Ric_{g(t)}\geq -100K\CO\qquad\text{on }B_{g(t)}(x_0,1-\ga)$$
for all $t\in [0,T)\intersect [0,\hat T)$.
\end{lemma}

In a different direction, it was proved in \cite{ChenXuZhang} that nonnegative Ricci curvature is preserved  for unbounded curvature complete smooth Ricci flows in three dimensions.


\newcommand{\rc}{\mathrm{Rc}}
\newcommand{\Lval}{L}
\newcommand{\twoval}{7}


\begin{proof}
By Hypothesis \eqref{hyp_a}, we know that $\Sc_{g(0)}\geq -3K$
on $B_{g(0)}(x_0,1)$, and so the 
scalar curvature estimate of Lemma \ref{Scalarestimate}
tells us that for $t\in [0,T)$ with
$t< \hat T$, for some $\hat T=\hat T(c_0,K,\ga)\in (0,1]$, 
we have
\beq
\label{Sc_bound}
\Sc_{g(t)}\geq -6K\qquad\text{ on }B_{g(t)}(x_0,1-\frac{\ga}{4}).
\eeq
On the other hand, Hypothesis \eqref{hyp_b} of the lemma implies that
\beq
\label{scalar_upper}
|\Sc_{g(t)}|\leq \frac{6\CO }{t}\qquad\text{on }B_{g(t)}(x_0,1)
\qquad\text{ for all }t\in (0,T).
\eeq
For $0<r_1 <r_2 \leq 1$, and with $\ep:=\frac{1}{10}$ fixed for the duration of the proof, we can reduce $\hat T>0$ and 
let $h: M \times [0,\hat T)\times [0,T) \to [0,1]$
be the corresponding Perelman cut-off function from Lemma \ref{cutoff}.
Note that $\hat T$ then depends on $r_1$ and $r_2$, but in practice, these values will be specific functions of $\ga$.
We define the $ ( 0 , 2)$ tensor $F$ by 
\beq
F(x,t):= h(x,t)\Ric(x,t) + \left[\Lval K\Sc(x,t)t^{\al}
+t^{\ep} + \twoval K\right]g(x,t), 
\eeq
for $\al\in [1/2,1]$ to be chosen depending on the context.
Just as for $\ep$, we fix the value of $L$, this time with $L=8$. Keeping these values as symbols helps show how their choice has affected the computations.
Notice that the  tensor  $F$ depends on $\al$, $r_1 $ and $r_2 $ in addition to $g(t)$ and $K$. There is no dependency on $\ep$ or $L$ since we have fixed their values. (The former affected the definition of $h$ as well as the $t^\ep$ term).

\cmt{why $7$ above? See \eqref{lalower} and the lines after \eqref{reaction}. It has to be a number larger than the lower bound multiple for scalar curvature, which is $6$.}

The proof requires a bootstrapping argument involving two steps (a
double bootstrap). First we show that after possibly reducing 
$\hat T$ to a smaller positive value (depending only on
$\CO$, $K$ and $\ga$) we have 
$$\Ric_{g(t)} \geq - t^{-3\ep}\qquad\text{  on }
B_{g(t)}(x_0,1-\frac{\ga}{2})$$ 
for all $t\in [0,\hat T)\cap [0,T)$.
This will be achieved by showing that
the tensor $F$, with $r_1  =1-\frac{\ga}{2}$,
$r_2  = 1-\frac{\ga}{4}$ and $\al = 1-2\ep$,
remains nonnegative definite on $B_{g(t)}(x_0,1-\frac{\ga}{4})$ for all 
$t\in [0,\hat T)\cap [0,T)$.
See {\bf Step One} below for details.

In {\bf Step Two} we use Step One to show that  the tensor $F$, with $r_1  = 1-\ga$, $r_2  = 1-\frac{\ga}{2}$ and  $\al$ now given by $\al = 1$, remains nonnegative definite on $B_{g(t)}(x_0,1-\frac{\ga}{2})$
for all $t\in [0,\hat T)\cap [0,T)$ after possibly reducing 
$\hat T$ again, still with the same dependencies.
This will then imply the desired estimate: See {\bf Step Two} below for details.

\cmt{subtlety in the paragraph below: we say ``$|\grad h| \leq V |h|^{1-\ep}\leq V h^\half$ (recall $h\leq 1$)'' but what we mean is an inequality for $H$. So it's really mainly relevant that $H\leq 1$, which luckily holds.}

Before performing Steps One and Two, we derive evolution equations and facts that are valid for $F$ in the case that 
$\al \in [\frac{1}{2},1]$, and 
$r_1$, $r_2$ are given by $r_1 = 1- \frac{\ga}{2}$, $r_2  = 1- \frac{\ga}{4}$ or $r_1  = 1- \ga$, $r_2  = 1- \frac{\ga}{2}$.
That is, the statements and estimates that we derive in the following  are valid in both Steps One and Two, though the $h$ etc. in each case will be different.
We can reduce $\hat T>0$ so that both cut-offs $h$, and hence both tensors $F$ are defined
for $t\in [0,\hat T]\cap [0,T)$, and we will always assume that $t$ is in that range, though we will reduce $\hat T>0$ further when necessary.
Both cut-offs satisfy
$( \partt -\lap_{g(t)} ) h \leq 0$ and $|\grad h| \leq V |h|^{1-\ep}\leq V h^\half$ (recall $h\leq 1$) in the barrier sense explained in Lemma \ref{cutoff},
where $V=V(\ga)$ and $k=k(\ga)$ can each be taken to be the maximum of the two values corresponding to the two cut-offs $h$.
(See Remark \ref{k_clarification}.)
By definition, $h(x,t) = e^{-kt}$ on
$B_{g(t)}(x_0,r_1 )$ and $h(x,t) = 0$ outside $B_{g(t)}(x_0,r_2 )$, 
($r_1  = 1-\frac{\ga}{2},$ $r_2  = 1- \frac{\ga}{4}$ or $r_1  = 1-\ga,$ $r_2  = 1- \frac{\ga}{2}$).

From the definition of $F$ and the initial conditions, we see that $F(\cdot,0)>0$ on $\overline{B_{g(0)}(x_0,r_2 )}.$
If there is a time $t_0\in [0,T)\cap [0,\hat T)$ for which $F(\cdot,t_0)>0$ on $\overline{B_{g(t_0)}(x_0,r_2 )}$ does {\bf not} hold, then there must be a first time $t_0$ for which this is the case. That is  there must be a  time
$t_0\in (0,T)\cap (0,\hat T)$ and a point $p_0 \in \overline{B_{g(t_0)}(x_0,r_2 )}$, and some direction $v \in T_{p_0} M$, for which $F(p_0,t_0)(v,v) =0$, 
$F(\cdot,t)>0$ on $\overline{B_{g(t)}(x_0,r_2 )}$ for all $t\in [0,t_0)$, and $F(\cdot,t_0) \geq 0$ on $\overline{B_{g(t_0)}(x_0,r_2 )}$.
Let $\la \leq \mu \leq  \nu$ denote the
eigenvalues of $\Ric_{g(t_0)}$, with corresponding 
eigenvectors $\{e_i\}_{i=1}^3$ at $p_0$ that are orthonormal with respect to $g(t_0)$.
That is,
$\Ric(e_1,e_1) = \la$, $\Ric(e_2,e_2) = \mu$, $\Ric(e_3,e_3) = \nu$ and
$\Ric(e_i,e_j) = 0$ for $i \neq j$, $i,j \in \{1,2,3\}$. 
Then $F(p_0,t_0)(e_1,e_1) = 0$ and $F(p_0,t_0)(e_i,e_j) =0$ for all $i \neq j \in \{1,2,3\}$.  Using $\Sc(\cdot,t) \geq -6K$ on  $\overline{B_{g(t)}(x_0,r_2 )}$, from \eqref{Sc_bound}, we see that we also have
\beq
\Lval K\Sc(x,t) t^{\al} +t^{\ep} +\twoval K \geq  -6\Lval K^2t^{\al} +  \twoval K >6K \label{littlescalar}
\eeq
everywhere on $\overline{B_{g(t)}(x_0,r_2 )}$ if $t< \hat T$ is sufficiently small, depending on $K$, 
(remember $\al\in [1/2,1]$ and $\Lval$ is fixed)
which we can without loss of generality assume by reducing  $\hat T$.
In particular this shows that $p_0 \in B_{g(t_0)}(x_0,r_2 )$: if it were not, then we would have $F(p_0,t_0) =
(\Lval K\Sc(p_0,t_0) t_0^{\al} +t_0^{\ep} +\twoval K)g(p_0,t_0) > 6Kg(p_0,t_0) > 0$, which is a contradiction.
(Note that normally in the paper we suppress the metric in expressions such as this.)
 Combining this with $F(p_0,t_0)(e_1,e_1) = 0$, we see that
$$\la \cdot h(p_0,t_0) = -(\Lval K\Sc(p_0,t_0) t_0^{\al} +t_0^{\ep} +\twoval K) < -6K,$$
and hence 
\beq
\label{lalower}
\la < -6K 
\eeq
since $h\leq 1$.

In principle, we would like to derive differential inequalities for $F$ near $(p_0,t_0)$, but because $h$ is not smooth, we take the barrier function $H\leq h$ from Lemma \ref{cutoff} corresponding to $h$ at $(p_0,t_0)$, which is defined in a space-time neighbourhood 
$\curlO$ of that point.
This allows us to define what is essentially a tensor barrier by
\beq
\f(x,t):= H(x,t)\Ric(x,t) + \left[\Lval K\Sc(x,t)t^{\al}
+t^{\ep} + \twoval K\right]g(x,t), 
\eeq
and its corresponding $(1,1)$ tensor
\beq
\ti\f(x,t):= H(x,t)\Rc(x,t) + \left[\Lval K\Sc(x,t)t^{\al}
+t^{\ep} + \twoval K\right],
\eeq
obtained by contracting the tensor $\f$ with the metric in the first position, where  $\rc$ denotes the $(1,1)$ tensor one  obtains by contracting the  tensor $\Ric$ in the first position.

\cmt{About 8 lines below, we restrict $\curlO$ further so that 
it doesn't hang outside the ball of radius $r_2$}

It will be convenient to extend the frame $\{e_i\}$ at $p_0$ to an adapted local orthonormal frame on a neighbourhood of $p_0$ within $(M,g(t_0))$. Thus we have $\grad e_i=0$ at time $t_0$. This frame can then be extended backwards and forwards in time, being  constant in time (so no longer orthonormal at different times).
Because $\Ric(e_1,e_1)=\la<0$ at $(p_0,t_0)$, we can restrict $\curlO$ to be sure that $R_{11}:=\Ric(e_1,e_1)<0$ throughout.
Working in index notation with respect to $\{e_i\}$, 
we then have 
$$\f_{11}\geq F_{11}\geq 0\qquad\text{ throughout }\curlO\cap\{t\leq t_0\}$$
possibly after restricting $\curlO$ further,
with equality throughout at $(p_0,t_0)$,
and consequently
$$\grad(\f_{11})=0,\qquad \lap(\f_{11})\geq 0,\qquad \pl{}{t}(\f_{11})\leq 0,$$
at $(p_0,t_0)$.
The first two of these facts immediately pass from statements about derivatives of coefficients of $\f$ to statements about coefficients of covariant derivatives of $\f$ and $\ti\f$.
For example, we have 
\beq
\label{tensor_conseq}
(\grad_X\ti\f)^1_{\ 1}= 0\text{ for any }X\in T_{p_0}M\qquad\text{and}\qquad (\lap\ti\f)^1_{\ 1}\geq 0.
\eeq
Statements concerning time derivatives need care because the metric is evolving. However, we have
\begin{equation*}
\begin{aligned}
(\pl{\ti\f}{t})^1_{\ 1}&=g(\pl{\ti\f}{t}(e_1),e_1)
=g(\pl{}{t}\ti\f(e_1),e_1)
=\pl{}{t}g(\ti\f(e_1),e_1)-\pl{g}{t}(\ti\f(e_1),e_1)\\
&=\pl{}{t}(\f_{11})\leq 0
\end{aligned}
\end{equation*}
at $(p_0,t_0)$ because $\ti\f(e_1)=0$ there.
We can compute the time derivative of $\ti\f$ using the standard evolution equation
$$\partt \rc = \lap \rc + P,$$
where the tensor $P$ is given in coordinates (see e.g. \cite[\S 9.3]{RFnotes}) by
\beq
\label{Pdef}
P^i_{\ j}=2R^i_{\ kjl}R^{kl},\quad\text{ and hence }\quad
P^1_{\ 1}=P_{11}=(\mu-\nu)^2+\la (\mu+\nu)\quad\text{ at }(p_0,t_0),
\eeq
and the evolution equation for the scalar curvature
$\partt \Sc = \lap \Sc + 2|\Ric|^2$ (see e.g. \cite[Proposition 2.5.4]{RFnotes}). 
At $(p_0,t_0)$ we have
\beqa
0 \geq{(\pl{\ti \f}{t})^1_{\ 1}} 
& = H \cdot {(\partt \rc)^1_{\ 1}} + \la\pl{H}{t}   \\
& \ \ \ + \Lval K \pl{\Sc}{t} t_0^{\al}
+ \Lval K\Sc\al t_0^{\al -1} + \ep t_0^{\ep-1}  \\
&  = H \cdot (\lap \rc + P)^1_{\ 1} + \la(\partt  - \lap )H + 
\rc^1_{\ 1}  \lap H  
\\
& \ \ \   +\Lval Kt_0^{\al}(\lap \Sc
+2|\Ric|^2) +  \Lval K\Sc\al t_0^{\al -1}  + \ep t_0^{\ep -1}  \\
 & = (\lap \ti\f)^1_{\ 1} -2g(\grad H, \grad \rc)^1_{\ 1}
 +(HP)^1_{\ 1} + \la (\partt  - \lap )H  \\
& \ \ \ +2\Lval Kt_0^{\al}|\Ric|^2 +
  \Lval K\Sc\al t_0^{\al -1}  +\ep t_0^{\ep -1}\cr 
  & \geq -2g(\grad H, \grad \rc)^1_{\ 1}
 +(HP)^1_{\ 1}  \cr
& \ \ \ +2\Lval Kt_0^{\al}|\Ric|^2 +
  \Lval K\Sc\al t_0^{\al -1}  +\ep t_0^{\ep -1}, 
  \label{evnF}
\eeqa
where we used the second part of \eqref{tensor_conseq}. Here 
$-2g(\grad H, \grad \rc)^i_{\ j} = 
-2g^{kl} \grad_k  H \cdot \grad_l {\Rc^i}_{j}$.

We now proceed to estimate the term $-2g(\grad H, \grad \rc)^1_{\ 1}(p_0,t_0)$ appearing in  Equation
\eqref{evnF}. If we had $H(p_0,t_0)=h(p_0,t_0) \leq C_3 t_0$, with $C_3 := \frac{K}{\CO } >0$, then by \eqref{littlescalar} 
and Hypothesis \eqref{hyp_b} of the lemma we would have
\beqa
0 &= h(p_0,t_0)\la +  \Lval K\Sc(p_0,t_0) t_0^{\al}
+ t_0^{\ep} + \twoval K\\
& \geq -h(p_0,t_0)\frac{2\CO }{t_0}  +6K\\
& \geq -2\CO C_3 + 6K\\
& > 0,
\eeqa 
which would be a contradiction. 
Hence, we have
\beqa
H(p_0,t_0) > C_3 t_0. \label{littlehest}
\eeqa
At $(p_0,t_0)$ we calculate 
\beqa
-2g(\grad H, \grad \rc)^1_{\ 1}
 & = -2g(\grad H , \frac{\grad
  (H\rc) }{H})^1_{\ 1} + 2 \la \frac{|\grad H |^2}{H}  \\
& = -2g(\grad H , \frac{\grad
  \ti \f }{H})^1_{\ 1}  + \frac{2\Lval Kt_0^{\al}}{H} g(\grad H, \grad \Sc) + 2 \la
\frac{|\grad H |^2}{H} \\
& = \frac{2\Lval Kt_0^{\al}}{H}g(\grad H, \grad \Sc) + 2 \la
\frac{|\grad H |^2}{H}\\
&\geq \frac{2\Lval Kt_0^{\al}}{H}g(\grad H, \grad \Sc) + 2V^2\la
\label{gradienterr}
\eeqa
where we used the definition of $\ti\f$, the first part of
\eqref{tensor_conseq} and the facts that 
$\la <0$ and 
$\frac{|\grad H |^2}{H}(p_0,t_0) \leq V^2$.

To estimate the first term on the right-hand side of \eqref{gradienterr} we use Shi's estimates to control 
$|\grad \Sc|$. 
Because $p_0 \in B_{g(t_0)}(x_0,r_2)$ and $r_2\leq 1-\frac{\ga}{4}$,
we have $p_0 \in B_{g(t_0)}(x_0,1-\frac{\ga}{4})$. 
Using Corollary \ref{shrinking_balls_cor} with $r = 1-\frac{\ga}{4} + \beta \sqrt{c_0 t_0}$, we see that
$p_0 \in  B_{g(t_0)}(x_0,1-\frac{\ga}{4})
=B_{g(t_0)}(x_0 , 
r
-\beta \sqrt{c_0 t_0})
\subset B_{g(t)}(x_0,r-\beta \sqrt{c_0 t})
\subset   B_{g(t)}(x_0, 1-\frac{\ga}{4} + \beta \sqrt{c_0 t_0}) \subset 
B_{g(t)}(x_0,1-\frac{\ga}{8})$ for all $t \in [0,t_0]$, if 
$\beta \sqrt{c_0 t_0} \leq \frac{\ga}{8}$, which we may assume is the case by reducing $\hat T$ if necessary.
Hence $B_{g(t)}(p_0,\frac{\ga}{16}) \subset  B_{g(t)}(x_0,1-\frac{\ga}{16})$ for all 
$t \in [0,t_0]$. 
This means that $|\Rm|_{g(t)} \leq \frac{c_0}{t}$ on 
$B_{g(t)}(p_0,\frac{\ga}{16})$ for all 
$t \in (0,t_0]$ and, in particular, that
$|\Rm|_{g(t)} \leq \frac{2c_0}{t_0}$ on $B_{g(t)}(p_0,\frac{\ga}{16})$ 
for all $t \in [(1-\si)t_0,t_0]$ for any $\si \leq \frac 1 2$; we choose $\si = \frac{\log 2}{4c_0}$.
Lemma \ref{balls_lemma2} with $K$ there equal to $\frac{4c_0}{t_0}$ and $r$ there equal to $\frac{\ga}{32}$ then tells us that  $B_{g((1-\si )t_0)}(p_0,\frac{\ga}{32}) \subset  B_{g(t)}(p_0,e^{(t_0\si)\frac{4c_0}{t_0} }\frac{\ga}{32}) 
= 
B_{g(t)}(p_0,\frac{\ga}{16}) $ for all $t \in [(1-\si)t_0, t_0],$ and hence
$|\Rm|_{g(t)} \leq \frac{2c_0}{t_0}$ on 
$B_{g((1-\si )t_0)}(p_0,\frac{\ga}{32})$ for all $t\in [(1-\si)t_0,t_0]$.
Shi's estimates (see for example Theorem 6.15 of \cite{CLN})
applied on 
$B_{g((1-\si )t_0)}(p_0,\frac{\ga}{64})$
now tell us that for universal $C$ we have
$|\grad \Rm(p_0,t_0)| 
\leq C \cdot \frac{c_0}{t_0}( \frac{1}{\ga^2} + \frac{1}{\si t_0} + \frac{c_0}{t_0})^{\frac{1}{2}}$, and hence
$|\grad \Sc(p_0,t_0)| \leq C_4(\ga,c_0)t_0^{-\frac{3}{2}}$.


Using this, and $H^{\ep} \geq C^{\ep}_3 t_0^{\ep}$ at $(p_0,t_0)$ by \eqref{littlehest},
we estimate the first term on the right-hand side of \eqref{gradienterr}
by
\beqa
|2\Lval Kt_0^{\al}  \frac{g(\grad H, \grad \Sc)|}{H}| & \leq 2\Lval Kt_0^\al
\Big(\frac{|\grad H|}{H} \Big) |\grad \Sc|\\
&  \leq  2\Lval Kt_0^\al \Big( \frac{V}{H^{\ep}}\Big)
\Big(\frac{C_4(\ga,\CO )}{t_0^{\frac 3 2}} \Big)\\
& = 2\Lval KVC_4 t_0^{\al - \frac{3}{2}} \frac{1 }{H^{\ep}} \cr
& \leq \frac{2\Lval K V C_4}{(C_3)^{\ep}}t_0^{\al -\frac{3}{2} - \ep } \\
& = C_5 t_0^{\al -\frac{3}{2} - \ep }, \label{ShiTypeEs}
\eeqa
where $C_5 = \frac{2\Lval KVC_4}{(C_3)^{\ep}}$ is a constant depending on $\ga,\CO ,K$ (as usual $L=8$ and $\ep =\frac{1}{10}$ are fixed).
Using the estimate \eqref{ShiTypeEs}
in the equation \eqref{gradienterr}, we obtain
\beqa
-2g(\grad H, \grad \rc)^1_{\ 1} (p_0,t_0)
& \geq -C_5 t_0^{\al -\frac{3}{2} - \ep }  + 2V^2\la.\label{gradienterrfinal}
\eeqa
We now estimate all the remaining terms in \eqref{evnF} (except $\ep t_0^{\ep -1}$), i.e. we consider 
$$Z:= (HP)^1_{\ 1}(p_0,t_0)  + 2\Lval Kt_0^{\al}|\Ric|^2(p_0,t_0) + \Lval K \Sc(p_0,t_0)\al t_0^{\al -1}.$$
Using the expression for $P$ in \eqref{Pdef},
we compute at $(p_0,t_0)$,
\beqa
Z& = (HP)^1_{\ 1}  + 2\Lval Kt_0^{\al}|\Ric|^2 + \Lval  K \al \Sc t_0^{\al -1} \cr
& = H[ (\mu-\nu)^2 +\la(\mu +\nu) ] + 2\Lval Kt_0^{\al}(\la^2 + \mu^2 +
\nu^2) +\Lval K\al \Sc t_0^{\al -1} \cr
& \geq H\la(\mu +\nu)  + 2\Lval Kt_0^{\al}(\la^2 + \mu^2 +
\nu^2) + \Lval K\al  \Sc t_0^{\al -1} \cr 
& = -(\Lval Kt_0^{\al}\Sc + t_0^{\ep} +\twoval K)(\mu +\nu) + 
 2\Lval Kt_0^{\al}(\la^2 + \mu^2 +
\nu^2)  + \Lval K\al \Sc t_0^{\al -1} \cr
& = \Lval Kt_0^{\al}[ -(\la + \mu +\nu)(\mu +\nu) +2 (\la^2 + \mu^2 +
\nu^2)] -(t_0^{\ep} + \twoval K)(\mu +\nu)\cr
& \ \ \  + \Lval K\al \Sc t_0^{\al -1} \cr
& = \Lval Kt_0^{\al}[ -\la(\mu +\nu) + (\mu-\nu)^2 + 2\la^2]
 -(t_0^{\ep} + \twoval K)(\mu +\nu)\cr
 & \ \ \  +\Lval K\al \Sc t_0^{\al -1}. \label{reaction}
\eeqa
Since  $\Sc(\cdot,t_0) \geq -6K $ on $B_{g(t_0)}(x_0,r_2 )$, by \eqref{Sc_bound}, 
and we are at a point where $\la < -6K$, by  \eqref{lalower}, 
we must have  $(\mu + \nu) \geq 0$ and hence the term $-\Lval Kt_0^{\al}\la(\mu + \nu)$ appearing above is nonnegative:
 $-\Lval Kt_0^{\al}\la(\mu + \nu) \geq 0$.
 The terms $-(t_0^{\ep} + \twoval K)(\mu +\nu) + \Lval K\al \Sc t_0^{\al -1}$ of \eqref{reaction} can be written
as
\beqa
\lefteqn{ -(t_0^{\ep} + \twoval K)(\mu
+\nu)  + \Lval K\al \Sc t_0^{\al -1}}\qquad\qquad& \\
& =  -(t_0^{\ep} + \twoval K)(\la + \mu
+\nu)  + (t_0^{\ep} + \twoval K)\la  + \Lval K\al \Sc t_0^{\al -1}\\
& = (-t_0^{\ep} -  \twoval K + \Lval  K\al t_0^{\al-1})\Sc + (t_0^{\ep} + \twoval K)\la\\
& \geq -6LK^2 t_0^{\al -1} + 8K\la 
\eeqa
after possibly reducing $\hat T>0$ (depending on $K$) so that
$0\leq -t_0^{\ep} -\twoval K +\Lval K\al t_0^{\al -1}\leq \Lval K\al t_0^{\al -1}\leq \Lval K t_0^{\al -1}$ and $t_0^\ep\leq K$,
where we used that $\Sc(p_0,t_0) \geq -6K$ and the facts that 
$\al\in [1/2,1]$ and $\Lval=8>7$ (required in the case that $\al\in [\frac89,1]$, say).
%
%

\cmt{
$\al \in [\frac{8}{9},1]$:
This way, it makes it completely clear that we can  reduce  $\hat T$ in a way that doesn't depend on $\al$ for $\al \in (1/2,1]$. It was originally written $\al=1$ because we are only applying the argument for two values of $\al$. But there is some confusion about whether $\al$ is being taken only to be one of these values, just as we only take specific values for $r_1$ and $r_2$. Let's keep to somewhat general $\al$.
}

Substituting these inequalities into  the inequality \eqref{reaction} gives us
\beqa
Z \geq 2\Lval Kt_0^{\al} \la^2 -6LK^2 t_0^{\al -1}  + 8K\la.  \label{Zestimate}
\eeqa
Putting the estimates \eqref{gradienterrfinal} and \eqref{Zestimate} into the equation \eqref{evnF}, we obtain
\beqa
0& \geq (\pl{\ti\f}{t})^1_{\ 1}(p_0,t_0)\cr
 & \geq -C_5 t_0^{\al -\frac{3}{2} - \ep }  + (2V^2 + 8K)\la + 2\Lval Kt_0^{\al}\la^2 
 -6L K^2 t_0^{\al -1}  + \ep t_0^{\ep -1}. \label{mainEvnF}
 \eeqa
In both Steps One and Two, we will use \eqref{mainEvnF} in order to obtain a lower bound on $t_0$, or equivalently to obtain a contradiction for $\hat T$ chosen small enough.
In both steps, this will tell us that $F (\cdot,t)\geq 0$ 
for any $t\in [0,T)\cap [0,\hat T)$ and throughout 
$B_{g(t)}(x_0,r_2)$.
In particular, for such $t$, by \eqref{scalar_upper} and the definition of $F$
we will have
\beq
\label{ricci_lower_reuse}
\Ric_{g(t)} \geq -6\Lval K e^{kt}\CO t^{\al-1} -e^{kt}t^{\ep} -7Ke^{kt} 
\eeq
within the smaller ball $B_{g(t)}(x_0,r_1)$ where we know that $h=e^{-kt}$.

\cmt{we say throughout $B_{g(t)}(x_0,r_2)$, not throughout $M$ above because we need the lower scalar bound to get this inequality, even where $h=0$.}

Now we perform Step One.\\
{\bf Step One}\\
Let $F$ be defined  as above with
$r_1 = 1-\frac{\ga}{2}$ and $r_2  = 1-\frac{\ga}{4}$ in the definition of $h$ (which appears in the definition of $F$) and  $\ep=\frac{1}{10}$, $L=8$,
(as always) and $\al = 1  - 2\ep$. 
In particular $h(x,t) = e^{-kt}$ on
$B_{g(t)}(x_0,1- \frac{\ga}{2})$, and $h(x,t) = 0$ outside $B_{g(t)}(x_0,1-\frac{\ga}{4})$, in view of the values of $r_1 ,r_2 $ we have chosen.
Let $t_0\in (0,\hat T)$  be a first time at which there is a point 
$p_0\in \overline{B_{g(t_0)}(x_0,r_2 )}$
where $F(p_0,t_0)>0$ fails to hold.
Using Young's inequality coarsely to estimate 
$(2V^2 + 8K)\la  \geq -\frac{(2V^2 +8K)^2}{K}t_0^{-\al} - K\la^2 t_0^{\al}$ in 
\eqref{mainEvnF}, 
and the fact that $2L=16\geq 1$,
we see that
\beqa
0& \geq (\partt \ti\f)^1_{\ 1}(p_0,t_0)\cr
 & \geq -C_5 t_0^{\al -\frac{3}{2} - \ep }  + (2V^2 + 8K)\la + 2\Lval Kt_0^{\al}\la^2 
 -6LK^2 t_0^{\al -1}  + \ep t_0^{\ep -1} \cr
 & \geq -C_5 t_0^{\al -\frac{3}{2} - \ep }  -\frac{(2V^2 +8K)^2}{K}t_0^{-\al}-6LK^2 t_0^{\al -1}  + \ep t_0^{\ep -1} \cr
 &= -C_5 t_0^{ -\frac{1}{2} - 3\ep }  -\frac{(2V^2 +8K)^2}{K}t_0^{-1 + 2\ep}-6LK^2 t_0^{-2\ep}  + \ep t_0^{\ep -1} \cr
 & >0,
 \eeqa
if $t_0 < \hat T(\ga,\CO ,K)$ is small enough: the dominating term is $\ep t_0^{\ep -1}$ because we took $\ep=\frac{1}{10}$.
This is a contradiction.

\cmt{This is the point that our choice $\ep=1/10$ is constrained. Also the point where it helps to have $\al=1-2\ep$ rather than $\al=1-\ep$!}

We have shown $F(\cdot,t) \geq 0$ on $B_{g(t)}(x_0,1-\frac{\ga}{4})$ for all $t\in [0,T)\cap [0,\hat T)$,
and in particular, by \eqref{ricci_lower_reuse} with our choice of $\al$, we have
\beqa
\Ric_{g(t)} &\geq -6\Lval K e^{kt}\CO t^{-2\ep} -e^{kt}t^{\ep} -7Ke^{kt} \geq -7\Lval K\CO t^{-2\ep}\\
&\geq -t^{-3\ep}, \label{stepone}
\eeqa 
on $B_{g(t)}(x_0,1-\frac{\ga}{2})$,
where we have reduced $\hat T$ again if necessary, and used that $k = k(\ga)$.
%
%
This is the first step in the double bootstrap argument.
{\bf End of Step One}

{\bf Step Two}\\
For the new definition of $F$, we choose
$r_1 = 1-\gamma$ and $r_2  = 1-\frac{\ga}{2}$ in the definition of $h$ (which appears in the definition $F$) $\ep = \frac{1}{10}$, $L=8$ as always,  and set now $\al = 1$.
In particular $h(x,t) = e^{-kt}$ on
$B_{g(t)}(x_0,1- \ga)$, and $h(x,t) = 0$ outside $B_{g(t)}(x_0,1-\frac{\ga}{2})$, for all 
$t\in [0,T)\cap [0,\hat T)$,
in view of the values of $r_1 ,r_2 $ we have chosen. 
Let $t_0\in (0,\hat T)$  be a first time at which there is a point 
$p_0\in \overline{B_{g(t_0)}(x_0,r_2 )}$
where $F(p_0,t_0)>0$ fails to hold.
Using  $\Ric(x,t) \geq -t^{-3\ep}$ for $x \in B_{g(t)}(x_0,1-\frac{\ga}{2})$, the result of Step One, in 
\eqref{mainEvnF}, we see that
\beqa
0& \geq (\partt \ti\f)^1_{\ 1}(p_0,t_0)\cr
 & \geq -C_5 t_0^{\al -\frac{3}{2} - \ep }  + (2V^2 + 8K)\la + 2\Lval Kt_0^{\al}\la^2 
 -6\Lval K^2 t_0^{\al -1}  + \ep t_0^{\ep -1} \cr
 &\geq 
-C_5 t_0^{-\frac{1}{2} - \ep }   - (2V^2 + 8K)
 t_0^{-3\ep} + 2\Lval Kt_0\la^2 
 -6\Lval K^2   + \ep t_0^{\ep -1} \cr
 & >0
 \eeqa
if $t_0 < \hat T(\ga,\CO ,K)$ is small enough, since then the dominating term is $\ep t_0^{\ep -1}.$
This is a contradiction.

We have shown $F(\cdot,t) \geq 0$ on $B_{g(t)}(x_0,1-\frac{\ga}{2})$ for all $t\in [0,T)\cap [0,\hat T)$,
and in particular, by \eqref{ricci_lower_reuse} with our new choice of $\al$, we have
\beq
\Ric_{g(t)} \geq -6\Lval K e^{kt}\CO  -e^{kt}t^{\ep} -7Ke^{kt} 
\geq -100K\CO\label{steptwo}
\eeq 
on $B_{g(t)}(x_0,1-\ga)$, after possibly reducing $\hat T>0$ again (without new dependencies) since $\CO\geq 1$ and $L=8$.
{\bf End of Step Two and the proof}.
\end{proof}



\section{{Proof of the main theorem \ref{mainthm}}}
\label{mainthm_proof_sect}

\cmt{\emph{Peter writes:}I DISAGREED with what you wrote here, so changed the proof as in the next paragraph: 
\emph{Miles writes:}
I went through the proof of Theorem 1.1 implies 1.5 and visa versa.
It seems like no change in the proof is necessary, if we change  $\ti v_0$ in the statement of the Theorems only to depend on $v_0,K$. I am strongly for this.}

Observe that without loss of generality, we may always assume that $\si\in (0,3]$ in Theorems \ref{mainthm} and \ref{mainthm_time0}.
Moreover, we observe that these theorems are equivalent to the corresponding results in which 
$\ti v_0$ is allowed also to depend on $\si$, as we now clarify.
If we have established either of these theorems allowing this extra dependency, then we may take the Ricci flow $g(t)$, parabolically scale up so that the ball of radius $1$ becomes of radius $2$, then apply the result to the rescaled flow with $\si=1$. Scaling back, we obtain a positive lower bound on $\VolB_{g(t)}(x_0,1/2)$, which implies Conclusion \ref{volume_conseq} (or Conclusion \ref{volume_conseq_prime}$'$).

Next we address the equivalence of Theorems \ref{mainthm} and \ref{mainthm_time0}, allowing this extra $\si$ dependency of $\ti v_0$. In each direction, we combine what curvature control we have to relate time $t$ balls to time $0$ balls using the results of Section \ref{balls_sect}.

{\bf Theorem \ref{mainthm} implies Theorem \ref{mainthm_time0}.}

Given a Ricci flow as in the theorems, 
Conclusions \ref{volume_conseq} and \ref{Ric_conseq_main} coupled with Bishop-Gromov imply a positive lower bound for the volume of $B_{g(t)}(x_0,1/2)$, say.
The shrinking balls corollary \ref{shrinking_balls_cor} tells us that
$B_{g(t)}(x_0,1/2)\subset B_{g(0)}(x_0,1)$ for a definite period of time, and hence we obtain the positive lower bound for the volume of $B_{g(0)}(x_0,1)$ required for Conclusion \ref{volume_conseq_prime}$'$
(for some different $\tilde v_0$ with the correct dependencies).

To obtain Conclusions \ref{Ric_conseq_main_prime}$'$ and \ref{conc3decay_prime}$'$,
we parabolically rescale down our Ricci flow slightly so that balls of radius $1+\si/2$ become balls of radius $1$, apply Theorem \ref{mainthm} (with the correspondingly smaller $\si$), rescale back, and obtain  Conclusions \ref{Ric_conseq_main}
and \ref{conc3decay} on the larger ball $B_{g(t)}(x_0,1+\si/2)$ (for correspondingly smaller time).
At this point, we can apply the expanding balls lemma \ref{balls_lemma2} to deduce that $B_{g(t)}(x_0,1+\si/2)\supset
B_{g(0)}(x_0,1)$ for a definite time, and so Conclusions 
\ref{Ric_conseq_main_prime}$'$ and \ref{conc3decay_prime}$'$
will hold.\qed

{\bf Theorem \ref{mainthm_time0} implies Theorem \ref{mainthm}.}

As above, we can apply the theorem we know on a slightly (parabolically) scaled down Ricci flow, in order to obtain Conclusions 
\ref{Ric_conseq_main_prime}$'$ and \ref{conc3decay_prime}$'$
on the ball
$B_{g(0)}(x_0,1+\si/2)$.
By the shrinking balls corollary \ref{shrinking_balls_cor}
we know that $B_{g(0)}(x_0,1+\si/2)\supset B_{g(t)}(x_0,1)$
for a short time, which gives us Conclusions 
\ref{Ric_conseq_main} and \ref{conc3decay}.

To obtain Conclusion \ref{volume_conseq}, we begin by applying Bishop-Gromov to the initial metric to deduce that $\VolB_{g(0)}(x_0,1/2)$ has a positive lower bound (depending only on $v_0$ and $K$). If we parabolically scale up our Ricci flow so that this ball of radius $1/2$ becomes of radius $1$, we can apply Theorem \ref{mainthm_time0} (in particular Conclusion 
\ref{volume_conseq_prime}$'$), and scale back down again, to deduce that $\Vol_{g(t)}(B_{g(0)}(x_0,1/2))$ has a positive lower bound for a definite time.
But then by the expanding balls lemma \ref{balls_lemma2} (coupled with Conclusion \ref{Ric_conseq_main} that we have just proved) we know that
$B_{g(0)}(x_0,1/2)\subset B_{g(t)}(x_0,1)$ over some uniform time interval, which implies then a positive lower bound for 
$\VolB_{g(t)}(x_0,1)$ as required.\qed

Having proved the equivalence of these theorems, we demonstrate that it suffices to prove Theorem \ref{mainthm} in the case $\si=3$ (i.e. the Ricci lower bound of \eqref{Ric_hyp_main} holds on the ball of radius $4$). By the equivalence we have shown above, we would like to deduce Theorem \ref{mainthm_time0}.
Confronted with a Ricci flow for which we only have the Ricci hypothesis on a smaller ball $B_{g(0)}(x_0,1+\si)$, pick an arbitrary point $z_0\in B_{g(0)}(x_0,1)$ at which we would like to establish Conclusions 
\ref{Ric_conseq_main_prime}$'$ and \ref{conc3decay_prime}$'$.
By scaling up our Ricci flow parabolically so that the ball 
of radius $\si$ centred at $z_0$ becomes of radius $4$, we can apply the $\si=3$ case of the theorem (with $x_0$ there equal to $z_0$ here). In particular, Conclusions \ref{volume_conseq} and \ref{Ric_conseq_main} restricted to the centre $z_0$ give a lower Ricci bound and upper sectional bound that when we return to the unscaled Ricci flow give Conclusions 
\ref{Ric_conseq_main_prime}$'$ and \ref{conc3decay_prime}$'$
at the arbitrary point $z_0$.

To obtain Conclusion \ref{volume_conseq_prime}$'$, we can parabolically scale up the Ricci flow so that the ball $B_{g(0)}(x_0,1)$ ends up with radius $4$, and then apply the result to give, after scaling back, a positive lower bound on the volume $\VolB_{g(t)}(x_0,1/4)$.
By the shrinking balls corollary \ref{shrinking_balls_cor} (together with Conclusion 
\ref{conc3decay_prime}$'$) 
we have $B_{g(t)}(x_0,1/4)\subset B_{g(0)}(x_0,1)$ for a definite time interval, and hence we have the desired positive lower bound for $\Vol_{g(t)}(B_{g(0)}(x_0,1))$.

We are thus reduced to proving Theorem \ref{mainthm} in the case that $\si=3$, and the remainder of this section is devoted to that task.


\cmt{Concerning claim below: Note that if $\VolB_{g(0)}(y_0,r)$ is much smaller, then the volume of an even smaller ball that is a bit more centred is also small, so then BG shows the volume of a bigger ball with the same centre is small.... keep going.}

\cmt{note that we'll need volume control to apply pseudo improve lemma. So we can now apply that only for centres within $B_{g(0)}(x_0,3)$, and radii no more than $1$.}

By Bishop-Gromov, by reducing $v_0$ it is sufficient to prove the theorem with hypothesis \eqref{volume_hyp_main} replaced by the apparently more restrictive hypothesis that 
\begin{equation} 
\label{volmin}
\VolB_{g(0)}(y_0,r)\geq v_0 r^3\qquad\text{ for all }y_0\in B_{g(0)}(x_0,3)\text{ and }r\in (0,1]. 
\end{equation}
Equipped with this new (lower) $v_0$ and the $K$ from the theorem, we can appeal to Lemma 
\ref{pseudo_improve_lemma3} for constants $\hat T$, $C_0$ and $\al$, in preparation for the application of this lemma later; the lemma will be applied not to $g(t)$, but a
scaled-up version of $g(t)$. 
The constant $C_0$ will describe the rate of curvature decay we can expect in certain situations. We set $c_0=C_0+1$. 
In particular, it is significant that
both $c_0>1$ and $c_0>C_0$.
This $c_0$ and $\al$ are now fixed for the rest of the argument and this $c_0$ is the $c_0$ for which we will show that Conclusion \ref{conc3decay} is valid.

In the proof below, we will need to apply the double 
bootstrap lemma \ref{DB_conseq}, once with
the lower Ricci curvature bound given by $\ti K=K$ and once with $\ti K=K/(100c_0)$. In both cases we assume the condition $|\Rm| \leq \frac{c_0}{t}$ for the $c_0$ that we have just defined.
If necessary, we will reduce the $\hat T$ we found above, so that the conclusions of 
Lemma \ref{DB_conseq} in each case will now also be valid for $t\leq \hat T$.

Equipped with this $c_0$ and $\al$, we are going to apply Lemma \ref{decay_or_earlier_control} to the Ricci flow of the main theorem, with $r_0=3$ and the same $x_0$. The constant $L$ in the statement of Lemma \ref{decay_or_earlier_control} will be chosen sufficiently large to guarantee two things: First, the latest time that is considered in Lemma \ref{decay_or_earlier_control}, ie.
$\frac{r_0^2}{\be^2c_0(L+1)^2}=\frac{9}{\be^2c_0(L+1)^2}$, should be  bounded  
by the $\hat T$ above -- in fact, we ask that it is less than $\hat T/(100c_0)$.
Ultimately, we will be taking certain $t_0\in [0,\frac{9}{\be^2c_0(L+1)^2}]$ and scaling the Ricci flow so that time $t_0$ becomes time $\hat T$. Under this scaling, we will be stretching time by a factor of at least $100c_0$, and the lower Ricci bound hypothesis
$\Ric_{g(0)}\geq -K$ will become $\Ric\geq -K/(100c_0)$.

The second constraint on $L$ is that we require  that the radius  
$L(1-\al)\be\sqrt{c_0 t_0}$ that is considered in case 2 of Lemma \ref{decay_or_earlier_control} is large enough. After scaling up time by a factor $\hat T/t_0$ as above,
this radius becomes $L(1-\al)\be\sqrt{c_0 \hat T}$, and we ask that this is at least $2$ so that we will be able to  apply Lemma \ref{DB_conseq}. We will be doing this with $\ti K=K/(100c_0)$,
so the output will be a Ricci lower bound $\Ric\geq -K$ on the time $t$ unit ball centred at $x_0$, for the rescaled flow, as we will see momentarily.

\cmt{$L$ depends on $v_0$, $K$. As will $\tilde T$ below.}


With this choice of $L$ and $r_0=3$,
Lemma \ref{decay_or_earlier_control} gives us two cases. 
We now show that case 1 implies the theorem, while case 2 leads to a contradiction.

{\bf Case 1:}
In this case, we have curvature decay
$|\Rm|_{g(t)}<\frac{c_0}{t}$ throughout $B_{g(t)}(x_0,3-(L+1)\be\sqrt{c_0 t})\subset B_{g(0)}(x_0,3)$ for each $t\in (0,T)$ with $t<\frac{9}{\be^2c_0(L+1)^2}$.
We choose $\ti T>0$ small enough so the ball on which we have curvature control always 
has radius at least $2$, i.e. we choose $\ti T=1/(\be^2(L+1)^2c_0)$.
Thus we have 
$|\Rm|_{g(t)}<\frac{c_0}{t}$ throughout $B_{g(t)}(x_0,2)$ for each $t\in (0,T)$ with 
$t<\ti T$, which implies Conclusion \ref{conc3decay} of the theorem.

Equipped with this curvature decay, we can apply Lemma \ref{DB_conseq} with $\ti K=K$
and $T$ there equal to $\ti T$ here.
The conclusion is a Ricci lower bound 
$\Ric_{g(t)}\geq -100Kc_0$ on $B_{g(t)}(x_0,1)$
for all $t\in [0,T)\intersect [0,\ti T)$.
(Recall that $\ti T\leq \hat T$ by our choice of $L$.)
In particular, we deduce Conclusion \ref{Ric_conseq_main}
of the theorem.

Finally, we can apply Lemma \ref{volume_control_lem} with $\ga=1$ and with $K$ there equal to $100Kc_0$ here, to give a later lower volume bound, which implies Conclusion \ref{volume_conseq}, after possibly reducing $\ti T>0$ so that $\ti T \leq \hat T$ of  Lemma \ref{volume_control_lem}.

{\bf Case 2:}
In this case, we know that there exist $t_0\in (0,T)$ with $t_0<\frac{9}{\be^2c_0(L+1)^2}$ and 
$z_0\in \overline{B_{g(t_0)}(x_0,3-(L+1)\be\sqrt{c_0 t_0})} \subset \overline{B_{g(0)}(x_0,3)}$ 
such that 
$$Q:=|\Rm|_{g(t_0)}(z_0)=\frac{c_0}{t_0},$$
and for all 
$t\in (0,\al^2 t_0]$, we have
$|\Rm|_{g(t)}<c_0/t$ throughout $B_{g(t)}(z_0,L(1-\al)\be\sqrt{c_0 t_0})$, where we are using the $\al \in (0,1)$ given earlier.
If we rescale $g(t)$ by stretching time by a factor $\hat T/t_0$, as discussed earlier,
i.e we define 
$$\ti g(t)=\frac{\hat T}{t_0}g\left(\frac{t_0}{\hat T}t\right),$$
then we end up with a Ricci flow $\ti g(t)$ defined for $t\in [0,\hat T]$ with the properties 
that 
$|\Rm|_{\ti g(\hat T)}(z_0)=\frac{c_0}{\hat T}$, 
but so that 
$|\Rm|_{\ti g(t)}<c_0/t$ for all $t\in (0,\al^2 \hat T]$ and 
throughout $B_{\ti g(t)}(z_0,L(1-\al)\be\sqrt{c_0 \hat T})$, and hence
throughout $B_{\ti g(t)}(z_0,2)$ by our choice of $L$.

Using the facts that $z_0 \in \overline{B_{g(0)}(x_0,3)}$ and  $B_{g(0)}(z_0,1) \subset B_{g(0)}(x_0,4)$,  we see that the initial conditions  \eqref{volmin} and $\Ric_{g(0)} \geq -K$  also hold at $y_0 = z_0$ respectively on $B_{g(0)}(z_0,1)$.
As discussed when choosing $L$, $t_0$ is sufficiently small so that the stretching factor 
$\hat T/t_0$ is larger than $100c_0 >4$, and hence 
the scaled initial conditions  $\VolB_{\ti g(0)}(z_0,1) \geq v_0$  and $\Ric_{\ti g(0)} \geq -\frac{K}{100c_0}$  on $B_{\ti g(0)}(z_0,2)$ hold.
Using 
this latter condition 
and  that $|\Rm|_{\ti g(t)}<c_0/t$ for all $t\in (0,\al^2 \hat T]$
throughout $B_{\ti g(t)}(z_0,2)$, we see that we 
may apply Lemma \ref{DB_conseq} with $g(t)=\ti g(t)$, $x_0=z_0$, $\ti K=K/(100c_0)$ and $T=\al^2\hat T$ to deduce that
$\Ric_{\ti g(t)}\geq -K$ on $B_{\ti g(t)}(z_0,1)$ for all $t\in [0,\al^2\hat T)$.

\cmt{
Note that if we were to unwind this scaling, then this lower Ricci bound could be horrible if the original $t_0$ were very small, but we don't care.
}

Using this lower Ricci bound  and $\VolB_{\ti g(0)}(z_0,1) \geq v_0$,  we see that we 
may apply the pseudolocality improvement lemma \ref{pseudo_improve_lemma3} to $\ti g(t)$. The output is that 
for all $t\in (0,\hat T]$ we have
$$|\Rm|_{\ti g(t)}(z_0)\leq \frac{C_0}{t},$$ 
and in particular for the original unscaled Ricci flow $g(t)$, at time $t_0$ and at the point $z_0$ 
we have
$$|\Rm|_{g(t_0)}(z_0)\leq \frac{C_0}{t_0}.$$
But this contradicts the assertion of case 2 that 
$$|\Rm|_{g(t_0)}(z_0)= \frac{c_0}{t_0},$$
because $c_0=C_0+1$.

\section{A Ricci flow  version of non-collapsing -- proof of Lemma \ref{volume_control_lem}}
\label{vol_control_sect}

Our task in this section is to restrict the drop in volume of a ball of fixed radius under Ricci flow, under the curvature hypotheses
\eqref{ric_bdi} and \eqref{curv_bdii} of Lemma \ref{volume_control_lem}. 
There is a simple formula for the evolution of the volume of a fixed subset $\Om\subset M$ under Ricci flow
$$\frac{d}{dt}\Vol_{g(t)}(\Om)=-\int_\Om \Sc\,d\mu_{g(t)},$$
and so the decay of volume is controlled by a sufficiently strong upper curvature bound. However, the natural upper curvature bound is $c_0/t$ as in the lemma, and this just fails to be integrable in time, thus permitting in theory an arbitrarily quick loss of all volume. 
Moreover, volume can be lost by the stretching of distance, sending it outside the ball of fixed radius rather than shedding it completely.
The proof must therefore be  more subtle. The curvature hypotheses we are allowed here will instead be used initially to uniformly control the evolution of the Riemannian distance function as time lifts off from zero, as described in Lemma \ref{DCL3}. 
One could deduce uniform estimates for the Gromov-Hausdorff distance between balls with respect to $g(0)$ and balls with respect to $g(t)$ and then revisit Cheeger-Colding theory to verify that volume convergence will apply in this slightly different context of incomplete Riemannian manifolds
and yield uniform volume bounds, depending on time, via a contradiction argument. 
Instead, we give a direct, self-contained proof that uses the curvature hypotheses in additional ways.

\cmt{could give refs, or not} \cmt{Of course, the argument above is not that $g(t)$ converges to $g(0)$ in a `Gromov-Hausdorff' sense and so the volume converges, but rather that the extended Cheeger Colding theory would say that because of the Ricci lower bound, we can make the volume as close as we like by asking no more than closeness in GH, which is implied by time being sufficiently small.} 

The central claim on the way to proving Lemma \ref{volume_control_lem} is the following.

\begin{lemma}
\label{root_t_vol_lemma}
Suppose that $(M^n ,g(t))$ is  a  Ricci flow for ${t\in [0,\twasS )}$,  
such that for some $x_0\in M$ and all $t \in [0,\twasS )$ we have ${B_{g(t)}(x_0,10)\subset\subset M}$.
Assume further that 
\begin{enumerate}[(i)]
\item
\label{ric_bdi2}
$\Ric_{g(t)} \geq -1$ on $B_{g(t)}(x_0,10)$ for all $t \in
  [0,\twasS )$,
\item
\label{curv_bdii2}
$|\Rm|_{g(t)} \leq \frac{c_0}{t} $    on
$ B_{g(t)}(x_0,10)$ for some $c_0<\infty$ and all $t \in
  (0,\twasS )$,
\item
\label{vol_bdiii2}
$\VolB_{g(0)}(x_0,1) \geq v_0$ for some $v_0>0$.
\end{enumerate}
Then there exist $\ti \ep_0>0$ depending only on $v_0$ and $n$, and,
for each $A \geq 1$, a time $\ti  T\in (0,A^{-2}]$ depending only on $c_0,A,n$ and $v_0$
such that for all $t \in [0,\ti T ] \cap[0,\twasS )$ and all $y_0 \in B_{g(t)}(x_0, 1)$,
we have 
\beq
\label{vol_lower_conc}
\VolB_{g(t)}(y_0, A\sqrt t)  \geq \ti \ep_0(A\sqrt t)^n .
\eeq
\end{lemma}

\cmt{I moved the introduction of $A$ later in the lemma, which looks a little odd, but allows us more accurately later to say ``we assume the hypotheses of the lemma'' without needing to talk about $A$.}

\cmt{imposed $\ti T\leq A^{-2}$ to be sure that $B_{g(t)}(y_0, A\sqrt t)\subset B_{g(t)}(x_0, 10)$}

\cmt{a glance at the proof suggests that we can actually take $\ti\ep_0$ to be a small universal constant times $v_0^{n-1}$. That seems very weird, if correct}

Before proving Lemma \ref{root_t_vol_lemma}, we use it to prove Lemma \ref{volume_control_lem}.

\cmt{Why radius $10$ in the lemma? Lemma \ref{DCL3} requires 5 times the radius of the ball on which we're working, which is currently 2 -- see below.}

\begin{proof}[Proof of Lemma \ref{volume_control_lem}]
Let $r>0$ be the maximum of $10/\ga$ and $\sqrt{K}$.
By Bishop-Gromov, hypotheses \eqref{vol_bdiii} and \eqref{ric_bdi} imply that 
$\VolB_{g(0)}(x_0,1/r)$ has a positive lower bound depending only on $v_0$, $K$, $n$ and $\ga$.
Therefore, making a parabolic rescaling that stretches lengths by a factor of $r$, we see that in fact we may assume 
that 
the hypotheses of Lemma \ref{root_t_vol_lemma} hold, and it then suffices to prove a lower bound  $\VolB_{g(t)}(x_0,1)>\ep_0=\ep(v_0,n)>0$ for a time $\thatwasT =\thatwasT (v_0,c_0,n)>0$.

We will prove that this equivalent version of Lemma \ref{volume_control_lem} is true with 
\beq
\label{ep0choice}
\ep_0=\half\left(\frac{\tilde \ep_0 v_0}{5^n\Om_n }\right),
\eeq
provided we take $\thatwasT $ small enough, where $\ti\ep_0$ is from Lemma \ref{root_t_vol_lemma}, 
and $\Om_n $ is the volume of a ball of radius $1$ in the $n$-dimensional model space with $\Ric\equiv-1$.

\cmt{The $\thatwasT \leq \frac{1}{25 R^2}$ constraint below is used in several ways, e.g. to make certain statements nonvacuous, to make geodesic from $y$ to $x_0$ long enough when that is considered below....}

The constraints on $\thatwasT $ are as follows. To begin with, because we wish to apply Lemma \ref{root_t_vol_lemma}, we ask that $\thatwasT $ is no greater than $\tilde T$ from that lemma, 
where $A$ there is taken to be $R:=\be\sqrt{c_0}$, with $\be$ given by Lemma \ref{shrinking_balls_lemma}.
Next, if necessary we reduce $\thatwasT $ further so that $\thatwasT \leq \frac{1}{25 R^2}$, and 
\beq
\label{weird_restriction}
1+R\sqrt{t}\geq e^t \quad\text{ for all }t\in [0,\thatwasT ].
\eeq
A first consequence of $\thatwasT \leq \frac{1}{25 R^2}$ is that 
Corollary \ref{shrinking_balls_cor} tells us that
\beq
\label{1in2}
B_{g(t)}(x_0,1)\subset B_{g(0)}(x_0,2)\text{ for all }t\in [0,\thatwasT ]\cap[0,\twasS ),
\eeq
while Lemma \ref{balls_lemma2} tells us that
\beq
\label{EBLapp}
B_{g(0)}(x_0,1)\subset B_{g(t)}(x_0,e^t)\subset B_{g(t)}(x_0,1+R\sqrt{t})
\text{ for all }t\in [0,\thatwasT ]\cap[0,\twasS ),
\eeq
thanks to \eqref{weird_restriction}.
Finally, by reducing $\thatwasT $ still further, part 3 of Lemma \ref{DCL3} will apply with $r=2$ and $K=1$ to tell us that for any $x\in B_{g(0)}(x_0,2)$, and $s\in [0,6]$, we have
$B_{g(0)}(x,s)\supset B_{g(t)}\left(\textstyle{x,s-R\sqrt{t}}\right)$. By \eqref{1in2}, this implies
\beq
\label{nestedRsqrttballs}
B_{g(0)}(x,5R\sqrt{t})\supset B_{g(t)}(x,4R\sqrt{t})\qquad
\quad\text{for all}\quad x\in B_{g(t)}(x_0,1),
\eeq
and all $t\in [0,\thatwasT ]\cap[0,\twasS )$.

\cmt{we need the full radius 10 control below}

Suppose our equivalent version of Lemma \ref{volume_control_lem} is false with the choice of $\ep_0$ in \eqref{ep0choice} and with $\thatwasT $ satisfying the constraints above. Then we can fix
$t \in [0,\thatwasT ]\cap[0,\twasS )$ so that $\VolB_{g(t)}(x_0,1)  = \ep_0$.
This choice of $t$  determines a scale for a covering of $B_{g(t)}(x_0,1)$ by balls as we now explain.
Let $N$ be the largest number of 
disjoint balls $B_{g(t)}(p_i,R\sqrt{t})\subset B_{g(t)}(x_0,1)$ that can be chosen.
Note that 
we can be sure that at least one ball exists, i.e. that $N\geq 1$, because we are assuming that $t\leq \frac{1}{25R^2}<\frac{1}{R^2}$. 
Moreover, $N$ is finite because by assumption we know that $B_{g(t)}(x_0,1)\subset\subset M$, so an infinite sequence $p_i$ would have to have a convergent subsequence.
In the following we will estimate $N$ from above explicitly.
Fix now such a choice of points/centres $p_1,\ldots,p_N$, and the corresponding disjoint balls.
Using the disjointness, and 
Lemma \ref{root_t_vol_lemma} we see that
$\VolB_{g(t)}(p_i,R\sqrt{t})\geq \tilde \ep_0(R\sqrt{t})^n $, 
and hence
\beq
\label{Nupperbd}
\ep_0=\VolB_{g(t)}(x_0,1)  \geq 
\Vol_{g(t)}( \cup_{i=1}^{N}   B_{g(t)} (p_i, R\sqrt{t})   ) 
\geq N \tilde\ep_0 (R\sqrt{t})^n .
\eeq

Whereas we have $B_{g(t)}(p_i,R\sqrt{t})\subset B_{g(t)}(x_0,1)$, 
we can be sure that the larger balls $B_{g(t)}(p_i,3R\sqrt{t})$ \emph{cover} $B_{g(t)}(x_0,1)$
because if there existed a point $y\in B_{g(t)}(x_0,1)$ outside all $N$ of these larger balls, then we could contradict the definition of $N$ as follows.
The ball $B_{g(t)}(y,2R\sqrt{t})$ would be disjoint from all the original balls 
$B_{g(t)}(p_i,R\sqrt{t})$.
If $y$ were  within $B_{g(t)}(x_0,1-R\sqrt{t})$, 
then we could add $y$ to our original list of points $p_1,\ldots,p_N$, i.e. add $B_{g(t)}(y,R\sqrt{t})$ to our original collection of balls, thus contradicting the choice of $N$ immediately. If instead $y\notin B_{g(t)}(x_0,1-R\sqrt{t})$ then this new ball might fall partly outside of $B_{g(t)}(x_0,1)$, which is not allowed. In this latter case, we could then slide $y$ a distance $R\sqrt{t}$ along any minimising geodesic from $y$ to $x_0$, to a new point $y_0\in B_{g(t)}(x_0,1-R\sqrt{t})$.
(Note that this minimising geodesic is longer than $R\sqrt{t}$ because we have chosen $\thatwasT \leq \frac{1}{25R^2}<\frac{1}{4R^2}$, so that $1-R\sqrt{t}>R\sqrt{t}$.) 
Then
$B_{g(t)}(y_0,R\sqrt{t})$ would be a ball lying within $B_{g(t)}(x_0,1)$ that is still disjoint from all the balls $B_{g(t)}(p_i,R\sqrt{t})$, thus contradicting the definition of $N$ also in this case.

The fact that the balls $B_{g(t)}(p_i,3R\sqrt{t})$ cover $B_{g(t)}(x_0,1)$ is a fact that has an analogue at $t=0$. Indeed, we claim that the balls $B_{g(0)}(p_i,5R\sqrt{t})$ will cover
$B_{g(0)}(x_0,1)$. To see this, first note that 
the larger balls $B_{g(t)}(p_i,4R\sqrt{t})$ cover $B_{g(t)}(x_0,1+R\sqrt{t})$, 
and thus by \eqref{EBLapp}, they also cover $B_{g(0)}(x_0,1)$.
By \eqref{nestedRsqrttballs}, the balls $B_{g(0)}(p_i,5R\sqrt{t})$ then also cover
$B_{g(0)}(x_0,1)$ as required.

This covering gives us a lower bound on $N$ as follows. By our assumption $\thatwasT \leq \frac{1}{25R^2}$, the radius $5R\sqrt{t}$ is no more than $1$. Therefore, by Bishop-Gromov, each of these balls has volume no more than 
$\Om_n (5R\sqrt{t})^n $, where $\Om_n $ has already been taken to be the volume of a ball of the maximal radius $1$ in the model space with $\Ric\equiv-1$. Therefore we have 
\beq
\label{Nlowerbd}
v_0\leq \VolB_{g(0)}(x_0,1)\leq \sum_{i=1}^N \VolB_{g(0)}(p_i,5R\sqrt{t})
\leq N\Om_n (5R\sqrt{t})^n .
\eeq
Combining with the upper bound \eqref{Nupperbd} for $N$, we find that
$$\ep_0\geq \frac{\tilde\ep_0 v_0}{5^n \Om_n }$$
which contradicts the value of $\ep_0$ chosen in \eqref{ep0choice} at the beginning.
\end{proof}

It remains to prove the key supporting lemma \ref{root_t_vol_lemma}. First we give a hint of the intuition behind the proof in the case $A=1$. If the lemma were false, then even for extremely small $\ti\ep_0>0$, we would be able to find Ricci flows for which the volume of at least one $\sqrt{t}$-ball was very small after a very short time. Parabolically rescaling, we would have a Ricci flow that would have virtually nonnegative Ricci curvature, and in which there would be a unit ball with very small volume after time $1$. The $c_0/t$ decay of curvature would survive the rescaling, and this can be turned into control on the changes of distances thanks to Lemma \ref{DCL3}.

The main task is to play off the consequence of the lower volume bound at time $t=0$ against the smallness of volume at time $1$. The smallness of volume at $t=1$ will allow us to make a covering of a unit ball at time $1$ by a relatively small number of geodesic balls. Meanwhile, our control on the change of distances allows us to transfer this back to a covering of a unit ball at time $t=0$ by a relatively small number of balls, each of which has a controlled amount of volume as a result of the almost nonnegative Ricci curvature. This can be made to  contradict the $t=0$ lower volume control.

A key subtlety that this brief discussion overlooks is how we go from smallness of the volume of a ball at $t=1$ to having a covering by a small number of smaller balls. To make this work, we need more control on the \emph{geometry} of the set we are covering, and this we can only obtain by considering a different ball at $t=1$. 
The key idea is that the property of \emph{not} having a ball with a covering by a small number of balls is something that we can apply at spatial infinity where the manifold splits. This property then passes to an equivalent space of one lower dimension. By iteration, we can reduce the discussion to one dimension, where a contradiction is apparent. The key lemma that articulates this is the following, which we prove at the end of the section.

\begin{lemma}[Dimension reduction]
\label{geometryatinfinity} 
Let $R\geq 1$ and $  \ep,  \de,   v >0$ be arbitrary, and let $(M^n,g)$ be a smooth $n$-dimensional complete, non-compact Riemannian manifold without boundary, with $n\geq 2$, such that 
\begin{compactenum}[(i)]
\item
\label{DRHi}
$\Ric_g \geq 0$, and the norm of the curvature and all covariant derivatives (of any order) is bounded, i.e. $|\grad^k \Rm|_g\leq C(k)$ for each $k\in 0,1,\ldots$;
\item
\label{DRHii}
$\VolB_{g}(x,1) \geq \de$  for all $x \in M$;
\item
\label{DRHiii}
If we have a ball $B_{g}(y,L)$, for $y\in M$ and $L>0$, that is covered by $N$  
balls of radius $r\geq R$ (with respect to $g$) then we must have $N\geq   v L^n r^{-n}$;
\item
\label{DRHiv}
There exists $y_\infty \in M$ such that $\VolB_{g}(y_{\infty},R) \leq    \ep R^n $.
\end{compactenum}
Then there exists a sequence of points $z_k\in M$
with $d_{g}(y_\infty, z_k)\to\infty$ as $k\to\infty$, such that
$(M,g,z_k)$ converges in the smooth Cheeger-Gromov sense to the product manifold $\R\times (M',g')$ marked with some point $(0,y')$, $y'\in M'$, where $(M',g')$ is a smooth $(n-1)$-dimensional complete, non-compact Riemannian manifold without boundary such that
\begin{compactenum}
\item
\label{DRC1}
$\Ric_{g'} \geq 0$, and the norm of the curvature and all covariant derivatives (of any order) is bounded, i.e. $|\grad^k \Rm|_{g'}\leq C'(k)$ for each $k\in 0,1,\ldots$;
\item
\label{DRC2}
$\VolB_{g'}(x,1) \geq \de/2$  for all $x \in M'$;
\item
\label{DRC3}
If we have a ball $B_{g'}(y,L)$, for $y\in M'$ and $L>0$, that is covered by $N$  
balls of radius $r\geq R$ (with respect to $g'$) then we must have 
$N\geq \eta_1  v L^{n-1} r^{-(n-1)}$, for some $\eta_1=\eta_1(n)>0$;
\item
\label{DRC4}
There exists $y_\infty' \in M'$ such that 
$\VolB_{g'}(y_{\infty}',R) \leq  C(n)\frac{\ep}{  v} R^{n-1} $.
\end{compactenum}
\end{lemma}


\cmt{slightly bizarre how the conclusion \eqref{DRC4} degenerates as $v$ gets small.}

\begin{proof}[Proof of Lemma \ref{root_t_vol_lemma}]
As before, let $\Om_n $ be the volume of the unit ball in the $n$-dimensional model space with $\Ric\equiv-1$.
By Bishop-Gromov, the hypotheses are vacuous unless 
\beq
\label{v0upperbd}
v_0\leq \Om_n ,
\eeq
and so we assume that throughout the proof. 

Assume that Lemma \ref{root_t_vol_lemma} fails for some $v_0>0$ and dimension $n$ that we now fix for the remainder of the proof. In this case, if we take any sequence $\tilde\ep_{\ell}\downto 0$, then for each $\ell$, there exist $c_{\ell}<\infty$ and $A_\ell\geq 1$ so that the conclusion \eqref{vol_lower_conc} of the lemma (with $\ti\ep_0=\ti\ep_{\ell}$) will fail in an arbitrarily short time, for a Ricci flow satisfying the hypotheses of the lemma with $c_0=c_{\ell}$ and
$A=A_\ell$. 
It will be convenient, and no loss of generality, to assume that $c_{\ell}\geq 1$.
By omitting a finite number of terms in the sequence, 
we may assume that $\tilde\ep_{\ell}$ is less than the volume $\om_n $ of the unit ball in Euclidean $n$-space, and thus for any Ricci flow the conclusion \eqref{vol_lower_conc} will hold for sufficiently small $t>0$; our assumption is therefore that the length of the time interval is not controlled uniformly from below.

By assumption then, for each $\ell$ there exist a sequence of $n$-manifolds $M_j$, a sequence of points $x_j\in M_j$, a sequence of Ricci flows $\ti g_j(t)$ on $M_j$ for $t\in [0,\ti t_j]$, with $\ti t_j\downto 0$ (in particular, we may assume $\ti t_j\in (0,A_\ell^{-2}]$)  and a sequence of points $\ti y_j\in B_{\ti g_j(\ti t_j)}(x_j,1)\subset M_j$ such that


\begin{compactenum}[\ (i)]
\item
\label{propi1}
$B_{\ti g_j(t)}(x_j,10)$ 
is compactly contained in $M_j$ for all $t \in [0,\ti t_j]$
\item 
\label{lower_ric_orig1}
$\Ric_{\ti g_j(t)} \geq -1$ on $B_{\ti g_j(t)}(x_j,10)$ for all $t \in [0,\ti t_j]$
\item 
\label{propiii1}
$|\Rm|_{\ti g_j(t)} \leq \frac{c_{\ell}}{t} $    on
$ B_{\ti g_j(t)}(x_j,10)$ for all $t \in (0,\ti t_j]$
\item
\label{v0prop1}
$\VolB_{\ti g_j(0)}(x_j,1) \geq v_0$
\item
\label{propertyv}
$\VolB_{\ti g_j(\ti t_j)}(\ti y_j, A_\ell \sqrt{\ti t_j})  < \ti \ep_{\ell}(A_\ell \sqrt {\ti t_j})^n $.
\end{compactenum}

We need to improve our $\ti t_j$ and $\ti y_j$ so that in some sense we are considering
the first time when the volume is too small, and the volume is not much worse nearby. More precisely, with $\be$ as in Lemma \ref{shrinking_balls_lemma}, 
for each $\ell$, $j$,
choose $L$ so that $B_{\ti g_j(\ti t_j)}(x_j,2-(L+1)\be\sqrt{c_{\ell} \ti t_j})
= B_{\ti g_j(\ti t_j)}(x_j,1)$, i.e. so that 
$(L+1)\be\sqrt{c_{\ell} \ti t_j}=1$, which implies 
\beq
\label{Lest}
L\geq \half\be^{-1}(c_{\ell} \ti t_j)^{-1/2},
\eeq
after possibly deleting finitely many terms in $j$ so that $c_{\ell}\ti t_j$ is small enough.
Now choose $t_j\in (0,\ti t_j]$ to be the first time at which 
there exists a point $y_j\in \overline{B_{\ti g_j(t_j)}(x_j,2-(L+1)\be\sqrt{c_{\ell} t_j})}$
so that 
\beq
\label{vol_equality}
\VolB_{\ti g_j(t_j)}(y_j, A_\ell \sqrt{t_j}) = \ti \ep_{\ell}(A_\ell \sqrt {t_j})^n .
\eeq
We have picked $y_j \in B_{\ti g_j(t_j)}(x_j,2)$, and so
by deleting finitely many terms in $j$ so that $t_j$ is small enough, we have 
\beq
\label{where_is_yj}
y_j\in B_{\ti g_j(t_j)}(x_j,2)\subset B_{\ti g_j(t)}(x_j,3)\text{ for each }t\in [0,t_j],
\eeq
by the shrinking balls corollary \ref{shrinking_balls_cor}.
Therefore 
\beq
\label{yjxj}
B_{\ti g_j(t)}(y_j,7)\subset B_{\ti g_j(t)}(x_j,10)
\text{ for each }t\in [0,t_j],
\eeq
where our curvature bounds \eqref{lower_ric_orig1} and \eqref{propiii1} hold.

Since $t_j$ is the first time for which \eqref{vol_equality} holds, we have
\beq
\label{nice_start_lower_vol_bd}
\VolB_{\ti g_j(t)}(y, A_\ell \sqrt t)  \geq \ti \ep_{\ell}(A_\ell \sqrt t)^n  \text{ for all }t\in [0,t_j]\text{ and all }y\in B_{\ti g_j(t)}(x_j,2-(L+1)\be\sqrt{c_{\ell} t}).
\eeq
By Bishop-Gromov, we see that 
\beq
\label{nice_lower_vol_bd}
\VolB_{\ti g_j(t)}(y, \sqrt t)  \geq \frac{3\ti \ep_{\ell}}{4}(\sqrt t)^n  \text{ for all }t\in [0,t_j] \text{ and all }y\in B_{\ti g_j(t)}(x_j,2-(L+1)\be\sqrt{c_{\ell} t})
\eeq
after possibly deleting finitely many terms in $j$ so that $A_\ell {\sqrt { t_j}}$ is small enough.
Now we need to find a controlled space-time region centred at $y_j$ in which a similar
lower volume bound as in \eqref{nice_lower_vol_bd} holds.
By the prior inclusion lemma \ref{sharks_fin_lemma}, for each $\al\in (0,1)$,
the space-time `cylinder' defined for $t\in [0,\al^2 t_j]$ and $y\in B_{\ti g_j(t)}(y_j,L(1-\al)\be\sqrt{c_{\ell} t_j})$ fits within this region where \eqref{nice_lower_vol_bd} holds.
Thus using \eqref{Lest} and choosing $\al$ so that $1-\al=\eta_1c_\ell^{-1}\in (0,\half)$ for $\eta_1\in (0,\half)$ to be chosen later to be sufficiently small depending only on $n$ (with hindsight, $\eta_1=[100n(n-1)]^{-1}$ would be fine) we have 
\beq
\label{late_addition}
\VolB_{\ti g_j(t)}(y, \sqrt t)  \geq \frac{3\ti \ep_{\ell}}{4}(\sqrt t)^n  \text{ for all }t\in [0,\al^2 t_j] \text{ and all }y\in B_{\ti g_j(t)}\left(y_j,
\frac{\eta_1}{2}c_\ell^{-1}\sqrt{t_j /\ti t_j}\right).
\eeq
We need to extend the space-time region where we get such a volume ratio bound from existing for $t\in [0,\al^2 t_j]$ to existing for $t\in [0,t_j]$ (even if the region becomes thinner). 
More precisely, we claim that
\beq
\label{even_better_vol_est}
\VolB_{\ti g_j(t)}(y, \sqrt t)\geq \frac{\ti \ep_{\ell}}{2^{n+1} }(\sqrt{t})^n 
\text{ for }t\in [0,t_j]\text{ and }y\in B_{\ti g_j(t)}\left(y_j,
\frac{\eta_1}{4}c_{\ell}^{-1}\sqrt{t_j/\ti t_j}\right),
\eeq
and emphasise that we may assume that $t\in [\al^2 t_j,t_j]$.
We can simplify the region where $y$ lies in \eqref{even_better_vol_est} by making it independent of $t$ with the claim that 
\beq
\label{another_nest}
B_{\ti g_j(\al^2 t_j)}(y_j,\frac{\eta_1}{2}c_{\ell}^{-1}\sqrt{t_j/\ti t_j})
\supset
B_{\ti g_j(t)}(y_j,\frac{\eta_1}{4}c_{\ell}^{-1}\sqrt{t_j/\ti t_j}),
\eeq
for $t\in [\al^2 t_j,t_j]$, after deleting finitely many terms in $j$ so that $\ti t_j$ is small enough.
To see this, we can apply the second part of the shrinking balls corollary \ref{shrinking_balls_cor}, which tells us that
$$B_{\ti g_j(\al^2 t_j)}(y_j,r-\be\sqrt{c_\ell \al^2t_j})
\supset
B_{\ti g_j(t)}(y_j,r-\be\sqrt{c_\ell t}),$$
for $r\leq 7$. Setting $r=\be \al\sqrt{c_\ell t_j}+\frac{\eta_1}{2}c_\ell^{-1}\sqrt{t_j/\ti t_j}$ (which is less than $7$ after deleting finitely many terms in $j$ so that $t_j$ is sufficiently small) this implies
$$B_{\ti g_j(\al^2 t_j)}(y_j,\frac{\eta_1}{2}c_\ell^{-1}\sqrt{t_j/\ti t_j})
\supset
B_{\ti g_j(t)}(y_j,\be \al\sqrt{c_\ell t_j}-\be\sqrt{c_\ell t}+\frac{\eta_1}{2}c_\ell^{-1}\sqrt{t_j/\ti t_j}),$$
but we can estimate 
$$\be \al\sqrt{c_\ell t_j}-\be\sqrt{c_\ell t}\geq
-\be (1-\al)\sqrt{c_\ell t_j}
=-\be \eta_1\sqrt{t_j/c_\ell}
\geq -\frac{\eta_1}{4}c_\ell^{-1}\sqrt{t_j/\ti t_j})
$$
after deleting enough terms in $j$ so $\ti t_j$ is small enough. Thus \eqref{another_nest} is proved, and \eqref{even_better_vol_est} reduces to the claim that
\beq
\label{kanga}
\VolB_{\ti g_j(t)}(y, \sqrt t)\geq \frac{\ti \ep_{\ell}}{2^{n+1} }(\sqrt{t})^n 
\text{ for }t\in [\al^2 t_j,t_j]\text{ and }y\in 
B_{\ti g_j(\al^2 t_j)}(y_j,\frac{\eta_1}{2}c_{\ell}^{-1}\sqrt{t_j/\ti t_j}).
\eeq
This is a claim about the volume of a ball that is dependent on $t$, and we can avoid this dependency with the claim that
\beq
\label{nice_nesting}
B_{\ti g_j(\al^2 t_j)}(y, \al\sqrt{t_j})\subset B_{\ti g_j(t)}(y, \sqrt{t})
\eeq
after deleting finitely many terms in $j$ (so $t_j$ is sufficiently small,
depending on $\al$).
To see \eqref{nice_nesting} holds, 
we apply part 1 of Lemma \ref{DCL3} to deduce that
$$B_{\ti g_j(\al^2 t_j)}(y, \al\sqrt{t_j})\subset B_{\ti g_j(t)}(y, 
\al\sqrt{t_j}e^{t-\al^2t_j}),$$
(after deleting finitely many terms in $j$)
and therefore to establish \eqref{nice_nesting} we must show that 
$\al(\sqrt{t_j/t})e^{t-\al^2t_j}\leq 1$, or equivalently
$\half\log (\al^2t_j/t)+t-\al^2 t_j\leq 0$. But $\log x\leq x-1$ for any $x\geq 0$, 
so 
$$\half\log (\al^2t_j/t)+t-\al^2 t_j\leq \half(\al^2t_j/t-1)+t-\al^2 t_j
= (t-\al^2 t_j)(1-\frac{1}{2t})\leq 0$$
provided $t_j\leq \half$ (so $1-\frac{1}{2t}\leq 0$) and we confirm that \eqref{nice_nesting}
holds.
Thus \eqref{kanga}, and hence \eqref{even_better_vol_est}, reduce to the claim that
\beqa
\label{kanga2}
\lefteqn{\Vol_{\ti g_j(t)}B_{\ti g_j(\al^2 t_j)}(y, \al\sqrt{t_j})
\geq \frac{\ti \ep_{\ell}}{2^{n+1} }(\sqrt{t})^n} \qquad& \\
&\text{ for }t\in [\al^2 t_j,t_j]\text{ and }y\in 
B_{\ti g_j(\al^2 t_j)}(y_j,\frac{\eta_1}{2}c_{\ell}^{-1}\sqrt{t_j/\ti t_j}).
\eeqa
This we already know for $t=\al^2t_j$ by \eqref{late_addition}.
The curvature upper bound $|\Rm|_{\ti g_j(t)}\leq c_{\ell}/t$ will prevent the volume 
considered in \eqref{kanga2} from dropping too much over short times. 
More precisely, for $y\in B_{\ti g_j(\al^2 t_j)}(y_j,
\frac{\eta_1}{2}c_{\ell}^{-1}\sqrt{t_j/\ti t_j}
)$ and $t\in [\al^2 t_j,t_j]$,
we can estimate
\beqa
\label{ddtvol}
\frac{d}{dt}\Vol_{\ti g_j(t)}B_{\ti g_j(\al^2 t_j)}(y, \al\sqrt{ t_j}) &= -\int_{B_{\ti g_j(\al^2 t_j)}(y, \al\sqrt{ t_j})}\Sc_{\ti g_j(t)} dV_{\ti g_j(t)}\\
&\geq -C\frac{c_{\ell}}{t}\Vol_{\ti g_j(t)}B_{\ti g_j(\al^2 t_j)}(y, \al\sqrt{ t_j}),
\eeqa
for  $C=C(n)$.
Note that for this to be valid, we need that the domain of integration $B_{\ti g_j(\al^2 t_j)}(y, \al\sqrt{ t_j})$ lies within the ball $B_{\ti g_j(t)}(y_j,7)$ where we have curvature estimates (by \eqref{yjxj}); this inclusion follows by part 1 of Lemma \ref{DCL3}, for example.
Integrating \eqref{ddtvol} from $\al^2 t_j$ to $t\in (\al^2 t_j,t_j]$, keeping in mind that
$\al\in (1/2,1)$, we find that
$$\left[
\log\Vol_{\ti g_j(t)}B_{\ti g_j(\al^2 t_j)}(y, \al \sqrt{t_j})
\right]_{\al^2 t_j}^{t}
\geq
-Cc_{\ell}\log\al^{-2}
\geq
-Cc_{\ell}(1-\al)\geq -C\eta_1\geq \log \frac34,$$
where we allow the  constant $C$ to change at each instance, and finally choose $\eta_1$ small enough to achieve the last inequality.
By \eqref{nice_lower_vol_bd} in the case $t=\al^2 t_j$ this implies
\beq
\label{prelim_vol_est}
\Vol_{\ti g_j(t)}B_{\ti g_j(\al^2 t_j)}(y, \al\sqrt{t_j})\geq \frac{3}{4}\VolB_{\ti g_j(\al^2 t_j)}(y, \al\sqrt{t_j})
\geq (\frac{3}{4})^2 \ti \ep_{\ell}(\al\sqrt{t_j})^n 
\geq \frac{\ti \ep_{\ell}}{2^{n+1} }(\sqrt{t})^n ,
\eeq
which is claim \eqref{kanga2}.
Thus we have established \eqref{even_better_vol_est}, 
which can be considered a development of property \eqref{propertyv}.

\cmt{After blowing up, the radius of the ball where $y$ lives will become like $c_{\ell}^{-1}\ti t_j^{-1/2}$.}

\cmt{below we're using radius $9$ - almost $10$}

Finally we turn to Property \eqref{v0prop1}, and turn it into a more useful statement about the volume of smaller balls, with more general centres. Indeed, we see that
for all $y\in B_{\ti g_j(0)}(x_j,4)$, we have
$$\VolB_{\ti g_j(0)}(y,5)\geq \VolB_{\ti g_j(0)}(x_j,1)\geq v_0,$$
and so by Bishop-Gromov we have
\beq
\label{eta0ineq}
\VolB_{\ti g_j(0)}(y,r)\geq \eta_0 v_0 r^n ,
\eeq
for all $r\in (0,1)$, say, and  $\eta_0=\eta_0(n)>0$.
By the $t=0$ case of \eqref{where_is_yj}, the estimate
\eqref{eta0ineq} holds in particular for all $y\in B_{\ti g_j(0)}(y_j,1)$.

To summarise the above discussions, we can change our viewpoint from $x_j$ to $y_j$
and have the properties
\begin{compactenum}[\ (A)]
\item
\label{propi2}
$B_{\ti g_j(t)}(y_j,7)$ 
is compactly contained in $M_j$ for all $t \in [0,t_j]$
\item 
\label{lower_ric_orig2}
$\Ric_{\ti g_j(t)} \geq -1$ on $B_{\ti g_j(t)}(y_j,7)$ for all $t \in [0,t_j]$
\item 
\label{propiii2}
$|\Rm|_{\ti g_j(t)} \leq \frac{c_{\ell}}{t} $    on
$ B_{\ti g_j(t)}(y_j,7)$ for all $t \in (0,t_j]$
\item
\label{v0prop2}
$\VolB_{\ti g_j(0)}(y,r)\geq \eta_0 v_0 r^n $ for all $r\in (0,1)$ and
$y\in B_{\ti g_j(0)}(y_j,1)$, for  $\eta_0=\eta_0(n)>0$
\item
$\VolB_{\ti g_j(t_j)}(y_j, A_\ell \sqrt{t_j})  = \ti \ep_{\ell}(A_\ell \sqrt {t_j})^n $
\item
\label{propvi2}
$\VolB_{\ti g_j(t)}(y, \sqrt t)  \geq \frac{\ti\ep_{\ell}}{2^{n+1} }(\sqrt t)^n $ for all $t\in [0,t_j]$ and all $y\in B_{\ti g_j(t)}(y_j,\frac{\eta_1}{4}c_{\ell}^{-1}\sqrt{t_j/\ti t_j})$.
\end{compactenum}
By rescaling the Ricci flows parabolically, expanding distances by 
a factor $t_j^{-\half}$ and time by a factor $t_j^{-1}$, we get a new sequence of Ricci flows, which we call $g_j(t)$, defined for $t\in [0,1]$, such that 
$y_j\in B_{g_j(1)}(x_j,2t_j^{-\half})\subset M_j$ and 
\begin{compactenum}[\ (a)]
\item
$B_{g_j(t)}(y_j,7t_j^{-\half})$ 
is compactly contained in $M_j$ for all $t \in [0,1]$
\item 
\label{ric_prop_gj}
$\Ric_{g_j(t)} \geq -t_j$ on $B_{g_j(t)}(y_j,7t_j^{-\half})$ for all $t \in [0,1]$
\item $|\Rm|_{g_j(t)} \leq \frac{c_{\ell}}{t} $    on
$ B_{g_j(t)}(y_j,7t_j^{-\half})$ for all $t \in (0,1]$
\item
\label{eta0prop}
$\VolB_{g_j(0)}(y,r) \geq \eta_0 v_0 r^n $
for all $r\in (0,t_j^{-\half})$ and 
$y\in B_{g_j(0)}(y_j,t_j^{-\half})$
\item
$\VolB_{g_j(1)}(y_j, A_\ell )  = \ti \ep_{\ell}A_\ell^n $
\item
$\VolB_{g_j(t)}(y, \sqrt t)  \geq \frac{\ti \ep_{\ell}}{2^{n+1} }(\sqrt t)^n $ for all $t\in [0,1]$
and all $y\in B_{g_j(t)}(y_j,\frac{\eta_1}{4}c_{\ell}^{-1}{\ti t_j}^{\,-\half})$.
\end{compactenum}

After passing to a subsequence in $j$, we can extract a smooth complete pointed limit Ricci flow
$(M_j,g_j(t),y_j)\to (M,g(t),y_\infty)$ for $t\in (0,1]$ by Hamilton's compactness theorem.


Moreover, $M$ cannot be compact. If it were, then all $M_j$ would be compact 
for sufficiently large $j$. By the lower volume bound at $t=0$ (Property \eqref{v0prop1} above), and the Ricci lower bound (Property \eqref{lower_ric_orig1}), the diameter of $(M_j,\ti g_j(0))$
is uniformly bounded below (independent of $j$). 
Corollary \ref{shrinking_balls_cor} then gives us a lower bound for the diameter of
$(M_j,\ti g_j(t))$ for $t\in [0,t_j]$, once $t_j$ is sufficiently small.
Therefore, the diameter of $(M_j,g_j(t))$ for $t\in [0,1]$
must blow up, and the limit cannot be compact.

This limit $g(t)$ clearly has the properties
\begin{compactenum}[\ (1)]
\item $\Ric_{g(t)} \geq 0$ for all $t \in (0,1]$
\item $|\Rm|_{g(t)} \leq \frac{c_{\ell}}{t} $  for all $t \in (0,1]$
\item
$\VolB_{g(1)}(y_\infty, A_\ell )  = \ti \ep_{\ell}A_\ell ^n $
\item
\label{prop4}
$\VolB_{g(t)}(y, \sqrt t)  \geq \frac{\ti\ep_{\ell}}{2^{n+1} }(\sqrt t)^n $ for all $t\in (0,1]$
and all $y\in M$. 
\end{compactenum}


We can also deduce some further useful properties. First, by Bishop-Gromov, we see that
\beq
\label{g1volgrowth}
\VolB_{g(1)}(y_\infty, r)  \leq  \ti \ep_{\ell} r^n , \text{ for all }r\geq A_\ell .
\eeq

Moreover, although we fail to obtain a smooth limit of the flows $g_j(t)$ at $t=0$ because of the lack of curvature control, we can apply Bishop-Gromov to each $g_j(0)$ and obtain the limiting statement
\beq
\label{jlim1}
\limsup_{j\to\infty}
\VolB_{g_j(0)}(z_j, r)  \leq \om_n  r^n 
\eeq
for all $r>0$, provided $B_{g_j(0)}(z_j, r)$
lies within the region where the $t=0$ instance of the lower Ricci bound of Property \eqref{ric_prop_gj} holds, for sufficiently large $j$ -- for example, if $d_{g_j(0)}(z_j,y_j)$ is uniformly bounded -- where $\om_n $ is the volume of the unit ball in Euclidean space as before.

The principal way in which we extract information from time $t=0$, and Properties \eqref{v0prop1} and \eqref{eta0prop} in particular, is by using them to argue that we cannot cover large balls by too few small balls, even at later times. This principle will be used in the following  claim; to state it, we require 
$R:=\max\{\be\sqrt{c_{\ell}},A_\ell ,1\}$, where $\be$ is the constant given by Lemma 
\ref{shrinking_balls_lemma}, and also Lemma \ref{DCL3}.
Parts 3 and 1 of Lemma \ref{DCL3} tell us that 
%
%
%
for any 
$x\in B_{g_j(0)}(y_j,t_j^{-\half} )$ and $r\in (0,t_j^{-\half} )$ (for example) and sufficiently large $j$, 
we have 
\beq
\label{shrinking_balls_conseq}
B_{g_j(0)}(x,r+R)\supset B_{g_j(1)}(x,r)
\eeq
and 
\beq
\label{expanding_balls_conseq}
B_{g_j(0)}(x,r)\subset B_{g_j(1)}(x,er).
\eeq

\cmt{I insisted $R\geq 1$ as it's required in Lemma \ref{geometryatinfinity} and eventually in the first line of \eqref{Mbound}. We could even ask $r\in (0,2t_j^{-\half} )$ above.}


{\bf Claim 1:}
If we have a ball $B_{g(1)}(y,L)$, for $y\in M$ and $L>0$, that is covered by $N$ 
balls of radius $r\geq R$ (with respect to $g(1)$) then we must have $N\geq \eta v_0 L^n  r^{-n} $ for some  $\eta=\eta(n)>0$.

\begin{proof}[Proof of Claim 1.]
By the smooth convergence of $g_j(t)$ to $g(t)$, we can be sure that for sufficiently large $j$, there exist centres $p^j_1,\ldots,p^j_{N}\in M_j$ so that the larger balls 
$B_{g_j(1)}(p^j_i,(r+R))$ cover some ball $B_{g_j(1)}(\hat y_j,L)$ within $M_j$, where
$\hat y_j$ remains a $j$-independent distance from $y_j$.
By \eqref{expanding_balls_conseq}, 
this latter ball contains the smaller, earlier
$B_{g_j(0)}(\hat y_j,L/e)$, and 
by \eqref{shrinking_balls_conseq}
the balls $B_{g_j(1)}(p^j_i,(r+R))$ are contained within the larger, earlier
$B_{g_j(0)}(p^j_i,(r+2R))$.
Therefore
\beq
\sum_{i=1}^{N} \VolB_{g_j(0)}(p^j_i,(r+2R))
\geq 
\VolB_{g_j(0)}(\hat y_j,L/e)
\geq \eta_0 v_0 \left(\frac{L}{e}\right)^n 
\eeq
by Property \eqref{eta0prop}, for sufficiently large $j$.
Taking the limit $j\to\infty$ and using \eqref{jlim1}, 
we obtain
$$N\om_n (r+2R)^n \geq 
\eta_0 v_0 \left(\frac{L}{e}\right)^n ,
$$
and so for  $\eta=\eta(n)>0$, we conclude
$$N\geq \eta v_0 L^n  r^{-n} .$$
\end{proof}

This claim is the final ingredient allowing us to apply Lemma \ref{geometryatinfinity}
to $(M,g(1))$. The control on the derivatives of the curvature follow from Shi's estimates. Condition \eqref{DRHii} is satisfied with $\de=\frac{\ti \ep_{\ell}}{2^{n+1} }$
by the $t=1$ instance of \eqref{prop4}. The $v$ of that lemma can be taken to be $\eta v_0$ here, thanks to the claim.
Condition \eqref{DRHiv} follows from \eqref{g1volgrowth}, with $\ep=\ti\ep_\ell$.

The output of Lemma \ref{geometryatinfinity} is a noncompact $(n-1)$-dimensional Riemannian manifold $(M',g')$ with the properties \eqref{DRC1} to \eqref{DRC4}, which allow us to immediately apply 
Lemma \ref{geometryatinfinity} again, with $n$ replaced by $n-1$, and this time with
$\de=\frac{\ti \ep_{\ell}}{2^{n+2}}$, $v=\eta_1 \eta v_0$, and 
$\ep=C\frac{\ti\ep_\ell}{v_0}$.
Again, the output is a manifold of dimension reduced by one, to which we can immediately reapply  the lemma. After $n-1$ iterations of this procedure, 
the output is a non-compact \emph{one}-dimensional complete, non-compact Riemannian manifold, i.e. $\R$, for which the volume of an $R$-ball is simply $2R$.
Conclusion \eqref{DRC4} of the final iteration tells us that 
$$2R\leq C\frac{\ti\ep_\ell}{v_0^{n-1}}R$$
for some  $C=C(n)<\infty$, which is a contradiction for sufficiently large $\ell$ so that 
$\ti\ep_\ell$ is sufficiently small.
\end{proof}

It remains to prove our dimension reduction lemma.

\cmt{Some of the arguments in the proof below are very similar to the proof of Lemma \ref{volume_control_lem}}

\begin{proof}[{Proof of Lemma \ref{geometryatinfinity}}]
First note that the set of geodesic rays emanating from $y_\infty$ is nonempty, because $M$ is
noncompact and complete. Indeed, choosing a sequence of unit tangent vectors at $y_\infty$ that exponentiate to minimising geodesics of diverging length, a subsequence converges to a unit tangent vector that exponentiates to a geodesic ray.

For $0\leq K< L<\infty$, consider the annulus
$A_{K,L} :=  B_{g}(y_\infty,L) \setminus B_{g}(y_\infty,K) $ and its nonempty subset
$$\ti A_{K,L} := \{ p \in A_{K,L} \ |\ p \text{ lies within some geodesic ray emanating from }y_\infty\}.$$
We call $\ti A_{K,L}$ the {\it modified annulus}.
Since 
$\Ric \geq 0$, we must have $ \Vol(A_{K,L} - \ti A_{K,L}) \leq \ep(K) L^n$ 
where $\ep:[0,\infty)\to [0,\infty)$ is a decreasing function, depending on $g$, with $\ep(K) \to 0$
as $K \to \infty$.
By Bishop-Gromov, we also have $\Vol(A_{K,L}) \leq \om_n (L^n -K^n)$, where $\om_n$ is the volume of the unit ball in Euclidean $n$-space.

\cmt{probably we already defined $\om_n$}

\cmt{look up visibility angle}

\cmt{KEEP: Small explanation for us: write everything in
geodesic cooordinates at $y_\infty$. 
Let $S_r \in S^{n-1}_1(0)$ be the points
such that $(s,\theta)$ is not a cut point of $0 \in \R^n$ for all $s \leq r$.
Let $\omega(r,\theta)d\theta = \sqrt{\det(g_{\theta_i \theta_j })(r,\theta)}d\theta$ be the volume
form of the metric on $S_1(0)$ at $(r,\theta)$. The Bishop-Gromov
comparison principle says $ \frac{\omega(r,\theta)}{r^{n-1}}  \leq \ti\omega_n $
(the volume element of the unit sphere) and
$ \frac{\omega(r,\theta)}{r^{n-1}}  $ is a decreasing function. We define
$\omega(t,\theta)= 0$ for all $t \geq r$ if $(r,\theta)$ is a cut
point.
We have $S_r \subset  S_s$ for all $r\geq s$. $W := \cap_{r \in \R^+}
S_r$ corresponds to the points that are infinite length minimising
geodesics: $\ga(s) := (s,\theta)$ is a length minimising infinite
length geodesic if and only if  $\theta \in W$. $\Vol(B_r) = \int_0^r
\int_{S_s} \omega(s,\theta) d\theta ds$.
From the above definitions of  $W$ and $S_l$, we know 
$| \Vol(S_l,d \theta) - \Vol( W,d \theta)| \leq \ep(l) \to 0$ as $l \to \infty$, 
with $\ep(l)\to 0$, monotonically,
where here $d \theta$ refers to
the standard volume form on $S^{n-1}_1(0)$.
Also $A_{K,L}- \ti A_{K,L} \subset  \{ (r,\theta) \ | \ r \in [K,L],
\theta \in S_K-W \}.$
This means
\begin{eqnarray}
\label{ep_def}
\Vol (A_{K,L}- \ti A_{K,L} ) &&\leq \Vol( \{ (r,\theta) \ | \ r \in [K,L],
\theta \in S_K -W \}) \cr
&& = \int_K^L \int_{S_K-W} \omega(r,\theta) d\theta dr \cr
&& \leq \ep(K) L^n.
\end{eqnarray}
}

\cmt{final sequence $z_k$ will be a subsequence of this:}

We first claim that there exists a sequence of points $z_k\in M$
with $d_{g}(y_\infty, z_k)\to\infty$ as $k\to\infty$, each of which lies within some 
(length minimising) geodesic ray $\ga_k:[0,\infty)\to (M,g)$ of
infinite length emanating from $y_\infty$, such that 
$$\VolB_{g}(z_k,R)<  (100R)^n  \frac{\ep}{v}.$$
Suppose the claim is false. Then there exists $\ti R>0$ such that for any $p\in M$ with 
$d_{g}(y_\infty, p)\geq \ti R$, and with $p$ lying on a geodesic ray, we have $\VolB_{g}(p,R)\geq\si_0$, where 
$\si_0= (100R)^n \frac{\ep}{v}$.

Clearly we are free to increase $\ti R$ if helpful, and so we may
assume that $\ti R\geq 10R$. We will also need $\ti R$ to be large compared with 
$R^{n+1}$, and it will be sufficient to ask that 
$\ti R\geq C(n) R^{n+1}/(\de v)$, for $C(n)$ that will be determined during the proof. 
Furthermore, we ask that $\ti R$ is large enough so that $\ep(\ti R)$ is small enough. 
It will be sufficient to ask that $\ep(\ti R)\leq \ep \de /(2\si_0)$.

Let $N\geq 1$ be the largest number of disjoint balls $B_{g}(p_i,R)$ that
we can pick that lie within $A_{ \ti R, 3 \ti R}$, but whose
centre points $p_i$ lie within the subset $\ti A_{\ti R, 3 \ti R}$.
Fix such a collection of centre points $p_1,\ldots,p_N$.
We see that $N$ has an upper bound because condition \eqref{DRHiv} and 
Bishop-Gromov tell us that 
\beq
\label{another_N_upper}
\ep (3 \ti R)^n\geq\VolB_{g}(y_\infty,3 \ti R)
\geq \sum_{i=1}^N \VolB_{g}(p_i,R)\geq N\si_0.
\eeq
Moreover, just as in the proof of Lemma \ref{volume_control_lem} (see the argument just after the Inequality \eqref{Nupperbd}), 
we can be sure that the balls $B_{g}(p_i,3R)$ cover the modified annulus
$\ti A_{ \ti R, 3 \ti R}$.
However for our argument, we need a covering of $ A_{ \ti R, 3 \ti  R}$. In order to achieve this we
consider further points $p_{N+1}, \ldots, p_{N+M}$ , such that 
$\{B_{g}(p_{N+i},R)\}_{i=1}^M$ is a maximal disjoint set of balls 
with $p_{N+1}, \ldots, p_{N+M} \in A_{ \ti R, 3 \ti R} \setminus \ti A_{ \ti R, 3 \ti R}$. Note here that we don't require that the balls also live in  $A_{ \ti R, 3 \ti R} \setminus \ti A_{ \ti R, 3 \ti R}.$ The set of balls
$\{B_{g}(p_{N+i},3R)\}_{i=1}^{M}$ is then a covering of $A_{ \ti R, 3 \ti R} \setminus \ti A_{ \ti R, 3 \ti R},$ and hence the set of balls $\{B_{g}(p_{i},3R)\}_{i=1}^{N+M}$ is  a covering of $A_{ \ti R, 3 \ti R}$. 

\cmt{annoying double use of $M$, but not worth changing at this stage - unlikely to cause confusion}

If $B_{g}(p_{N+i},R)$ intersects $\ti A_{ \ti R, 3 \ti R}$ for
some $i \in \{1,2, \ldots, M\}$, then throw the ball
$B_{g}(p_{N+i},R)$ away: It 
will be covered by
$B_{g}(p_j,5R),$
for some
$j \in \{1, 2, \ldots, N\}$, and the new set of balls
$B_{g}(p_i,5R)_{i=1}^{N+M}$ of radius $5R$ (we do not change the
name of the centre points of the balls, or $M$, even though some balls have been removed) still covers $A_{
  \ti R, 3 \ti R}.$ 
We can bound the number $M$ of remaining balls because the remaining balls $\{B_{g}(p_{N+j},R)\}_{j=1}^{M}$ with centre points in $A_{
  \ti R, 3 \ti R} \setminus \ti A_{ \ti R, 3 \ti R}$ have very little
volume (relatively), since  they are disjoint and completely contained in  
$A_{  \ti R-R, 3 \ti R+R} \setminus \ti A_{ \ti R, 3 \ti R}$. Indeed, 
because $R\geq 1$, condition \eqref{DRHii} gives
%
%
\beqa
\label{Mbound}
{\de } M & \leq \sum_{j=1}^M \VolB_{g}(p_{N+j}, R) \\
&  \leq \Vol_{g}( A_{ \ti R-R, 3\ti R+R} \setminus \ti A_{ \ti R, 3\ti
R}) \\
& = \Vol_{g}((A_{\ti R-R,\ti R}  \cup A_{\ti R,3 \ti R}\cup A_{3 \ti R, 3 \ti R + R}) \setminus \ti A_{ \ti R, 3\ti
R}) \\
& \leq \om_n( \ti R^n - (\ti R -R)^n) +   \om_n ( (3 \ti R + R)^n -( 3 \ti R)^n) + 
\ep(\ti R) \ti R^n\\
& \leq   C(n) \ti R^{n-1}R + \ep(\ti R) \ti R^n\\
& \leq \frac{\ep \de}{\si_0} \ti R^n 
\eeqa
if $\ti R$ is as large as we previously asked, where we have used that 
$(a+R)^n-a^n\leq C(n)a^{n-1}R$ for $a\geq R$, 
by the mean value theorem, to obtain the penultimate line.
Hence
\beqa
(N+ M)  & \leq \frac{ 2\ep  }{\si_0}(3 \ti R)^n \cr
& <  \frac{  v \ti R^n}{ (5R)^n}
 \label{uppity}
\eeqa
from the definition of $\si_0$.
Now take a ball $B_{g}(q,\ti R)$ of radius $\ti R$ and centre $q$ sitting in $A_{\ti R, 3 \ti R}$.
We can cover $B_{g}(q,\ti R)$ with the $N+M$ balls of radius $r= 5R$ that we have just constructed, since they cover
$A_{\ti R, 3 \ti R}$. But then by \eqref{DRHiii}, we must have 
$N+M \geq \frac{ v \ti R^n}{(5R)^n}.$ 
This contradicts  \eqref{uppity}, and completes the proof of the claim that we can find the points $z_k$.

We now use the points $z_k$ of this claim to extract a limit. Indeed, by condition \eqref{DRHi}, and condition \eqref{DRHii} applied to each of the points $z_k$, 
we know that we can apply the 
Cheeger-Gromov-Hamilton compactness theorem, and passing to a subsequence in $k$ we obtain smooth convergence
$(M,g, z_k)\to (M_\infty,h,z_\infty)$ to a limit which still has $\Ric\geq 0$, but now also contains a (length minimising) geodesic line through $z_\infty$, and hence splits isometrically as a product $\R\times (M',g')$, where $(M',g')$ is an $(n-1)$-dimensional Riemannian manifold
with nonnegative Ricci curvature. This is Conclusion \ref{DRC1}.
We may write $z_\infty=(0,y_\infty')$ for some point $y_\infty'\in M'$.

By passing the claim to the limit, we have
\beq
\label{claim2_conseq}
\VolB_{h}(z_\infty,R)\leq C(n)\frac{\ep}{v} R^n.
\eeq
By Bishop-Gromov, this volume ratio estimate extends to larger radii (e.g. to radius $2R$) and
subsequently  also to the factor $M'$,
giving
\beq
\label{claim2_conseq2}
\VolB_{g'}(y_\infty',R)\leq C(n)\frac{\ep}{v} R^{n-1},
\eeq
for some (possibly different) $C=C(n)$. 
This is Conclusion \ref{DRC4}.
We also have, for any $z \in M'$, that
$2\VolB_{g'}(z,1) =
\Vol_h( B_{g'}(z,1) \times [-1,1]) \geq \VolB_h((0,z),1) \geq \de$ and hence
$\VolB_{g'}(z,1) \geq \frac{\de}{2}$.
This is Conclusion \ref{DRC2}.

Next we pass condition \eqref{DRHiii} to the limit to find that
if we have a ball $B_{h}(y,L)$, for $y\in M_\infty$ and $L>0$, that is covered by $N$ 
balls of radius $r\geq R$ (with respect to $h$) then we must have $N\geq v L^n r^{-n}$.
Because of the splitting of $M_\infty$, this then implies Conclusion \ref{DRC3}.

\cmt{Note to us: given such a covering, then we can multiply the radius by $100$, and use of the order of $L/r$ copies of the covering, shifted by $r$ in the $\R$ direction, to cover the 3D ball. So the number of balls in the $n$-dim covering is of the order of $NL/r$...}

From Conclusion \ref{DRC3}, we immediately see that $M'$ cannot be compact, because in that case we would be able to cover the whole space with $N=1$ balls of a fixed radius, and then allow $L$ to diverge to infinity to obtain a contradiction.
This completes the proof of the dimension reduction lemma \ref{geometryatinfinity}.
\end{proof}

\emph{Acknowledgements:} This work was supported by EPSRC grant number EP/K00865X/1.
Key parts of this paper originated at the 2013 Oberwolfach PDE meeting. We would like to thank the MFO and the organisers Alice Chang, Camillo De Lellis and Reiner Schaetzle for their invitation.


{\sc MS: institut f\"ur analysis und numerik (IAN),
universit\"at magdeburg, universit\"atsplatz 2,
39106 magdeburg, germany}

{\sc PT: mathematics institute, university of warwick, coventry, CV4 7AL,
uk}\\
\url{http://www.warwick.ac.uk/~maseq}
\end{document}